\documentclass[reqno,11pt]{amsart}
\numberwithin{equation}{section}
\usepackage{amsmath}
\usepackage{esint} 
\usepackage{amsthm}
\usepackage{epsfig}
\usepackage{psfrag}
\usepackage{graphicx}
\graphicspath{{figure/}}
\usepackage{graphpap,latexsym,epsf}
\usepackage{color}
\usepackage{amssymb,mathrsfs,enumerate}
\usepackage{mathtools}
\usepackage{etoolbox}
\cslet{blx@noerroretextools}\empty
\usepackage[giveninits=true,style=ext-alphabetic,backend=biber,maxbibnames=99]{biblatex}
\definecolor{citegreen}{rgb}{0,0.6,0}
\definecolor{refred}{rgb}{0.8,0,0}
\usepackage[colorlinks, citecolor=citegreen, linkcolor=refred]{hyperref}
\usepackage{cleveref}
\crefname{definition}{Definition}{Definitions}
\crefname{theorem}{Theorem}{Theorems}
\crefname{thmx}{Theorem}{Theorems}
\crefname{lemma}{Lemma}{Lemmas}
\crefname{step}{Step}{Steps}
\crefname{substep}{Step}{Steps}
\crefname{claim}{Claim}{Claims}
\crefname{proposition}{Proposition}{Propositions}
\crefname{corollary}{Corollary}{Corollaries}
\crefname{remark}{Remark}{Remarks}
\crefname{section}{Section}{Sections}
\crefname{subsection}{Section}{Sections}
\crefname{chapter}{Chapter}{Chapters}
\crefname{appendix}{Appendix}{Appendices}
\crefname{equation}{}{}

\usepackage{autonum}
\newcommand{\R}{\mathbb{R}}

\newcommand{\N}{\mathbb{N}}

\newcommand{\Sph}{\mathbb{S}}

\def\HHH{{\rm H}}
\def\RRR{{\mathrm R}}

\def\a{\alpha}

\newcommand{\pa}{\partial}

\newcommand{\ffi}{\varphi}

\newcommand{\ep}{\varepsilon}

\newcommand{\Ric}{{\rm Ric}}

\newcommand{\D}{{\rm D}}

\newcommand{\De}{\Delta}

\newcommand{\na}{\nabla}
\newcommand{\nana}{\nabla^2}

\renewcommand{\epsilon}{\varepsilon}
\renewcommand{\phi}{\varphi}

\renewcommand{\div}{\mathrm{div}}

\newcommand{\gr}{\mathfrak{g}}
\newcommand{\AVR}{{\rm AVR}}
\newcommand{\Sf}{\mathbb{S}}

\mathchardef\emptyset="001F

\DeclarePairedDelimiter{\abs}{\lvert}{\rvert}

\definecolor{vgreen}{rgb}{0.1,0.5,0.2}
\definecolor{viola}{RGB}{85,26,139}

\newtheorem{theorem}{Theorem}[section]
\newtheorem{remark}[theorem]{Remark}
\newtheorem{corollary}[theorem]{Corollary}
\newtheorem{definition}[theorem]{Definition}
\newtheorem{proposition}[theorem]{Proposition}

\newtheorem{lemma}[theorem]{Lemma}

\newtheorem{thmx}{Theorem}

\begin{document}

\hyphenation{ma-ni-fold}

\title[Comparison geometry for substatic manifolds
and weighted isoperimetry]{Comparison geometry for substatic manifolds and a weighted Isoperimetric Inequality}

\author[S.~Borghini]{Stefano Borghini}
\address{S.~Borghini, Universit\`a degli Studi di Trento,
via Sommarive 14, 38123 Povo (TN), Italy}
\email{stefano.borghini@unitn.it}

\author[M.~Fogagnolo]{Mattia Fogagnolo}
\address{M.~Fogagnolo, Universit\`a degli Studi di Padova , via Trieste 63, 35121 Padova (PD), Italy}
\email{mattia.fogagnolo@unipd.it}

\begin{abstract} 
Substatic Riemannian manifolds with minimal boundary arise naturally in General Relativity as spatial slices of static spacetimes
satisfying the Null Energy Condition. Moreover, they constitute a vast generalization of nonnegative Ricci curvature.
In this paper we will prove various geometric results in this class, culminating in a sharp, weighted Isoperimetric inequality that quantifies the area minimizing property of the boundary. Its formulation and proof will build on a comparison theory 
partially stemming from a newly discovered conformal connection with $\mathrm{CD}(0, 1)$ metrics.  
\end{abstract}

\maketitle


\smallskip
\noindent 

\date{\today}

\maketitle

\section{Introduction}

In this paper we are interested in the study of triples $(M,g,f)$, where $(M,g)$ is a Riemannian  manifold of dimension $n\geq 3$ with (possibly empty) compact boundary $\pa M$ and $f:M\to\R$ is a smooth function that is positive in the interior of $M$ and zero on $\pa M$, satisfying the following inequality
\begin{equation}
\label{eq:substatic}
f\Ric - \nana f + (\De f)g \geq 0,
\end{equation}
where $\Ric$ is the Ricci tensor of the metric $g$, $\nana$ is the Hessian and $\De={\rm tr}\nana$ is the Laplace--Beltrami operator with respect to the Levi-Civita connection $\na$ of $g$. We will refer to such triples $(M,g,f)$ as {\em substatic triples} or simply {\em substatic manifolds}. 
We say that a substatic manifold has {\em horizon boundary} if $\pa M$ is either empty or it is a minimal hypersurface and $\abs{\nabla f} \neq 0$ on $\partial M$.

Condition~\eqref{eq:substatic} arises naturally in the study of static spacetimes satisfying the Null Energy Condition, as already observed in~\cite{Wang_Wang_Zhang}.
More precisely, a Lorentzian manifold $(L, \gr)$ 
of the form
\[
L=\R\times M\,,\qquad \gr\,=\,-f^2dt\otimes dt+g,
\]
happens to be a solution to the Einstein Field Equation 
\[
\Ric_\gr\,+\,\left(\Lambda-\frac{1}{2}\RRR_\gr\right)\gr\,=\,T\,,
\]
subject to $T(X, X) \geq 0$ for any vector field $X$ satisfying $\gr (X, X) = 0$, exactly when $f$ and $g$ satisfy~\eqref{eq:substatic}. A minimal boundary represents, in this framework, the event horizon of a black hole. For the reader's sake, we included the computations in \cref{app:phys_mot}. 
The class of substatic manifolds obviously includes the very large and thoroughly studied class of manifolds with nonnegative Ricci curvature, where $f$ is just constant, and consequently the minimal boundary is empty. However, even considering explicit model  warped products only, a whole new zoo of examples arises. 

As an example, we recall that the following family of triples $(M, g, f)$ is in fact a family of substatic triples:
\begin{equation}
\label{eq:models}
M=I\times\Sigma\,,\qquad g=\frac{dr\otimes dr}{f^2}+r^2 g_{\Sigma},\qquad f=\sqrt{1-\frac{2\Lambda}{n(n-1)} r^2-\frac{2m}{r^{n-2}} + \frac{q^2}{r^{2n-4}}},
\end{equation}
where $(\Sigma, g_\Sigma)$ is a closed $(n-1)$-dimensional Riemannian manifold satisfying $\Ric_{g_\Sigma} \geq (n-2) g_\Sigma$,
$\Lambda, q\in\R$, $m\geq 0$ and $I\subseteq[0,+\infty)$ is the maximal interval such that the quantity in square root in~\eqref{eq:models} is nonnegative for all $r\in I$. 
According to the sign of $\Lambda$, the case $m=q=0$ corresponds to the space forms. 
If instead $m>0$, $q = 0$,  one obtains the families of the Schwarzschild, Schwarzschild--de Sitter and Schwarzschild--Anti de Sitter black holes, again with respect to $\Lambda$ being vanishing, positive or negative. If $m > 0$ and $q \neq 0$, one gets the Reissner--Nordstr\"om versions of these last spaces. From a physical point of view, $\Lambda$ is the cosmological constant, $m$ is the mass and $q$ is the charge of the black hole. 

{
We will always tacitly assume that $(M,g)$ is complete as a metric space. This holds true for the models~\eqref{eq:models}, provided the absolute value of the charge $q$ is not too big. For instance, for $\Lambda=0$, the solution has a singularity at $r=0$ when $\abs{q}>m$.
}

The main achievement of the present work is the following sharp
Isoperimetric Inequality, taking place in a relevant subclass of substatic triples. It is saturated by warped product metrics only, such as the ones in~\eqref{eq:models}.

\begin{thmx}[Substatic $f$-isoperimetric inequality]
\label{thm:isoperimetric_intro}
Let $(M,g,f)$ be a substatic triple of dimension $n\leq 7$, with horizon boundary and one uniform $f$-complete end.
Assume there exists an exhaustion of nonmimimal outward minimizing hypersurfaces homologous to the boundary. Then, for any bounded domain $\Omega_\Sigma$ with smooth boundary $\pa\Omega_\Sigma=\pa M\sqcup\Sigma$ it holds  
\begin{equation}
\label{eq:isoperimetric_intro}
\abs{\Sigma}^{\frac{n}{n-1}} - \abs{\partial M}^{\frac{n}{n-1}} \geq n \left[\mathrm{AVR}(M, g, f) \abs{\Sf^{n-1}}\right]^{\frac{1}{n-1}} \abs{\Omega_\Sigma}_f.
\end{equation}
Moreover, in the case $\mathrm{AVR}(M, g, f) > 0$, 
\begin{itemize} 
\item if $\partial M \neq \emptyset$, the equality holds in~\eqref{eq:isoperimetric_intro} if and only if $\partial M$ is connected and $(M, g)$ is isometric to 
\begin{equation}
\label{eq:rigidiso_intro}
\left([\overline{s}, +\infty) \times \partial M, \, \frac{ds \otimes ds}{f(s)^2} + \frac{s^2}{\overline{s}^2}\, g_{\partial M}\right), 
\end{equation}
where $g_{\partial M}$ is the metric induced by $g$ on $\partial M$
and $\Sigma$ is a level set of $s$. In particular, $f=f(s)$ is a function of $s$ alone. 
\item If $\partial M = \emptyset$, the equality holds in~\eqref{eq:isoperimetric_intro} if and only if $(M, g)$ is isometric to 
\begin{equation}
\label{eq:rigidiso_intro-nobordo}
\left([0, +\infty) \times \Sf^{n-1}, \, \frac{ds \otimes ds}{f(s)^2} + \frac{s^2}{f(x)^2} \,g_{\Sf^{n-1}}\right), 
\end{equation}
where $g_{\Sf^{n-1}}$ is the round metric on the $(n-1)$-dimensional sphere $\Sf^{n-1}$ and $x \in \Omega_\Sigma$. In this case, $\Sigma$ is a level set of $s$ homothetic to the round sphere. The function $f = f(s)$ depends on $s$ alone also in this case.
\end{itemize}
\end{thmx}
The asymptotic assumptions entering in the above statement will be better understood in the next Subsection, in connection with the comparison results presented below.
Concerning the quantities appearing in~\eqref{eq:isoperimetric_intro}, we have denoted by $\abs{\Omega_\Sigma}_f$ the weighted volume $\int_{\Omega_\Sigma} f\,d\mu$, whereas $\mathrm{AVR}(M, g, f)$ is a suitable substatic generalization of the classical Asymptotic Volume Ratio for nonnegative Ricci curvature, see \cref{eq:avr-intro}. When it is nonzero, inequality~\cref{eq:isoperimetric_intro} in particular yields a quantitative information about the minimal boundary being in fact area minimizing, in terms of a suitable weighted volume. Observe that a priori the boundary is not even assumed to be area minimizing at all. 
From a more analytical point of view, formula~\eqref{eq:isoperimetric_intro} constitutes a nonstandard weighted isoperimetric inequality, as the perimeter is actually unweighted. The geometric intuition behind it will be given by the end of the following Subsection. 
One can interpret the very thoroughly recently studied  Isoperimetric Inequality in nonnegative Ricci curvature  \cite{agostiniani_sharpgeometricinequalitiesclosed_2020,  brendle_sobolevinequalitiesmanifoldsnonnegative_2021, antonelli-pasqualetto-pozzetta-semola, balogh-kristaly, johne, cavalletti-manini1, cavalletti-manini2, pozzetta2023isoperimetry} as a special case of \cref{eq:isoperimetric_intro}, obtained when the boundary is empty and $f$ is constant. The rigidity statement accordingly generalizes the one of the nonnegative Ricci curvature case.

We point out that Theorem~\ref{thm:isoperimetric_intro} is particularly meaningful and perfectly sharp already in the above recalled  Reissner--Nordstr\"om and Schwarzschild metrics, consisting in \cref{eq:models} for $\Lambda =0$, $m > 0$, $|q|<m$, and more generally in \emph{asymptotically flat} substatic manifolds. With asymptotically flat we mean that the manifold converges to the Euclidean space (in a very weak sense) and that $f$ goes to $1$ at infinity, see \cref{def:AF}. From the definition, it follows that an asymptotically flat end is automatically uniform and $f$-complete, it possesses a natural exhaustion of coordinate spheres and it is possible to compute ${\rm AVR}(M,g,f)=1$. Thus, the above statement simplifies significantly.

\begin{corollary}
Let $(M,g,f)$ be a substatic triple of dimension $n\leq 7$, with horizon boundary and one asymptotically flat end.
Then, for any bounded domain $\Omega_\Sigma$ with smooth boundary $\pa\Omega_\Sigma=\pa M\sqcup\Sigma$ it holds  
\begin{equation}
\label{eq:isoperimetric_intro_AF}
\abs{\Sigma}^{\frac{n}{n-1}} - \abs{\partial M}^{\frac{n}{n-1}} \geq n \abs{\Sf^{n-1}}^{\frac{1}{n-1}} \abs{\Omega_\Sigma}_f.
\end{equation}
The same rigidity statement as in Theorem~\ref{thm:isoperimetric_intro} applies in case of equality.
\end{corollary}

To our knowledge, even in the model cases, inequality~\cref{eq:isoperimetric_intro_AF} was never observed before, and does not seem to be inferable from the characterization of classical isoperimetric sets  resulting from the work of Brendle \cite{brendle-alexandrov}, or the earlier \cite{bray-thesis, bray-morgan} about the Schwarzschild case.

\subsection{Substatic comparison geometry}
Our analysis begins with the aim of working out a satisfactory substatic comparison theory, inspired by the classical nonnegative Ricci case. While, in such case, the model to be compared with is $\R^n$, or more generally a cone, in the substatic generalization the model should be constituted by the large family of substatic warped products in fact appearing in the rigidity statement of \cref{thm:isoperimetric_intro}. 

To pursue our goal,  an initial step consists in comparing the mean curvature of geodesic spheres with that of the models. Interestingly, in order to obtain a manageable Riccati equation ruling such comparison, we are led to work in the metric $\tilde g=g/f^2$. This is no accident: the metric $\tilde g$ happens to fulfil the ${\rm CD}(0,1)$ condition, consisting in a metric subject to 
\begin{equation}
\label{CD-intro}
\Ric_{\tilde g}+\widetilde\na^2 \psi+\frac{1}{n-1}d\psi\otimes d\psi \geq 0
\end{equation}
for some smooth function $\psi$. To our knowledge, such explicit conformal relation was not pointed out in literature yet. However, a remarkable link is described by Li--Xia \cite{Li_Xia_17}: they come up with a family of connections with Ricci curvatures interpolating between the tensor in the left-hand side of \cref{CD-intro} and the tensor in the left-hand side of \cref{eq:substatic}. 
We discuss  the $\mathrm{CD}(0, 1)$ conformal change and Li-Xia connections in more details in \cref{app:LiXia}. 
 We also point out that the conformal metric $\tilde g$ has a natural physical interpretation in the context of static spacetimes, where it is referred to as {\em optical metric}. We give some more details on this point at the end of \cref{app:phys_mot}. 

We will denote by $\rho$ the $\tilde g$-distance from a point $p\in M$, or the signed $\tilde g$-distance from a smooth strictly mean-convex hypersurface $\Sigma$ homologous to the boundary.
We give some more details on this second case, which is slightly less classical but will be crucial for the Willmore-type inequality~\eqref{eq:willmore-intro} discussed below and in turn for the proof of Theorem~\ref{thm:isoperimetric_intro}. With homologous to the boundary we mean that there exists a compact domain $\Omega$ with boundary $\pa\Omega=\pa M\sqcup\Sigma$, and by strictly mean-convex we understand that $\Sigma$ has pointwise positive $g$-mean curvature $\HHH$ with respect to the normal pointing towards infinity. We always choose the signed distance $\rho$ to be positive in the noncompact region $M\setminus\Omega$, that is,
\begin{equation}
\label{eq:signed_distance}
\rho(x)\,=\,\begin{cases}
{\rm d}_{\tilde g}(x,\Sigma) & \hbox{if $x\not\in\Omega$},
\\
-{\rm d}_{\tilde g}(x,\Sigma) & \hbox{if $x\in\Omega$}.
\end{cases}
\end{equation}

Both in the case of the distance from a point and in the case of the signed distance from a hypersurface, through an analysis of the evolution of the mean curvature of the level sets of $\rho$, we come up (see \cref{thm:H_bound} and \cref{pro:H_bound_Sigma}) with the following inequality
\begin{equation}
\label{eq:Lapl_comp_intro}
0\,<\,\frac{\HHH}{f}\,=\,\De\rho+\frac{1}{f}\langle\na f\,|\,\na\rho\rangle\,\leq\,\frac{n-1}{\eta}\,,
\end{equation}
where $\HHH$ denotes the $g$-mean curvature of a level set of the $\tilde{g}$-distance $\rho$, and $\eta$ denotes an useful auxiliary function that will be called {\em reparametrized distance}. It is defined by the first order PDE~\cref{eq:eta}, when the distance from a point is concerned, and in \cref{eq:eta_Sigma} when $\rho$ is the $\tilde{g}$-distance from a hypersurface. The function 
 $\eta$ represents the distance along the radial geodesics computed with respect to the metric $\overline g=f^2 g=f^4 \tilde g$. This third conformal metric will not play a prominent role in the paper, but we will take some advantage from this along the proof of \cref{thm:isoperimetric_intro}.
More details on this point and further comments on $\eta$ (in particular its relation with the weighted connection introduced by Li--Xia) may be found in \cref{rem:eta}.


We remark that \cref{eq:Lapl_comp_intro} could be derived also from \cite[Theorem~3.2]{Wylie}, rewriting it in the substatic setting thanks to the conformal relation with $\mathrm{CD}(0,1)$-metrics. Nevertheless, we have preferred to include a full proof of it, in order to emphasize the role of $\eta$ and to show the substatic point of view. 

A main consequence we draw out of the Laplacian Comparison Theorem above is a Bishop--Gromov Monotonicity Theorem. We state here a version substantially gathering Theorem~\ref{thm:BG_nb} and Theorem~\ref{thm:BGvolumi_nb} below.

\begin{thmx}[Substatic Bishop--Gromov]
\label{thm:BG_intro}
Let $(M,g,f)$ be a substatic triple. Suppose that $M\setminus\pa M$ is geodesically complete with respect to the metric $\tilde g=g/f^2$. Let $\rho$ be the $\tilde g$-distance function from a point or the signed $\tilde g$-distance function from a strictly mean-convex hypersurface $\Sigma$ homologous to $\partial M$ and disjoint from it.
Let $\eta$ be the corresponding reparametrized distance, defined by \cref{eq:eta} or by \cref{eq:eta_Sigma}, and let ${\rm Cut}^{\tilde g}$ be the cut locus of the point/hypersurface. Then, for any $k>0$, the functions
\begin{equation}
\label{eq:avr}
A(t)\,=\,\frac{1}{|\Sph^{n-1}|}\int_{\{\rho=t\}\setminus {\rm Cut}^{\tilde g}}\frac{1}{\eta^{n-1}}d\sigma\,,\qquad V(t)\,=\,\frac{1}{|\mathbb{B}^n|t^k}\int_{\{0\leq\rho\leq t\}}\frac{\rho^{k-1}}{f\eta^{n-1}}d\mu\,,
\end{equation}
are well defined and monotonically nonincreasing. 
Furthermore:
\begin{itemize}
\item if $A(t_1)=A(t_2)$ for $0<t_1<t_2$, then the set $\{t_1\leq\rho\leq t_2\}$ is isometric to $[t_1,t_2]\times\Sigma$, for some $(n-1)$-dimensional manifold $(\Sigma,g_0)$, with metric
\[
g=f^2\,d\rho\otimes d\rho+\eta^2g_0\,;
\]
\item if $V(t_1)=V(t_2)$ for $0<t_1<t_2$, then the set $\{0\leq\rho\leq t_2\}$ is isometric to $[0,t_2]\times\Sigma$, for some $(n-1)$-dimensional manifold $(\Sigma,g_0)$, with metric
\[
g=f^2\,d\rho\otimes d\rho+\eta^2g_0\,;
\]
in the case where $\rho$ is the distance from a point $x$, then $f$ and $\eta$ are functions of $\rho$ only and $g_0=f(x)^{-2}g_{\Sph^{n-1}}$. 
\end{itemize}
\end{thmx}

A first $\mathrm{CD}(0, 1)$-version of Bishop--Gromov monotonicity has been obtained by Wylie--Yeroshkin \cite{wylie-yeroshkin}. Various further generalizations then arose; dropping any attempt to be complete, we mention \cite{kuwae-li, lu-minguzzi-ohta, kuwae-sakurai}, and leave the interested reader to the references therein and to Ohta's monograph \cite{ohta-book}. However, even appealing to the conformal change $g/f^2$, formally relating the $\mathrm{CD}(0, 1)$ to the substatic condition, it does not seem straightforward to deduce a Bishop--Gromov statement in the form above, that will be ruling the Willmore-type inequality \cref{eq:willmore-intro} below and in turn \cref{eq:isoperimetric_intro}. To the authors' knowledge, the rigidity statements contained in \cref{thm:BG_intro} have not been considered in literature yet.   

In the case where $f$ is constant equal to $1$, then both $\rho$ and $\eta$ coincide with the $g$-distance,
hence we recover the standard Bishop--Gromov monotonicity for nonnegative Ricci tensor.
Remarkably, in contrast with the standard Bishop--Gromov, in our setting we do not have to require the boundary $\partial M$ to be empty. This is because, since we have $f=0$ at $\pa M$, the boundary becomes an end with respect to the conformal metric $\tilde g=g/f^2$, see \cref{le:completeness}. As a consequence, we do not have issues when the geodesic spheres $\{\rho=t\}$ intersect the boundary, simply because the boundary is at infinite $\tilde g$-distance so the intersection is always empty. 
{On the other hand, in general $\tilde g$-geodesics may have finite length when going towards the ends of $M$. To avoid this, we added the assumption of $\tilde g$-geodesic completeness in the statement. The main type of ends considered in this paper (the $f$-complete ends introduced just below) will be $\tilde g$-geodesically complete by definition.}



Proceeding in analogy with the nonnegative Ricci curvature case, we are interested in defining a suitable Asymptotic Volume Ratio, motivated by the Bishop--Gromov monotonicity above. In order to get a satisfactory notion, we first aim at understanding basic properties of the ends of substatic manifolds, at least under asymptotic assumptions on the potential of $f$. A fundamental tool for this kind of study in the classical theory is Cheeger--Gromoll Splitting Theorem \cite{cheeger-gromoll}.

Wylie \cite{Wylie} in fact exploited the Laplacian Comparison Theorem to prove a splitting theorem in the ${\rm CD}(0,1)$ setting, that will in turn provide surprising pieces of information in the conformal substatic setting. For our main geometric goals, the kind of end that we will be mostly interested in is that of $f$-complete ends. We say that an end is {\em $f$-complete} if for any $g$-unit speed curve $\gamma:[0,+\infty)\to M$ going to infinity along the end it holds
\begin{equation}
\label{eq:f-completeness_intro}
\lim_{t\to +\infty}\rho(\gamma(t))\,=\,+\infty\,,\qquad\int_0^{+\infty}f(\gamma(t)) dt\,=\,+\infty\,.
\end{equation}
The first condition is essentially asking that the end remains an end even with respect to the conformal metric $\tilde g=g/f^2$. In other words, the first condition is equivalent to the requirement that the end is geodesically complete with respect to the metric $\tilde g$. The second condition instead is a way to ensure that the reparametrized distance $\eta$ diverges to $+\infty$ 
and is connected to an analogous definition given in~\cite{Wylie} in the ${\rm CD}(0,1)$ framework. Further discussion and comments on this definition can be found after Definition~\ref{def:f-complete}. 
It is easy to check that an end is $f$-complete whenever there exists a constant $c$ such that $cr^{-k}<f<cr^k$ for $0<k<1$ at sufficiently large distances, where $r$ is the $g$-distance from a point (see \cref{pro:fcomplete_AF}).


We provide here the full statement of the substatic Splitting Theorem.

\begin{thmx}[Substatic Splitting Theorem]
\label{thm:splitting_intro}
Let $(M,g,f)$ be a substatic triple with ends that are all $f$-complete. If there is more than one end, then $(M,g)$ is isometric to
$$
(\R\times\Sigma,\,f^2\,ds\otimes ds+g_\Sigma)\,,
$$ 
for some $(n-1)$-dimensional Riemannian manifold $(\Sigma,g_\Sigma)$. In particular, if $\partial M$ is nonempty, then $(M, g, f)$ has only one end. 
\end{thmx}


Finally, we point out that similar arguments can be performed also for a different kind of ends, known as conformally compact ends, see Theorems~\ref{thm:splitting_cc} and~\ref{thm:splitting_mixed}.
In particular we will show in Theorem~\ref{thm:splitting_cc} that conformally compact substatic manifolds necessarily have connected conformal infinity, generalizing a known result in the literature of \emph{static vacuum solutions}~\cite{Chrusciel_Simon}.
Static vacuum solutions are in fact substatic triples such that \cref{eq:substatic} is satisfied with equality on the whole space.

Focusing now for simplicity on $f$-complete substatic triples that have only one end,
one would then be led to define the {\em Asymptotic Volume Ratio} as
\begin{equation}
\label{eq:avr-intro}
{\rm AVR}(M,g,f)\,=\,\frac{1}{|\Sph^{n-1}|}\lim_{t\to+\infty}\int_{\{\rho=t\}}\frac{1}{\eta^{n-1}}d\sigma\,=\,\frac{1}{|\mathbb{B}^n|}\lim_{t\to+\infty}\frac{1}{t^n}\int_{\{\rho\leq t\}}\frac{\rho^{n-1}}{f\eta^{n-1}}d\mu\,.
\end{equation}
with $\rho$ denoting the $g/f^2$-distance from  a mean-convex hypersurface homologous to $\partial M$ or from a point, if the boundary is empty. 
The fact that both limits above give the same result is easy to establish. However, we have to make sure that such quantity is independent of the initial hypersurface. We accomplish this task under the assumption of \emph{uniformity of the end}, meaning that the quotient $\eta_x/\eta_y$ of the reparametrized distances with respect to two different points $x,y$ converges uniformly to $1$ at infinity. Again, such condition is immediately checked to be fulfilled in the asymptotically flat regime. More generally, it can in fact be inferred under a natural decay condition on the gradient of $f$ only, see \cref{prop:condition-uniformatinf}.  

Exploiting the global features of our Bishop--Gromov Theorem
we obtain the following Willmore-like inequality for mean-convex hypersurfaces homologous to $\partial M$.

\begin{thmx}[Substatic Willmore inequality]
\label{thm:Willmore_intro}
Let $(M,g,f)$ be a substatic triple with one uniform $f$-complete end. Let $\Sigma$ be a hypersurface homologous to the boundary. Suppose that the mean curvature $\HHH$ of $\Sigma$ with respect to the normal pointing towards infinity satisfies $\HHH> 0$ pointwise. Then
\begin{equation}
\label{eq:willmore-intro}
\int_\Sigma \left[\frac{\HHH}{(n-1)f}\right]^{n-1}d\sigma\,\geq\,
{\rm AVR}(M,g,f)\,|\Sph^{n-1}|\,.
\end{equation}
Furthermore, if the equality holds, then the noncompact connected component $U$ of $M\setminus\Sigma$ is isometric to $[0,+\infty)\times\Sigma$ with metric 
\[
g\,=\,f^2 d\rho\otimes d\rho+\eta^2 g_0\,,
\]
where $g_0$ is a metric on $\Sigma$. 
\end{thmx}
Notice that if there are multiple ends, the above inequality is trivial, since in this case the Asymptotic Volume Ratio of each end is zero as pointed out in \cref{lem:positiveavr->1end}.

In the classical nonnegative Ricci curvature setting, the above result was obtained in \cite{agostiniani_sharpgeometricinequalitiesclosed_2020}. However, the more elementary proof we propose displays more resemblances with the alternative argument of Wang \cite{wang-willmore}. 

\smallskip

The validity of the above Willmore-type inequality very naturally suggests the isoperimetric inequality \cref{eq:isoperimetric_intro}. Indeed, assume that  smooth area minimizers exist among hypersurfaces $\Sigma_V$  enclosing a given weighted volume $\int_\Omega f d\mu = V$ with $\partial M$, for any given value $V$. We will call such $\Sigma_V$ $f$-isoperimetric. Then, through standard variation formulas, there must exist a Lagrange multiplier $\lambda \in \R$ such that
\begin{equation}
\label{eq:lagrange-intro}
    \int_\Sigma (\HHH - \lambda f) \varphi d\sigma = 0
\end{equation}
for any $\varphi \in \mathscr{C}^\infty_c (\Sigma)$, implying that the mean curvature of $\Sigma$ satisfies $\HHH/ f = \lambda$. Letting $I_f (V) = \abs{\Sigma_V}$, one has that $(I_f)' (V)$ is proportional to $\lambda$. If this multiplier happens to positive, it is sharply estimated in terms of $I_f (V)$  by \cref{eq:willmore-intro}. The resulting differential inequality leads to 
\begin{equation}
\label{eq:ineq-isop-intro}
\abs{\Sigma_V}^{\frac{n}{n-1}} - \abs{\Sigma_{V_0}}^{\frac{n}{n-1}} \geq n (\mathrm{AVR}(M, g, f)\abs{\Sf^{n-1}})^{\frac{1}{n-1}}(V- V_0)
\end{equation}
for any $V_0 < V$.
Now, if $\partial M$ happens to be, in addition to minimal, also area-minimizing, \cref{eq:ineq-isop-intro} directly implies \cref{eq:isoperimetric_intro} for $\Sigma_V$ in the limit as $V_0 \to 0^+$. But $\Sigma_V$ being the best competitor, it holds for any $\Sigma$ as in the claimed statement.

The assumptions of $f$-completeness and uniformity at infinity are obviously added in order to count on the validity of \cref{thm:Willmore_intro}. The additional requirement of existence of a nonminimal outward minimizing exhaustion (see Section~\ref{sec:isoperimetric}) is ultimately added in order to overcome the problem of the possible nonexistence of $f$-isoperimetric $\Sigma_V$'s. We follow the general strategy devised by Kleiner \cite{kleiner}, and reinterpreted in the nonnegative Ricci curvature setting in \cite{Fogagnolo_Mazzieri-minimising}, that consists in considering $f$-isoperimetric sets constrained in mean-convex boxes, in our case given by the exhaustion.  However, some new geometric difficulties arise, mostly given by the new portion of boundary $\partial M$. They will be overcome by exploiting the fact that such boundary is in turn \emph{a priori} outermost area-minimizing, a new piece of information that we obtain through an argument involving the Mean Curvature Flow of the outward minimizing exhaustion (see \cref{prop:outermost}), and by discovering that the constrained $f$-isoperimetric sets crucially never touch $\partial M$, see \cref{thm:f-isoperimetric}.
We will in the end have all the tools at hand to run the argument sketched above, under the usual dimensional threshold ensuring the constrained $f$-isoperimetric sets to be regular enough. 
The strong rigidity statement contained in \cref{thm:isoperimetric_intro} will stem from the fact that in case of equality all of the $\Sigma_V$ must satisfy the equality in \cref{eq:willmore-intro}. The rigidity statement of \cref{thm:Willmore_intro} will be thus complemented with the additional information given by $\HHH = \lambda f$, forcing the metric to split as \cref{eq:rigidiso_intro}.



\subsection{Further directions}

The results presented in this paper raise a number of natural questions, especially out of our main geometric inequalities, \cref{thm:isoperimetric_intro} and \cref{thm:Willmore_intro}. 
The Willmore-type inequality in nonnegative Ricci curvature \cite[Theorem 1.1]{agostiniani_sharpgeometricinequalitiesclosed_2020} has been first obtained with a completely different technique, involving the evolution of the initial hypersurface along the level sets of a harmonic potential function. Understanding a version of such a route in the substatic context may have various interests. First of all, it may allow to remove the mean-convexity assumption on $\Sigma$ we have in \cref{thm:Willmore_intro} in favour of the absolute value of the mean curvature in \cref{eq:willmore-intro}. Secondly, and more interestingly, it would suggest the viability of a suitable version of the analysis through $p$-harmonic functions performed in nonnegative curvature in \cite{benatti2022minkowski}, likely leading to a new substatic Minkowski inequality, potentially stronger than our Willmore-type. Moreover, studying the behaviour of such substatic $p$-harmonic functions may have implications in the existence of the weak Inverse Mean Curvature Flow \cite{huisken_inverse_2001} in the substatic regime, furnishing a vast extension of the important existence results in nonnegative Ricci curvature \cite{mari-rigoli-setti}. Recalling the outward minimizing properties of the evolving hypersurfaces \cite[Minimizing Hull Property 1.4]{huisken_inverse_2001}, the existence of the weak IMCF would imply the a priori existence of the outward minimizing exhaustion requested in \cref{thm:isoperimetric_intro}.
In the special case of asymptotically flat static vacuum solutions, the weak IMCF has already been introduced and employed to prove Minkowski-type inequalities in~\cite{wei,mccormick,Har_Wan}. Such inequalities are lower bound on the integral of $f \HHH$, and, as such, do not seem related to \cref{eq:willmore-intro}.

It would also be rather interesting to explore other approaches for the proof of the $f$-Isoperimetric Inequality \cref{eq:isoperimetric_intro} as well, possibly allowing to remove the dimensional threshold. Antonelli--Pasqua\-letto--Pozzetta--Semola \cite[Theorem 1.1]{antonelli-pasqualetto-pozzetta-semola} provided a natural and very strong proof in the nonnegative Ricci curvature case  taking advantage of a generalized compactness result \cite{nardulli, antonelli-fogagnolo-pozzetta, antonelli-nardulli-pozzetta} for isoperimetric minimizing sequences in the nonsmooth $\mathrm{RCD}$ setting. This immediately invites to study the nonsmooth counterpart of the substatic condition. A possible key for this may lie in the recent optimal transport equivalent definition of $\mathrm{CD}(0, 1)$ given in \cite{sakurai-convexity}. 

Another, completely different approach one may undertake consists in Brendle's \cite{brendle_sobolevinequalitiesmanifoldsnonnegative_2021}, building on the ABP method applied to a torsion problem with Neumann condition. A substatic version of such approach promises to deal with the PDE considered in \cite{Li_Xia_17, fogagnolo-pinamonti} in relation with the Heintze-Karcher inequality. 
We also point out that both these alternative approaches should have consequences in going beyond the dimensional threshold we imposed.

\smallskip

From the comparison geometry point of view, the validity of the Splitting Theorem and of the Bishop--Gromov monotonicity strongly suggests that other classical results, such as the Cheng eigenvalue estimate and Cheng--Yau gradient estimate, should have analogues in the substatic setting. A promising advance in this direction has been obtained in the $\mathrm{CD}(0,1)$ setting \cite{fujitani2022functional}. 

\smallskip

It may also be interesting to study compact substatic triples. Important models for this class of manifolds are given by static solutions with positive cosmological constant, most notably the Schwarzschild--de Sitter and Reissner--Nordstr\"om--de Sitter spacetimes, corresponding to~\eqref{eq:models} with $\Lambda$ positive and $(\Sigma, g_\Sigma)$ a round sphere.
Another natural direction is to investigate what can be said for the more general problem of studying triples $(M,g,f)$ satisfying
\begin{equation}
\label{eq:sub-vstatic}
f\Ric - \nana f + (\De f)g \geq -\mu g\,,\quad\mu\in\R\,.
\end{equation}
The case $\mu=0$ corresponds to the substatic condition. The case $\mu\neq 0$ is also of interest: triples that saturate~\eqref{eq:sub-vstatic} for $\mu\neq 0$ are called $V$-static and are connected with the critical point equation and the Besse conjecture, see~\cite{Fang_Yuan,He} and references therein for more details on these topics. We mention that inequality~\eqref{eq:sub-vstatic} has been considered in~\cite{Zeng}, where an almost-Schur inequality has been proved and exploited to generalize results in~\cite{Cheng,Li_Xia_19}.


\subsection{Structure of the paper}

In Section~\ref{sec:Riccati} we compute the evolution of the mean curvature of geodesic spheres, leading to the aforementioned Laplacian Comparison Theorem, formula~\eqref{eq:Lapl_comp_intro} (see Theorem~\ref{thm:H_bound}). Building on it, we prove Theorem~\ref{thm:BG_intro}, first for the functional $A$ (Theorem~\ref{thm:BG_nb}) and then for the functional $V$ (Theorem~\ref{thm:BGvolumi_nb}).
Section~\ref{sec:splitting} is dedicated to the proof of the Splitting Theorem. We first analyze the most important case of $f$-complete ends and prove Theorem~\ref{thm:splitting_intro} (Subsection~\ref{sub:splitting_fcomplete}), then we discuss analogous results for conformally compact ends as well, see Theorems~\ref{thm:splitting_cc} and~\ref{thm:splitting_mixed}.
In Section~\ref{sec:AVR} we introduce the notions of uniform ends (Definition~\ref{def:uniform}) and Asymptotic Volume Ratio (Definition~\ref{def:AVR}) and prove the Willmore Inequality (see Theorem~\ref{thm:Willmore}).
Finally, in Section~\ref{sec:isoperimetric} we prove Theorem~\ref{thm:isoperimetric_intro}.
We include an Appendix encompassing the physical motivation for the substatic condition, the conformal relation with the $\mathrm{CD}(0, 1)$ curvature-dimension condition and some additional comments.

\medskip
\subsection*{Acknowledgements}
The work was initiated during the authors' participation at the conference\emph{Special Riemannian Metrics and Curvature Functionals} held at \emph{Centro De Giorgi} in Pisa in 2022. A substantial
part of the work has been carried out during the authors' attendance to the \emph{Thematic Program on Nonsmooth Riemannian and Lorentzian Geometry} that took place at the Fields Institute in Toronto. 
They warmly thank the staff, the organizers and the colleagues for the wonderful atmosphere and the excellent working conditions set up there. 

During the preparation of the work, M. F. was supported by the European Union – NextGenerationEU and by the University of Padova under the 2021 STARS Grants@Unipd programme ``QuASAR".

The authors are members of Gruppo Nazionale per l’Analisi Matematica, la Probabilit\`a e le loro Applicazioni (GNAMPA), which is part of the Istituto
Nazionale di Alta Matematica (INdAM), and are partially funded by the GNAMPA project ``Problemi al bordo e applicazioni geometriche".

The authors are grateful to Lucas Ambrozio, Gioacchino Antonelli, Luca Benatti, Camillo Brena, Philippe Castillion, Nicola Gigli, Marc Herzlich, Lorenzo Mazzieri, Marco Pozzetta and Eric Woolgar   
for their interest in their work and for various useful conversations.

\section{Riccati comparison and Bishop-Gromov Theorem}
\label{sec:Riccati}

\subsection{Evolution of the mean curvature}
\label{sub:evolution}

Let $(M,g,f)$ be a substatic solution.
Since $f$ does not vanish inside $M$ by definition, the metric $\tilde g=g/f^2$ is well defined in $M\setminus\pa M$. Let $\rho$ be the distance function from a point $p\in M\setminus\pa M$ with respect to the metric $\tilde g$ and consider Riemannian polar coordinates $(\rho,\theta^1,\dots,\theta^{n-1})$ in a neighborhood $U$ of $p$. In this subsection we focus our computation in $U$ and we assume that $U$ does not contain points in the cut locus of $p$. This guarantees that $\rho$ is smooth in $U$, that $|d\rho|_{\tilde g}=1$ at all points of $U$ and that the metric $\tilde g$ has the following form:
\begin{equation}
\label{eq:gamma}
\tilde g\,=\,d\rho\otimes d\rho+\tilde g_{ij}(\rho,\theta^1,\dots,\theta^{n-1})d\theta^i\otimes d\theta^j\,.
\end{equation}
We denote by $\widetilde\nabla$ the Levi-Civita connection with respect to $\tilde g$. It is well known that the Hessian of a distance function satisfies the inequality
\begin{equation}
\label{eq:Hessian_estimate_gamma}
\big|\widetilde\na^2\rho\big|^2_{\tilde g}\geq\frac{(\De_{\tilde g} \rho)^2}{n-1}\,.
\end{equation}
We now write down this inequality in terms of the original metric. To this end, we first compute
\begin{equation}
\label{eq:useful_formulas}
\begin{aligned}
|\na\rho|^2\,&=\,\frac{1}{f^2}\,,
\\
\widetilde\na^2\rho\,&=\,\nana\rho\,+\,\frac{1}{f}\left(d\rho\otimes df+df\otimes d\rho-\langle \na\rho\,|\,\na f\rangle g\right)\,,
\\
\De_{\tilde g}\rho\,&=\,f^2\De\rho-(n-2)f\langle\na\rho\,|\,\na f\rangle\,.
\end{aligned}
\end{equation}
Using these identities, with some computations we can rewrite~\eqref{eq:Hessian_estimate_gamma} in terms of the original metric as 
\begin{equation}
\label{eq:Hessian_estimate}
|\nana\rho|^2\,\geq\,
\frac{(\De\rho)^2}{n-1}
-\frac{n-2}{n-1}\frac{1}{f^2}\langle \na\rho\,|\,\na f\rangle^2
+\frac{2}{n-1}\frac{1}{f}\De\rho\langle \na\rho\,|\,\na f\rangle
+\frac{2}{f^4}|\na f|^2\,.
\end{equation}
From this formula onwards we just focus on the original metric $g$. 
From~\eqref{eq:gamma}, with respect to the coordinates $(\rho,\theta^1,\dots,\theta^{n-1})$, we have
$$
g\,=\,f^2 \, d\rho\otimes d\rho+g_{ij}(\rho,\theta^1,\dots,\theta^{n-1})d\theta^i\otimes d\theta^j\,.
$$
We are interested in the evolution of the mean curvature $\HHH$ of the level sets of $\rho$ with respect to the metric $g$. We have
$$
\HHH\,=\,\frac{\De\rho}{|\na\rho|}\,-\,\frac{\nana\rho(\na\rho,\na\rho)}{|\na\rho|^3}\,=\,f\De\rho\,-\,\frac{1}{2}f^3\langle \na f\,|\,\na\rho\rangle\,=\,f\De\rho\,+\,\langle\na f\,|\,\na\rho\rangle\,.
$$
On the other hand, using the fact that $|\na\rho|=1/f$ and the Bochner formula we compute
$$
\frac{6}{f^4}|\na f|^2\,-\,\frac{2}{f^3}\De f\,=\,\De|\na\rho|^2\,=\,
2\,|\nana\rho|^2\,+\,2\,\Ric(\na\rho,\na\rho)\,+\,2\,\langle\na\De\rho\,|\,\na\rho\rangle\,.
$$
Combining this with~\eqref{eq:substatic}, some computations lead to
$$
\langle\na\De\rho\,|\,\na\rho\rangle\,\leq\,\De\rho\langle \na\rho\,|\,\na f\rangle-\frac{1}{f}\nana f(\na\rho,\na\rho)\,-\,|\nana\rho|^2
$$
We can use this information to find the evolution of $\HHH$:
\begin{align}
\langle\na\HHH\,|\,\na\rho\rangle\,&\leq\,-f|\nana\rho|^2\,+\,\frac{2}{f^3}|\na f|^2\,+\,\De\rho\langle\na\rho\,|\,\na f\rangle
\\
&\leq\,-f\frac{(\De\rho)^2}{n-1}\,+\,\frac{n-3}{n-1}\De\rho\langle\na\rho\,|\,\na f\rangle\,+\,\frac{n-2}{n-1}\frac{1}{f}\left(\langle\na\rho\,|\,\na f\rangle\right)^2
\\
&=\,-\frac{1}{n-1}\frac{1}{f}\HHH^2\,+\,\frac{1}{f}\HHH\langle\na\rho\,|\,\na f\rangle\,,
\end{align}
where in the second inequality we have used estimate~\eqref{eq:Hessian_estimate}. This formula can be rewritten as
$$
\bigg\langle\na\left(\frac{\HHH}{f}\right)\,\bigg|\,
\na\rho\bigg\rangle
\,\leq\,-\frac{1}{n-1}\left(\frac{\HHH}{f}\right)^2\,.
$$
In other words, we have found the following formula for the evolution of the mean curvature of the level sets of $\rho$.
\begin{lemma}
	\label{le:Riccati}
In the notations above, at any point of $U$ the evolution of the mean curvature $\HHH$ along $\rho$ satisfies
\begin{equation}
\label{eq:H_Riccati}
\frac{\pa}{\pa\rho}\left(\frac{\HHH}{f}\right)\,\leq\,-\frac{1}{n-1}\,\HHH^2\,.
\end{equation}
\end{lemma}

\subsection{Bounds on the mean curvature and Laplacian comparison}
\label{sub:Laplacian_comparison}

Let $p\in M\setminus\pa M$ and $\rho$ be the distance function from $p$ with respect to the metric $\tilde g=g/f^2$.
We assume that $M\setminus\pa M$ is geodesically complete with respect to the metric $\tilde g$.
We denote by ${\rm Cut}^{\tilde g}_p(M)$ the cut locus of $p$, again with respect to $\tilde g$.
The function $\rho$ is smooth in $U=M\setminus({\rm Cut}^{\tilde g}_p(M)\cup\{p\})$, where ${\rm Cut}^{\tilde g}_p(M)$ is the cut locus of $p$ with respect to the metric $\tilde g$.

For every $\theta\in\Sph^{n-1}\subset T_p M$, we denote by $\sigma_\theta$ the geodesic starting from $p$ in the direction $\theta$ and by 
$\tau(\theta)\in(0,+\infty]$ the smallest positive value such that $\sigma_\theta(\tau(\theta))\in {\rm Cut}^{\tilde g}_p(M)$. We recall from~\cite[Proposition~2.9, Chapter~13]{doCarmo} that $\tau:\Sph^{n-1}\to(0,+\infty]$ is a continuous function.
Notice that there is a diffeomorphism between $U$ and the set
$$
\{(\rho,\theta)\in(0,+\infty)\times\Sph^{n-1}\,:\,\rho<\tau(\theta)\}\,,
$$
hence we can use $\rho,\theta$ as coordinates in $U$.

We can now exploit
Lemma~\ref{le:Riccati} to find a bound for $\HHH$ in $U$. To this end, given a positive function $\eta\in\mathscr{C}^\infty$ we use~\eqref{eq:H_Riccati} to compute
\begin{equation}
\label{eq:Riccati_eta}
\frac{\pa}{\pa\rho}\left(\frac{f}{\HHH}-\frac{\eta}{n-1}\right)\,=\,
-\frac{f^2}{\HHH^2}\frac{\pa}{\pa\rho}\left(\frac{\HHH}{f}\right)-\frac{1}{n-1}\frac{\pa\eta}{\pa\rho}\,\geq\,\frac{1}{n-1}\left(f^2-\frac{\pa\eta}{\pa\rho}\right)\,.
\end{equation}
We then choose $\eta$ so that the the right hand side vanishes pointwise. Since $f$ is smooth, the equation
$$
\frac{\pa\eta}{\pa\rho}\,=\,f^2
$$
can be solved and yields a unique solution once we fix its value on a level set of $\rho$.

\begin{proposition}
\label{pro:eta}
There exists a unique function $\eta\in\mathscr{C}^\infty(U)$ satisfying
\begin{equation}
\label{eq:eta}
\begin{dcases}
\frac{\pa\eta}{\pa\rho}\,=\,f^2\,,
\\
\lim_{\rho\to 0^+}\eta\,=\,0\,.
\end{dcases}
\end{equation}
\end{proposition}

\begin{remark}
\label{rmk:eta_nonregularoncutlocus}
It should be remarked that $\eta$ is not even necessarily continuous outside $U$. In fact, if there are two minimizing geodesics from $p$ to $q\in {\rm Cut}^{\tilde g}_p(M)$, the function $\eta$ may behave differently on the two geodesics, hence the limit of $\eta$ as we approach $q$ along the two different geodesics would give different results.
\end{remark}

\begin{proof}
For every $\ep>0$, consider the function $\eta_\ep$ defined by 
\begin{equation}
\label{eq:eta_ep}
\begin{dcases}
\frac{\pa\eta_\ep}{\pa\rho}\,=\,f^2 & \hbox{ in }\{\rho>\ep\}\cap U\,,
\\
\eta_\ep\,=\,0 & \hbox{on }\{\rho=\ep\}\,.
\end{dcases}
\end{equation}
Since $f$ is smooth, it is well known (see for instance~\cite[Section~3.2.4]{Evans}) that~\eqref{eq:eta_ep} can be solved and yields a unique $\mathscr{C}^2$ solution. Furthermore, by differentiating the first equation we find that the first derivatives $\pa_\alpha\eta$ also solve a first order PDE, namely $\pa_\rho\pa_\alpha\eta=\pa_\alpha f^2$. Since $\pa_\alpha f^2$ is smooth, it follows that the derivatives $\pa_\alpha \eta$ are also $\mathscr{C}^2$. Proceeding this way, we deduce that $\eta_\ep$ is smooth.

It is now sufficient to pass to the limit as $\ep\to 0$ using Ascoli--Arzel\`a. Since the functions $\eta_\ep$ (and their derivatives as well) are uniformly continuous and uniformly bounded on any compact domain inside $U$, it follows that $\eta_\ep$ converge to a smooth function $\eta$, defined on the whole $U$.

Concerning uniqueness, if there were two different solutions $\eta_1,\eta_2$ of~\eqref{eq:eta}, then the difference $\eta_1-\eta_2$ would have derivative equal to zero along the direction $\pa/\pa\rho$. Since the limit as $\rho\to 0$ is zero, one immediately obtains $\eta_1-\eta_2=0$.
\end{proof}

\begin{remark}[The reparametrized distance $\eta$]
\label{rem:eta}
The function $\eta$ is also called {\em reparametrized distance}. The reason for this terminology is that the radial $\tilde g$-geodesics from our point $p$, reparametri\-zed with respect to $\eta$, are geodesics for the weighted connection
\[
\D_X Y=\nabla_X Y+\frac{1}{f}g(X,Y)\na f\,.
\]
The significance of such connection is that the Ricci tensor associated to $\D$ is nonnegative if and only if the substatic condition is satisfied. More details on this can be found in \cref{app:LiXia} and in~\cite{Li_Xia_17} (see also~\cite{kennard_wylie_yeroshkin,wylie-yeroshkin} for further discussions in the conformally related ${\rm CD(0,1)}$ setting). Alternatively, as mentioned in the Introduction, $\eta$ can also be seen to represent the distance along radial $\tilde g$-geodesics with respect to the metric $\overline g=f^2 g$. In fact, if $\sigma:[0,S]\to M$ is a radial $\tilde g$-geodesic with $\sigma(0)=p$ and $\dot\sigma=\pa/\pa\rho$, we have $|\dot\sigma(s)|_{\tilde g}=1$ and so
\[
{\rm d}_{\overline g}(\gamma(S),p)\,=\,\int_0^S|\dot\sigma(s)|_{\overline g}ds\,=\,\int_0^S f^2(\sigma(s))ds\,=\,\eta(S)\,.
\]
\end{remark}

We are now in a position to prove the following crucial bound on the mean curvature of the level sets of $f$. This bound can be naturally interpreted as the Laplacian Comparison for the conformal distance function $\rho$. This result corresponds to~\cite[Theorem~3.2]{Wylie} in the ${\rm CD}(0,1)$ framework.

\begin{theorem}[Laplacian Comparison]
\label{thm:H_bound}
Let $(M,g,f)$ be a substatic triple. Suppose that $M\setminus\pa M$ is geodesically complete with respect to the metric $\tilde g=g/f^2$. Let $\rho$ be the distance function to a point $p\in M\setminus\pa M$ with respect to the metric $\tilde g=g/f^2$ and $\eta$ be the solution to~\eqref{eq:eta}. Then the mean curvature $\HHH$ of the level sets of $\rho$ with respect to the metric $g$ satisfies
\begin{equation}
\label{eq:H_bound}
0\,<\,\frac{\HHH}{f}\,=\,\De\rho+\frac{1}{f}\langle\na f\,|\,\na\rho\rangle\,\leq\,\frac{n-1}{\eta}
\end{equation}
in the classical sense in the open dense set $U=M\setminus ({\rm Cut}^{\tilde g}_p(M)\cup\{p\})$.
\end{theorem}

\begin{proof}
Let us first prove the thesis working inside the open set $U$.
From the definition of $\eta$ and~\eqref{eq:Riccati_eta}, we immediately deduce
$$
\frac{\pa}{\pa\rho}\left(\frac{f}{\HHH}-\frac{\eta}{n-1}\right)\,\geq\,0\,.
$$
In other words, the function $f/\HHH-\eta/(n-1)$ is nondecreasing. 

We then estimate its value near the point $p$.
It is well known that, for small $\rho$'s, the mean curvature $\HHH_{\tilde g}$ of the geodesic balls grows as the geodesic balls in Euclidean space, namely $\HHH_{\tilde g}= (n-1)/\rho+o(1/\rho)$ when $\rho$ is sufficiently small. Here we have denoted by $\HHH_{\tilde g}$ the mean curvature with respect to the metric $\tilde g$. This is related to the mean curvature with respect to $g$ by $\HHH=\HHH_{\tilde g}/f+(n-1)\langle\na\rho\,|\,\na f\rangle$. Since $\langle\na\rho\,|\,\na f\rangle$ is bounded in a neighborhood of $p$, we obtain
$$
\HHH=\frac{n-1}{f(p)\rho}+o(1/\rho)\,, \quad \hbox{as }\rho\to 0\,.
$$
In particular, $f/\HHH\to 0$ as $\rho\to 0$. Since $\eta\to 0$ as well by definition, we have obtained that $f/\HHH-\eta/(n-1)\to 0$ when $\rho\to 0$. From the monotonicity of $f/\HHH-\eta/(n-1)$ we then deduce $(n-1)f/\HHH\geq \eta$ on the whole $U$. Since $\eta$ is positive on the whole manifold by construction, the conclusion follows.
\end{proof}

A first important observation is that there is a more effective version of the above inequality. We now argue that the Laplacian comparison obtained in Theorem~\ref{thm:H_bound} and Proposition~\ref{pro:H_bound_Sigma} gives us a vector with nonnegative divergence, which is
\begin{equation}
\label{eq:X}
X\,=\,\frac{f}{\eta^{n-1}}\na\rho\,,
\end{equation}
defined on the open dense set $U=M\setminus ({\rm Cut}^{\tilde g}_p(M)\cup\{p\})$.
In fact:
\begin{align}
{\rm div}\,X\,&=\,\frac{f}{\eta^{n-1}}\left[\De\rho\,+\,\frac{1}{f}\langle\na f\,|\,\na\rho\rangle	\right]
\,-\,(n-1)\frac{f}{\eta^n}\langle\na\eta\,|\,\na\rho\rangle
\\
&\leq\,(n-1)\frac{f}{\eta^{n}}
\,-\,(n-1)\frac{f}{\eta^n}\frac{1}{f^2}\frac{\pa\eta}{\pa\rho}\,=\,0\,.
\end{align}
The inequality ${\rm div}\,X\leq 0$ holds then in the classical sense in the whole $U$. It is crucial that this inequality actually holds in the distributional sense in the whole manifold.

\begin{theorem}
\label{thm:nonnegative_divX}
Let $(M,g,f)$ be a substatic triple. Suppose that $M\setminus\pa M$ is geodesically complete with respect to the metric $\tilde g=g/f^2$. Let $\rho$ be the distance function to a point $p$ with respect to the metric $\tilde g=g/f^2$ and $\eta$ be the solution to~\eqref{eq:eta}. Then the vector $X$ defined in~\eqref{eq:X} has nonpositive divergence in the weak sense in the whole $M\setminus\{p\}$. Namely, for every nonnegative test function $\chi\in\mathscr{C}_c^\infty(M)$ with $\chi\equiv 0$ in a neighborhood of $p$, it holds
$$
\int_{M}\langle X\,|\,\na\chi\rangle d\sigma\,\geq\,0\,.
$$
\end{theorem}

\begin{proof}
Let $\sqrt{g}$ be the volume element of the metric $g$ with respect to the coordinates $(\rho,\theta)$. We first observe that at all points of $U=M\setminus ({\rm Cut}^{\tilde g}_p(M)\cup\{p\})$ it holds
\begin{align}
\frac{\pa}{\pa\rho}\left(\frac{\sqrt{g}}{f\eta^{n-1}}\right)\,&=\,\frac{\pa}{\pa\rho}\left(\frac{f^{n-1}\sqrt{\tilde g}}{\eta^{n-1}}\right)
\\
&=\,
\left[\frac{f^{n-1}\HHH_{\tilde g}}{\eta^{n-1}}\,-\,(n-1)\frac{f^{n+1}}{\eta^n}\,+\,(n-1)\frac{f^{n-2}}{\eta^{n-1}}\frac{\pa f}{\pa\rho}\right]\sqrt{\tilde g}
\\
&=\,
\frac{f^{n}}{\eta^{n-1}}\left[\HHH-(n-1)\frac{1}{f^2}\frac{\pa f}{\pa\rho}\,-\,(n-1)\frac{f}{\eta}\,+\,(n-1)\frac{1}{f^2}\frac{\pa f}{\pa\rho}\right]\sqrt{\tilde g}
\\
\label{eq:riccati_for_sqrtg}
&\leq\,0\,,
\end{align}
where the last inequality follows from the Laplacian comparison. Recalling that in polar coordinates $(\rho,\theta)$ the set $U$ is diffeomorphic to the set of pairs $(\rho,\theta)$ with $\rho<\tau(\theta)$ for a suitable continuous function $\tau:\Sph^{n-1}\to(0,+\infty]$, we can then compute
\begin{align}
\int_{M}\langle X\,|\,\na\chi\rangle d\mu\,&=\,\int_{\Sph^{n-1}}\int_0^{\tau(\theta)}\frac{f}{\eta^{n-1}}\langle \na\rho\,|\,\na\chi\rangle\sqrt{g} d\rho d\theta
\\
&=\,\int_{\Sph^{n-1}}\int_0^{\tau(\theta)}\frac{\pa \chi}{\pa\rho}\frac{\sqrt{g}}{f\eta^{n-1}} d\rho d\theta
\\
&=\,-\int_{\Sph^{n-1}}\!\int_0^{\tau(\theta)}\!\!\chi\,\frac{\pa}{\pa\rho}\left(\frac{\sqrt{g}}{f\eta^{n-1}}\right)\! d\rho d\theta
+\!\int_{\{\theta\in\Sph^{n-1}:\,\tau(\theta)<\infty\}}\left[\chi\frac{\sqrt{g}}{f\eta^{n-1}}\right]\!(\tau(\theta),\theta)d\theta
\\
\label{eq:weak_divergence}
&\qquad\qquad\qquad\qquad\qquad\qquad\qquad\qquad\qquad-\lim_{\ep\to 0^+}\int_{\Sph^{n-1}}\left[\chi\frac{\sqrt{g}}{f\eta^{n-1}}\right]\!(\ep,\theta)d\theta\,.
\end{align}
In the last identity, the first integral is nonnegative thanks to~\eqref{eq:riccati_for_sqrtg} whereas the second integral is nonnegative by construction. Concerning the third and final integral, notice that $\eta$ behaves near $p$ as $f^2(p)\rho$ by definition, hence the integral is easily seen to converge to $\chi(p)|\Sph^{n-1}|/f^{2n-1}(p)$ as $\ep \to 0$. Since we are assuming that $\chi$ vanishes in a neighborhood of $p$, the third integral goes to zero as $\ep\to 0$. The conclusion follows.
\end{proof}

\subsection{Evolution of the mean curvature of hypersurfaces}
\label{sub:growth_mc_hypersurfaces}

Let $\Sigma$ be a compact smooth hypersurface. We are interested to the case where $\Sigma$ is homologous to the boundary: namely, there exists a compact domain $\Omega$ such that $\pa\Omega=\pa M\sqcup\Sigma$. Let $\rho$ be the $\tilde g$-distance from $\Sigma$ in $M\setminus\Omega$. 
If $\Sigma$ is smooth, then so is $\rho$ in a collar of $\Sigma$.
Under the usual assumption of $\tilde g$-geodesic completeness of $M\setminus\pa M$, it is known from~\cite[Proposition~4.6]{Mantegazza_Mennucci} that the open set $U=(M\setminus\Omega)\setminus {\rm Cut}^{\tilde g}_\Sigma(M)$ of the points where $\rho$ is smooth is dense in $M\setminus\Omega$. 
In particular, there is a function $\tau:\Sigma\to(0,+\infty]$ such that the gradient flow of $\pa/\pa\rho$ gives a diffeomorphism between $U$ and
\begin{equation}
\label{eq:simil-polar_coordinates}
\{(\rho,x)\,:\,x\in\Sigma\,,\ \rho\in(0,\tau(x))\}\,.
\end{equation}
It is convenient to estimate the evolution of the geometry of $\Sigma$ with respect to these coordinates.
If $\HHH>0$ at the point $x\in\Sigma$, the evolution of the mean curvature $\HHH(0,x)$ is quite similar to the one described for geodesic spheres in Subsection~\ref{sub:evolution} and~\ref{sub:Laplacian_comparison}. One can then define the function $\eta\in\mathscr{C}^\infty(U)$ as the solution to
\begin{equation}
\label{eq:eta_Sigma}
\begin{dcases}	\frac{\pa\eta}{\pa\rho}(\rho,x)\,=\,f^2(\rho,x)\,,
\\
\eta(0,x)\,=\,(n-1)\frac{f(0,x)}{\HHH(0,x)}\,.
\end{dcases}
\end{equation}
As in Proposition~\ref{pro:eta}, one shows that $\eta$ is well defined and smooth in $U$.
We can then replicate the proof of Theorem~\ref{thm:H_bound} to find $\HHH/f(\rho,x)\leq(n-1)/\eta(\rho,x)$ for any $0\leq \rho<\tau(x)$. In this case, the proof is actually even easier, since $f/\HHH-\eta/(n-1)$ vanishes at $x$ by construction, without the need of proving it. 

Notice that from~\eqref{eq:H_Riccati} we can get interesting information on the evolution of the mean curvature also in the case where $\HHH$ is nonpositive at a point $x\in\Sigma$.
If $\HHH<0$ at $x$, we deduce from~\eqref{eq:H_Riccati} that $\HHH$ must remain negative for all times, whereas if $\HHH=0$ then $\HHH$ remains nonpositive.
Furthermore, even when $\HHH$ is nonpositive one can still define $\eta$ by~\eqref{eq:eta_Sigma} and find $\eta/(n-1)\leq f/\HHH<0$ at the points $(\rho,x)$.
From this fact we can obtain even more information under the assumption that the end is $f$-complete. We  recall that an end is $f$-complete if~\eqref{eq:f-completeness_intro} is satisfied, see also Definition~\ref{def:f-complete} below for a more extensive discussion about this notion.
The main feature of $f$-complete ends is that the end is complete with respect to the metric $\tilde g$ (so that $\rho$ goes to $+\infty$) and $\eta$ grows to $+\infty$ along the end. Since $\eta(x)$ is negative when the mean curvature at $x$ is negative, in particular $\eta$ must reach the value zero, at which point the bound  $\eta/(n-1)\leq f/\HHH<0$ fails. This implies that the line $\rho\mapsto(\rho,x)$ must reach the cut locus before $\eta$ hits zero. Finally, if $\HHH=0$ then from~\eqref{eq:H_Riccati} we would get that $\HHH$ remains nonpositive. If at some point $\HHH$ becomes negative, then the previous argument applies and the line $\rho\mapsto(\rho,x)$ must reach the cut locus. Summarizing, we have obtained the following:
 
\begin{proposition}
\label{pro:H_bound_Sigma}
Let $(M,g,f)$ be a substatic triple. Suppose that $M\setminus\pa M$ is geodesically complete with respect to the metric $\tilde g=g/f^2$. Let $\rho$ be the distance from an hypersurface $\Sigma$ homologous to the boundary with respect to the metric $\tilde g=g/f^2$, and let $\eta$ be the solution to~\eqref{eq:eta_Sigma}. Finally, let $x\in\Sigma$ and consider the evolution $\HHH(\rho,x)$ of the mean curvature in the direction of the end.
\begin{itemize}
\item[$(i)$] If $\HHH(0,x)>0$ then for any $0<\rho<\tau(x)$ it holds 
\begin{equation}
\label{eq:H_bound_Sigma}
0\,<\,\frac{\HHH}{f}(\rho,x)\,\leq\,\frac{n-1}{\eta(\rho,x)}\,.
\end{equation}
\item[$(ii)$] If $\HHH(0,x)<0$, then  for every $0<\rho<\tau(x)$ it holds
\begin{equation}
\label{eq:H_bound_negative_Sigma}
\frac{\HHH}{f}(\rho,x)\,\leq\,\frac{n-1}{\eta(\rho,x)}\,<\,0\,.
\end{equation}
Furthermore, if the ends of $M$ are $f$-complete then $\tau(x)<+\infty$.
\smallskip
\item[$(iii)$] If $\HHH(0,x)=0$, then $\HHH(\rho,x)\leq 0$ for every $0<\rho<\tau(x)$. 
Furthermore, if the ends of $M$ are $f$-complete and $\tau(x)=+\infty$, then $\HHH(\rho,x)=0$ for all $\rho\geq 0$.
\end{itemize}
\end{proposition}

In the following we will focus on the case where the hypersurface $\Sigma$ is homologous to $\pa M$ and strictly mean-convex, meaning that $\Sigma$ has positive mean curvature $\HHH$ with respect to the normal pointing outside $\Omega$. In this case, Proposition~\ref{pro:H_bound_Sigma}-$(i)$ tells us that the bound $\HHH/f\leq (n-1)/\eta$ is in place on the whole $U=M\setminus(\Omega\cup {\rm Cut}^{\tilde g}_\Sigma)$. Furthermore, the vector field~\eqref{eq:X}, that we recall here for convenience,
\begin{equation}
\label{eq:X_Sigma}
X\,=\,\frac{f}{\eta^{n-1}}\na\rho\,,
\end{equation}
is also well defined on $U$. We can now proceed exactly as in the proof of Theorem~\ref{thm:nonnegative_divX} to show that the vector field $X$ has nonnegative divergence in the weak sense.


\begin{theorem}
\label{thm:nonnegative_divX_Sigma}
Let $(M,g,f)$ be a substatic triple. Suppose that $M\setminus\pa M$ is geodesically complete with respect to the metric $\tilde g=g/f^2$. Let $\Sigma$ be a strictly mean-convex hypersurface homologous to $\partial M$ and disjoint from it. Suppose that the mean curvature $\HHH$ of $\Sigma$ with respect to the normal pointing towards infinity satisfies $\HHH> 0$ pointwise. Let $\rho$ be the $\tilde g$-distance function from $\Sigma$ and $\eta$ be the solution to~\eqref{eq:eta_Sigma}. Then the vector $X$ defined in~\eqref{eq:X_Sigma} has nonpositive divergence in the weak sense in the whole $M\setminus\Omega$. Namely, for every nonnegative test function $\chi\in\mathscr{C}_c^\infty(M\setminus\Omega)$ it holds
$$
\int_{M\setminus\Omega}\langle X\,|\,\na\chi\rangle d\sigma\,\geq\,0\,.
$$
\end{theorem}

\subsection{Growth of weighted areas and volumes}

In this subsection, we exploit the monotonicity of the mean curvature of the level sets to deduce a Bishop--Gromov-type theorem for the behaviour of areas and volumes of geodesic spheres. We first study the monotonicity of the following functional
\begin{equation}
\label{eq:area_functional_explicit}
A(t)\,=\,\frac{1}{|\Sph^{n-1}|}\int_{\{\rho=t\}\setminus {\rm Cut}^{\tilde g}}\frac{1}{\eta^{n-1}}d\sigma
\,,
\end{equation}
where $\rho$ is the distance function from a point or the signed distance from a strictly mean-convex hypersurface homologous to the boundary, with respect to the metric $\tilde g$, whereas ${\rm Cut}^{\tilde g}$ is the cut locus of the point/hypersurface with respect to $\tilde g$. It is important that we remove the cut locus from the domain of the integral, as we have observed in Remark~\ref{rmk:eta_nonregularoncutlocus} that the function $\eta$ is not well defined on it. 
When $\rho$ is the distance from a point, the function $A$ can be written in polar coordinates as
\begin{equation}
\label{eq:rigorous_A}
A(t)\,=\,\frac{1}{|\Sph^{n-1}|}\int_{\{\theta\in\Sph^{n-1}\,:\,\tau(\theta)> t\}}\frac{\sqrt{g}(t,\theta)}{\eta^{n-1}(t,\theta)}d\theta\,,
\end{equation}
where we recall that $\tau(\theta)$ is the minimum value of $\rho$ such that the point with coordinate $(\rho,\theta)$ belongs to the cut locus. An analogous definition can of course be given for the distance from an hypersurface using coordinates~\eqref{eq:simil-polar_coordinates}. Notice that $\tau$ is a continuous function, hence the domain of the integral in~\eqref{eq:rigorous_A} is measurable, meaning that the function $A$ is well defined for all $t\in(0,+\infty)$.
The domain of the integral shrinks as $t$ increases, whereas the integrand is positive and continuous, hence it is easily seen that for all values $a\in(0,+\infty)$ it holds 
\begin{equation}
\label{eq:punctual_monotonicity}
\liminf_{t\to a^-} A(t)\,\geq\, A(a)\,\geq\,\limsup_{t\to a^+} A(t)\,.
\end{equation}
Furthermore, notice that if the cut locus intersects $\{\rho=a\}$ in a set with positive $\mathscr{H}^{n-1}$-measure, then the first inequality is strict, that is $\liminf_{t\to a^-} A(t)\,\geq\, A(a)$. We are finally ready to state the first main result of this subsection.

\begin{theorem}
\label{thm:BG_nb}
Let $(M,g,f)$ be a substatic triple. Suppose that $M\setminus\pa M$ is geodesically complete with respect to the metric $\tilde g=g/f^2$. Let $\rho$ be the $\tilde g$-distance function from a point or the signed $\tilde g$-distance function from a strictly mean-convex hypersurface $\Sigma$ homologous to $\partial M$ and disjoint from it.
Let $\eta$ be the corresponding reparametrized distance, defined by \cref{eq:eta} or by \cref{eq:eta_Sigma}, and let ${\rm Cut}^{\tilde g}$ be the cut locus of the point/hypersurface.
Then the function
$$
A(t)\,=\,\frac{1}{|\Sph^{n-1}|}\int_{\{\rho=t\}\setminus {\rm Cut}^{\tilde g}}\frac{1}{\eta^{n-1}}d\sigma
$$
is monotonically nonincreasing.

Furthermore, if $A(t_1)=A(t_2)$ for any $0<t_1< t_2$, then the set $U=\{t_1\leq\rho\leq t_2\}$ is isometric to $[t_1,t_2]\times\Sigma$ with metric 
$$
g\,=\,f^2 d\rho\otimes d\rho+\eta^2g_0\,,
$$
where $g_0$ is a metric on the level set $\Sigma$. In $U$ the functions $f$ and $\eta$ satisfy
\begin{equation}
\label{eq:rigidity_A}
\frac{1}{\eta}\frac{\pa\eta}{\pa\theta^i}\,=\,\psi\eta-\ffi\frac{1}{\eta^{n-1}}\,,\qquad
\frac{1}{f}\frac{\pa f}{\pa\theta^i}\,=\,\psi\eta+\frac{n-2}{2}\ffi\frac{1}{\eta^{n-1}}\,,
\end{equation}
where $\ffi,\psi$ are independent of $\rho$.
\end{theorem}

\begin{remark}
\label{rem:analogy_with_standard_BG}
If we set $f\equiv 1$ then $\rho$ is the distance function with respect to $g$ and $\eta=\rho$, hence $A(t)=|\{\rho= t\}\setminus {\rm Cut}^{\tilde g}|/(|\Sph^{n-1}|t^{n-1})$ and the above monotonicity becomes completely analogous to the standard Bishop--Gromov monotonicity of the areas of geodesic spheres when $\Ric\geq 0$, which concerns the function $|\pa\{\rho\leq t\}|/(|\Sph^{n-1}|t^{n-1})$. Clearly, the two functions coincide almost everywhere, except at the values $t$ such that $\{\rho=t\}\cap {\rm Cut}^{\tilde g}$ has nonzero $\mathscr{H}^{n-1}$-measure. Notice that the number of values for which this may happen is necessarily countable. A way to see this is to observe that, as mentioned below formula~\eqref{eq:punctual_monotonicity}, every value such that $\{\rho=t\}\cap {\rm Cut}^{\tilde g}$ has nonzero $\mathscr{H}^{n-1}$-measure must correspond to a jump of $A$, and these are at most countable since $A$ has bounded variation (this is shown in the proof below). 
\end{remark}

\begin{proof}
If the cut locus were empty, the proof of the monotonicity of $A$ would follow easily by integrating the inequality ${\rm div}\,X\leq 0$ between two level sets of $\rho$, where $X$ is the vector field defined in~\eqref{eq:X}, and then applying the Divergence Theorem. In the general case, in order to take into account the lack of smoothness of $\rho$ at the cut locus, we will need a more refined analysis, that we now discuss.

Since $\rho$ is a $\tilde g$-distance function, in particular it is locally Lipschitz, hence its gradient is well defined almost everywhere. Furthermore, as highlighted in~\cite[Proposition~2.1]{benatti2022minkowski}, $\rho$ being Lipschitz also implies that the coarea formula can be applied to the level sets of $\rho$. In particular, for any $0<a<b<+\infty$ we have
\[
\int_{a}^{b}A(t)dt\,=\,\int_{a}^{b}\left(\int_{\{\rho=t\}\setminus {\rm Cut}^{\tilde g}}\frac{1}{\eta^{n-1}}d\sigma\right)dt\,=\,\int_{\{a<\rho<b\}}\frac{1}{f\eta^{n-1}}d\mu\,<\,+\infty\,.
\]
In the last integral we did not have to specify that we are not integrating on ${\rm Cut}^{\tilde g}$, since ${\rm Cut}^{\tilde g}$ is negligible when integrating on a volume.
The above tells us that $A$ is locally integrable. 
Consider now a test function $
\chi\in\mathscr{C}^\infty_c((0,+\infty))$ and let $X$ be the vector field defined by~\eqref{eq:X}. We then compute 
\begin{align}
\int_M \langle\na(\chi\circ\rho)\,|\,\,X\rangle\,d\mu&=\,\int_M\chi'(\rho)\langle X\,|\,\na\rho\rangle d\mu
\\
&=\,\int_M\chi'(\rho)\frac{f}{\eta^{n-1}}|\na\rho|^2 d\mu
\\
&=\,\int_0^{+\infty}\int_{\{\rho=t\}}\chi'(t)\frac{f}{\eta^{n-1}}|\na\rho|\, d\sigma\, dt
\\
\label{eq:Achi_byparts}
&=\,|\Sph^{n-1}|\int_0^{+\infty}\chi'(t)A(t)\, dt\,.
\end{align}
On the other hand, Theorem~\ref{thm:nonnegative_divX} (when $\rho$ is the distance from a point) and Theorem~\ref{thm:nonnegative_divX_Sigma}  tells us that the first integral in the above chain of identities is nonnegative whenever the test function $\chi$ is nonnegative. More precisely, from~\eqref{eq:weak_divergence}, since $\chi(0)=0$, we have
\begin{multline}
\label{eq:weak_divergence_aux}
\int_{M}\langle \na(\chi\circ\rho)\,|\,X\rangle d\mu\,
=-\int_{\Sph^{n-1}}\!\int_0^{\tau(\theta)}\!\!\!\chi(\rho)\,\frac{\pa}{\pa\rho}\!\left(\frac{\sqrt{g}}{f\eta^{n-1}}\right) d\rho d\theta
\\
\,+\int_{\{\theta\in\Sph^{n-1}:\,\tau(\theta)<\infty\}}\!\left[\chi\frac{\sqrt{g}}{f\eta^{n-1}}\right]\!\!(\tau(\theta),\theta)d\theta\,\geq\,0\,,
\end{multline}
where as usual $\tau(\theta)$ is the minimum value of $\rho$ such that the point with coordinate $(\rho,\theta)$ belongs to the cut locus. Combining~\eqref{eq:weak_divergence_aux} with the above chain of identities we have obtained $\int_0^{+\infty}\chi'(t)A(t)dt\geq 0$ for any nonnegative test function. If we knew $A$ to be weakly differentiable, this would force its weak derivative to be nonpositive thus proving that $A$ is nonincreasing. However, we have no information on the regularity of the function $A$ at the moment. In the following we will show that $A$ has bounded variation, which will be enough to infer the desired monotonicity.

If we fix $0<a<b<+\infty$, for any $\chi\in\mathscr{C}^\infty([a,b])$ with $\|\chi\|_\infty=1$ we have from~\eqref{eq:weak_divergence_aux} the following bound
\begin{align}
\int_{M}\langle \na(\chi\circ\rho)\,|\,X\rangle d\mu\,
&\leq-\int_{\{\theta\in\Sph^{n-1}:\,\tau(\theta)>a\}}\!\int_a^{\min\{\tau(\theta),b\}}\!\!\frac{\pa}{\pa\rho}\!\left(\frac{\sqrt{g}}{f\eta^{n-1}}\right) d\rho d\theta
\\
&\qquad\qquad\qquad\qquad\qquad\qquad\qquad\qquad\,+\int_{\{\theta\in\Sph^{n-1}:\,a\leq\tau(\theta)\leq b\}}\!\left[\frac{\sqrt{g}}{f\eta^{n-1}}\right]\!\!(\tau(\theta),\theta)\,d\theta
\\
&=\,\int_{\{\theta\in\Sph^{n-1}:\,\tau(\theta)>a\}}\left[\frac{\sqrt{g}}{f\eta^{n-1}}\right]\!\!(a,\theta)\, d\theta\,-\int_{\{\theta\in\Sph^{n-1}:\,\tau(\theta)>b\}}\left[\frac{\sqrt{g}}{f\eta^{n-1}}\right]\!\!(b,\theta)\, d\theta
\,.
\end{align}
In other words, the quantity $\int_{M}\langle \na(\chi\circ\rho)\,|\,X\rangle d\mu$, and thus also $\int_0^{+\infty}\chi'(t)A(t)\, dt$, is bounded from above by a constant that depends on $a$, $b$, $f$ and $g$, but not on $\chi$. It follows that $A$ has bounded variation in $[a,b]$. As a consequence, the signed finite Radon measure $\mu$ on $(a,b)$ defined by $\mu((c,d))=\lim_{t\to d^-}A(t)-\lim_{t\to c^+}A(t)$ for any $a<c<d<b$, is such that
\[
\int_0^{+\infty}\chi'(t)A(t)\, dt\,=\,-\int_0^{+\infty}\chi(t) \mu(dt)\,.
\]
Since we have already shown that $\int_0^{+\infty}\chi'(t)A(t)\, dt\geq 0$ for any nonnegative test function $\chi$, it follows that the measure $\mu$ is nonpositive. From the definition of $\mu$ and~\eqref{eq:punctual_monotonicity}, we deduce then that the function $A$ is monotonically nonincreasing in $(a,b)$.
Since this should hold for any $0<a<b<+\infty$, it must necessarily hold on the whole $(0,+\infty)$. This proves that $A$ is monotonically nonincreasing.

It remains to prove the rigidity statement. 
If $A(t_1)=A(t_2)$, then thanks to the discussion above it follows $A(t)=A(t_1)$ for all $t_1<t<t_2$. As a consequence, for any test function $\chi$ supported in $[t_1,t_2]$, we get that the last line in the computation~\eqref{eq:Achi_byparts} vanishes, that is, $\int_M\langle\na(\chi\circ\rho)\,|\,\,X\rangle\,d\mu=0$. On the other hand, this integral can also be computed as in~\eqref{eq:weak_divergence_aux}. 
From the fact that $\chi\geq 0$ and from~\eqref{eq:riccati_for_sqrtg}, we know that the two terms on the right hand side of~\eqref{eq:weak_divergence_aux} are both nonnegative, hence they must both vanish for all $\chi\in\mathscr{C}^\infty_c([0,+\infty))$. This implies that $\tau(\theta)$ never belongs to $(t_1,t_2)$, meaning that the cut locus does not intersect $\{t_1<\rho <t_2\}$  and that equality is achieved in~\eqref{eq:riccati_for_sqrtg}. In other words, the following holds
\begin{equation}
\label{eq:H/f_eta}
\frac{\HHH}{f}\,=\,\De\rho+\frac{1}{f}\langle\na f\,|\,\na\rho\rangle\,=\,\frac{n-1}{\eta}\qquad \hbox{ in }\{t_1\leq \rho\leq t_2\}\,.
\end{equation}
This identity in turn triggers the equality in the estimates made in Subsection~\ref{sub:evolution}, namely
\begin{equation}
\label{eq:rigidity}
\big|\widetilde{\na}^2\rho\big|^2_{\tilde g}\,=\,\frac{(\De_{\tilde g} \rho)^2}{n-1}\,,\qquad \Ric\big(\widetilde\na\rho,\widetilde\na\rho\big)\,=\,(n-1)\left[\frac{1}{f}\widetilde\na^2 f\big(\widetilde\na\rho,\widetilde\na\rho\big)-\frac{2}{f^2}\big\langle\widetilde\na f\,|\,\widetilde\na \rho\big\rangle^2\right]\,,
\end{equation}
where we recall that $\widetilde\na$ is the Levi-Civita connection with respect to $\tilde g$. Notice that, since $|\widetilde\na\rho|_{\tilde g}=1$, for any vector $X$ it holds
$$
\widetilde\na^2\rho\big(\widetilde\na\rho,X\big)\,=\,\big\langle\widetilde\na|\widetilde\na\rho|_{\tilde g}^2\,|\,X\big\rangle_{\tilde g}\,=\,0\,.
$$
It follows immediately from this and the first equation in~\eqref{eq:rigidity} that, in the coordinates in which $\tilde g$ has the form~\eqref{eq:gamma}, for any $i,j=1,\dots,n-1$ it holds 
$$
\widetilde\na^2_{ij}\rho\,=\,\frac{\De_{\tilde g}\rho}{n-1}\tilde g_{ij}\,=\,\left(\frac{f^2}{\eta}-\frac{1}{f}\frac{\pa f}{\pa\rho}\right)\tilde g_{ij}\,.
$$
where the latter identity makes use of~\eqref{eq:useful_formulas}.
On the other hand, from the definition of Hessian we have $\widetilde\na^2_{ij}\rho=-\Gamma_{ij}^\rho=\pa_\rho \tilde g_{ij}/2$, hence
$$
\frac{\pa\tilde g_{ij}}{\pa\rho}\,=\,2\left(\frac{f^2}{\eta}-\frac{1}{f}\frac{\pa f}{\pa\rho}\right)\tilde g_{ij}\,=\,2\frac{\pa}{\pa\rho}\left(\log\eta-\log f\right)\tilde g_{ij}\,.
$$
This identity can be solved explicitly, yielding
$$
\tilde g_{ij}\,=\,\frac{\eta^2}{f^2} (g_0)_{ij}\,,
$$
where $(g_0)_{ij}$ does not depend on $\rho$. Comparing with~\eqref{eq:gamma} and recalling $g=f^2\tilde g$, we have obtained
\begin{equation}
\label{eq:rigidity_1}
g\,=\,f^2 d\rho\otimes d\rho+\eta^2g_0\,.
\end{equation}
The functions $f$ and $\eta$ may be functions of both the radial coordinate $\rho$ and of $\{\theta^1,\dots,\theta^{n-1}\}$. Any metric $g$ having the form~\eqref{eq:rigidity_1} satisfies the substatic condition with equality in the radial direction, that is:
$$
f\RRR_{\rho\rho}\,-\,\nana_{\rho\rho} f\,+\,(\De f)g_{\rho\rho}\,=\,0\,.
$$
From this identity and the substatic condition we find out that, for any vector $X=\pa/\pa\rho+\lambda\pa/\pa_i$, it holds
\begin{align}
 0\,&\leq\,\left[f\Ric\,-\,\nana f-(\De f)g\right](X,X)
\\
&=\,\lambda^2\,\left[f\RRR_{ii}\,-\,\nana_{ii} f\,+\,(\De f)g_{ii}\right]
\,+\,2\lambda\,
\left[f\RRR_{i\rho}\,-\,\nana_{i\rho} f\,+\,(\De f)g_{i\rho}\right]\,.
\end{align}
Since this inequality holds pointwise for any $\lambda\in\R$, it follows that \begin{equation}
\label{eq:rigidity_2}
f\RRR_{i\rho}\,-\,\nana_{i\rho} f\,+\,(\De f)g_{i\rho}\,=\,0\,.
\end{equation}
Recalling the expression~\eqref{eq:rigidity_1} for the metric, a direct computation gives us that~\eqref{eq:rigidity_2} is equivalent to
\begin{equation}
\label{eq:PDE_FG_1}
(n-2)\pa^2_{i\rho}G\,+\,\pa^2_{i\rho}F-(n-1)\pa_\rho G\pa_i F\,=\,0\,,
\end{equation}
where $F=\log f$, $G=\log \eta$. On the other hand, since $\pa_\rho \eta=f^2$, we have $\pa_\rho G=e^{2F-G}$, from which we compute
\begin{equation}
\label{eq:PDE_FG_2}
\pa_{i\rho}^2 G\,-\,2\pa_i F\pa_\rho G \,+\,\pa_i G \pa_\rho G\,=\,0\,.
\end{equation}
We will now combine~\eqref{eq:PDE_FG_1} and~\eqref{eq:PDE_FG_2} in two different ways. 

On the one hand, if we subtract $(n-1)$ times equation~\eqref{eq:PDE_FG_2} from equation~\eqref{eq:PDE_FG_1}, we obtain
$$
\pa^2_{i\rho}F\,+\,(n-1)\pa_i F\pa_\rho G\,=\,\pa^2_{i\rho}G\,+\,(n-1)\pa_\rho G\pa_i G\,,
$$
which can be rewritten as
$$
\pa_\rho\left(e^{(n-1)G}\pa_i F\right)\,=\,\pa_\rho\left(e^{(n-1)G}\pa_i G\right)\,.
$$
In other words, we have found
\begin{equation}
\label{eq:pde1}
e^{(n-1)G}(\pa_i F-\pa_i G)=\ffi(\theta)\,,\ \ \forall i=1,\dots, n-1\,.
\end{equation}

On the other hand, subtracting $(n/2-1)$ times equation~\eqref{eq:PDE_FG_2} from~\eqref{eq:PDE_FG_1}, we obtain
$$
\pa^2_{i\rho}F\,-\,\pa_i F\pa_\rho G\,=\,-\frac{n-2}{2}\pa^2_{i\rho}G\,+\,\frac{n-2}{2}\pa_\rho G\pa_i G\,,
$$
which can be rewritten as
$$
\pa_\rho\left(e^{-G}\pa_i F\right)\,=\,-\frac{n-2}{2}\pa_\rho\left(e^{-G}\pa_i G\right)\,,
$$
that gives
\begin{equation}
\label{eq:pde2}
e^{-G}\left(\pa_i F-\frac{n-2}{2}\pa_i G\right)=\psi(\theta)\,,\ \ \forall i=1,\dots, n-1\,.
\end{equation}
Writing~\eqref{eq:pde1} and~\eqref{eq:pde2} in terms of $\eta$ and $f$, with straightforward computations we obtain formulas~\eqref{eq:rigidity_A}.
\end{proof}

An immediate consequence of the above result is that we can find a bound for the area functional $A(t)$ by taking its limit as $t\to 0$. In the case where $\rho$ is the $\tilde g$-distance from a hypersurface, then $A(0)$ is in fact well defined and we have $A(t)\leq A(0)$. This will be exploited later, in Subsection~\ref{sub:Willmore} to prove Theorem~\ref{thm:Willmore_intro}.

We focus now briefly on the case where $\rho$ is the $\tilde g$-distance from a point $x$. As $\rho\to 0$ we have $\pa\eta/\pa\rho= f(x)^2+o(1)$, hence
$\eta= f(x)^2\rho+o(\rho)$. Furthermore, $d\sigma= (1+o(1))f(x)^{n-1}d\sigma_{\tilde g}=(1+o(1)) f(x)^{n-1}\rho^{n-1}d\sigma_{\Sph^{n-1}}$. It follows that
$$
\lim_{t\to 0^+} A(t)\,=\,\frac{1}{f(x)^{n-1}}\,.
$$
As a consequence of the monotonicity of $A$ we then deduce that, for every $t>0$, it holds
\begin{equation}
\label{eq:BG_effective}
A(t)\,\leq\,\frac{1}{f(x)^{n-1}}\,.
\end{equation}
If the equality holds, then the rigidity statement in \cref{thm:BG_nb} applies in $\{\rho\leq t\}$.

Building on Theorem~\ref{thm:BG_nb}, we can also show the following volumetric version of the Bishop--Gromov monotonicity theorem.

\begin{theorem}
\label{thm:BGvolumi_nb}
Let $(M,g,f)$ be a substatic triple. Suppose that $M\setminus\pa M$ is geodesically complete with respect to the metric $\tilde g=g/f^2$.
Let $\rho$ be the $\tilde g$-distance function from a point or the signed $\tilde g$-distance function from a strictly mean-convex hypersurface $\Sigma$ homologous to $\partial M$ and disjoint from it.
Let $\eta$ be the corresponding reparametrized distance, defined by \cref{eq:eta} or by \cref{eq:eta_Sigma}. Then, for any $k>0$, the function
\begin{equation}
\label{eq:V_k}
V(t)\,=\,\frac{1}{|\mathbb{B}^n|t^k}\int_{\{0\leq\rho\leq t\}}\frac{\rho^{k-1}}{f\eta^{n-1}}d\mu
\end{equation}
is well defined and monotonically nonincreasing.

Furthermore, if $V(t_1)=V(t_2)$ for $0<t_1< t_2$, then the set $U=\{0\leq\rho\leq t_2\}$ is isometric to $[0,t_2]\times\Sigma$ with metric
$$
g\,=\,f^2 d\rho\otimes d\rho+\eta^2 g_0\,,
$$
where $g_0$ is a metric on the level sets $\Sigma$.
In $U$ the functions $f$ and $\eta$ satisfy
\begin{equation}
\label{eq:rigidity_B}
\frac{1}{\eta}\frac{\pa\eta}{\pa\theta^i}\,=\,\psi\eta-\ffi\frac{1}{\eta^{n-1}}\,,\qquad
\frac{1}{f}\frac{\pa f}{\pa\theta^i}\,=\,\psi\eta+\frac{n-2}{2}\ffi\frac{1}{\eta^{n-1}}\,,
\end{equation}
where $\ffi,\psi$ are independent of $\rho$.
If $\rho$ is the distance from a point $x$, then $g_0=f(x)^{-2}g_{\Sph^{n-1}}$ and $\ffi=\psi=0$ in the whole $U$, that is, both $f$ and $\eta$ are functions of $\rho$ in $U$.
\end{theorem}

\begin{remark}
Recall that $\eta$ is smooth outside the cut locus. Since the cut locus has finite $\mathscr{H}^{n-1}$-measure, the functional $V$ in the statement is well posed.
When $k=n$ we have recovered the standard Bishop--Gromov monotonicity for volumes of geodesic spheres for $\Ric\geq 0$. Indeed if one sets $f=1$ in the statement above, so that $\eta=\rho$, one gets $V(t)=|\{\rho\leq t\}|/(|\mathbb{B}^n|t^n)$.
\end{remark}

\begin{proof}
We start by observing that the coarea formula (together with the fact that $|\na\rho|=1/f$) gives the following relation between the functionals $A$ and $V$:
\begin{equation}
\label{eq:BG_integral_V}
V(t)\,=\,\frac{n}{|\Sph^{n-1}|t^k}
\int_0^t\left(\int_{\{\rho=\tau\}}\frac{\tau^{k-1}}{\eta^{n-1}}d\sigma\right)d\tau\,=\,\frac{n}{t^k}\int_0^t\tau^{k-1}A(\tau)d\tau\,.
\end{equation}
From Theorem~\ref{thm:BG_nb}, we know that for almost every $\tau\leq t$ the area integral $A(\tau)$ is well defined and that it is nonincreasing in $\tau$. In the case where $\rho$ is the distance from a point $x$, we have observed in~\eqref{eq:BG_effective} that $A(\tau)$ is bounded by the constant $1/f(x)^{n-1}$. If instead $\rho$ is the distance from a strictly mean-convex hypersurface, then $A(0)$ is well defined and we have $A(\tau)\leq A(0)$. In both cases, it holds $A(\tau)\leq \mathrm{C}$ for some constant $\mathrm{C}$, hence from~\eqref{eq:BG_integral_V}, recalling $k>0$, we compute
\[
V(t)\,\leq\, \frac{\mathrm{C}}{t^k}\int_0^t\tau^{k-1}d\tau\,=\,\frac{\mathrm{C}}{k}\,.
\]
As a consequence, $V(t)$ is well defined.
Furthermore, the monotonicity of $A$ also implies
\(
A(t)\leq A(\tau).
\)
Plugging this information in~\eqref{eq:BG_integral_V} gives
\begin{equation}
V(t)\,\geq\,\frac{n}{t^k}A(t)\int_0^t\tau^{k-1}d\tau\,=\,\frac{n}{k}A(t)\,.
\end{equation}
With this information at hand, we are ready to compute the derivative of $V(t)$:

\begin{align}
V'(t)\,&=\,\lim_{\ep\to 0}[V(t+\ep)-V(t)]/\ep
\\
&=\,\lim_{\ep\to 0}\frac{1}{\ep|\mathbb{B}^n|}\left(\frac{1}{(t+\ep)^k}-\frac{1}{t^k}\right)\int_{\{\rho\leq t+\ep\}}\frac{\rho^{k-1}}{f\eta^{n-1}}d\mu\,+\,\lim_{\ep\to 0}\frac{1}{t^k}\frac{1}{\ep|\mathbb{B}^n|}\int_{\{t\leq\rho\leq t+\ep\}}\frac{\rho^{k-1}}{f\eta^{n-1}}d\mu
\\
&=\,\lim_{\ep\to 0}\frac{t^k-(t+\ep)^k}{\ep t^k}V(t+\ep)\,+\,\lim_{\ep\to 0}\frac{n}{t^k}\frac{1}{\ep}\int_t^{t+\ep}\tau^{k-1}A(\tau) d\tau
\\
&=\,-\frac{k}{t}V(t)+\frac{n}{t}A(t)
\\
&\leq\,0\,.
\end{align}

We now prove the rigidity statement. If $V(t_1)=V(t_2)$ for two values $0<t_1<t_2$, then retracing our computations we find out that $A(\tau)=A(t)$ for all $0<\tau<t_2$.
From the rigidity statement of Theorem~\ref{thm:BG_nb} we then deduce that in $\{0\leq\rho\leq t\}$ the metric writes as 
\[
g\,=\,f^2d\rho\otimes d\rho+\eta^2 g_0\,,
\]
and $f,\eta$ satisfy formulas~\eqref{eq:rigidity_B}.

Finally, we suppose now that $\rho$ is the distance from a point $x$ and we prove that $f,\eta$ must necessarily depend on $\rho$ only. 
To this end, notice that formulas~\eqref{eq:rigidity_B} must hold up to $\rho=0$ (that is, up to the point $x$), hence at the limit as $\rho$ goes to zero, the derivative $\pa f/\pa\theta^i$ goes to zero. Since $\eta$ goes to $0$ as $\rho\to 0$, it follows then from the second formula in~\eqref{eq:rigidity_B} that $\ffi$ must vanish identically. 

As a consequence, the first formula in~\eqref{eq:rigidity_B} can be rewritten as
\begin{equation}
\label{eq:rigidity_eta}
\frac{1}{\eta^2}\frac{\pa\eta}{\pa\theta^i}\,=\,\psi\,.  
\end{equation}

In particular, taking the derivative with respect to $\rho$ we deduce that $\pa^2(1/\eta)/\pa\rho\pa\theta=0$.
In other words, $1/\eta=\alpha+\beta$, where $\alpha$ is a function of $\rho$ and $\beta$ is a function of the $\theta^i$'s. Substituting this expression for $\eta$ in~\eqref{eq:rigidity_eta}, we deduce that 
\[
\frac{\pa\beta}{\pa\theta^i}\,=\,-\psi
\]

On the other hand, taking the difference of the two formulas in~\eqref{eq:rigidity_B} we have
$$
\frac{\pa}{\pa\theta^i}\left(\log\eta-\log f\right)\,=\,0\,.
$$
In other words, the quantity $\eta/f$ must be a function of $\rho$. Recalling the decomposition $1/\eta=\alpha+\beta$ shown right above, it follows then that $1/f=\lambda(\alpha+\beta)$,
where $\lambda$ is a function of $\rho$. 

Notice now that, when $\rho$ goes to zero, the limit of $f$ must go to the value of $f$ at the point $x$, so that in particular the limit of $1/f$ as $\rho\to 0$ must not depend on $\theta^i$. It follows that $\beta=0$, hence $\psi=-\pa\beta/\pa\theta^i$ vanishes as well.
We have proved that both $\ffi$ and $\psi$ in~\eqref{eq:rigidity_B} vanish, hence both $f$ and $\eta$ must be functions of the sole $\rho$ in the whole $\{\rho\leq t\}$. Since the metric 
\[
\tilde g=d\rho\otimes d\rho+\frac{\eta^2}{f^2}g_0
\]
is smooth at the point $x$, it follows that $(\eta^2/f^2)g_0$ should be close to $\rho^2 g_{\Sph^{n-1}}$ near $x$. From the definition of $\eta$ it follows $\eta=f(x)^2\rho+o(\rho)$ close to $x$, hence $g_0=f(x)^{-2}g_{\Sph^{n-1}}$
and we conclude the rigidity statement.
\end{proof}

\section{Wylie's Splitting Theorem for substatic manifolds}
\label{sec:splitting}

\subsection{$f$-complete and conformally compact ends}
From now on we will study noncompact manifolds with some special behaviour at infinity, focusing mainly on $f$-complete ends.

\begin{definition}
\label{def:f-complete}
We say that an end is {\em $f$-complete} if for any $g$-unit speed curve $\gamma:[0,+\infty)\to M$ going to infinity along the end it holds
\begin{equation}
\label{eq:f-completeness}
\lim_{t\to +\infty}\rho(\gamma(t))\,=\,+\infty\,,\qquad\int_0^{+\infty}f(\gamma(t)) dt\,=\,+\infty\,,
\end{equation}
where $\rho$ is the distance from a point with respect to $\tilde g=g/f^2$.
\end{definition}

It is clear from the triangle inequality that the definition above does not depend on the point we are taking the distance $\rho$ from. It also would not change if we replace the distance from a point with the distance from a hypersurface.

For all the arguments that follows it would actually be enough to require~\eqref{eq:f-completeness} only along $\tilde g$-geodesics. More precisely, it is enough to require the end to be $\tilde g$-complete and satisfying the second condition in~\eqref{eq:f-completeness} along any $\tilde g$-geodesic. In fact, the above definition is analogous to the one given in~\cite[Definition~6.2]{Wylie} in the ${\rm CD}(0,1)$ framework: there, a triple $(M,\tilde g,\psi)$ satisfying the ${\rm CD}(0,1)$ condition is said to be $\psi$-complete if for any $\tilde g$-geodesic $\sigma:[0,+\infty)\to M$ going to infinity along the end it holds 
\[
\int_0^{+\infty}e^{-\frac{2\psi(\sigma(t))}{n-1}} dt\,=\,+\infty\,.
\]
Recalling the relations $\tilde g=g/f^2$ and $\psi=-(n-1)\log f$ between the ${\rm CD}(0,1)$ and substatic setting (see \cref{app:CD_subst}), it is easily seen that this integrability condition is equivalent to the second requirement in~\eqref{eq:f-completeness}. As already observed in~\cite{Wylie}, this integrability condition can be interpreted as completeness with respect to the metric $f^2 g$ or, alternatively, as completeness with respect to the weighted connection introduced in~\cite{wylie-yeroshkin} and~\cite{Li_Xia_17} (see \cref{app:LiXia}).
For what concerns this paper however, the only relevance of the second condition in~\eqref{eq:f-completeness} is that it implies that the reparametrized distance $\eta$ defined in Section~\ref{sec:Riccati} goes to infinity along the end. This is easy to show as follows. Let $\rho$ be the $\tilde g$-distance to a point or hypersurface and let $\eta$ be defined by~\eqref{eq:eta} or~\eqref{eq:eta_Sigma} depending on whether we are taking the distance from a point or hypersurface. Let $\sigma:[0,+\infty)\to M$ be a $\tilde g$-geodesic with $\dot\sigma=\widetilde\na\rho$ and let $\gamma:[0,+\infty)\to M$ be the reparametrization of $\sigma$ that has $g$-length constant and equal to $1$. We then have
\[
\eta(\gamma(t))-\eta(\gamma(0))\,=\,
\int_0^{t}f^2(\gamma(t))|\dot\gamma(t)|_{\tilde g}dt\,=\,
\int_0^{t}f(\gamma(t))|\dot\gamma(t)|_g dt\,=\,
\int_0^{t}f(\gamma(t)) dt\,,
\]
hence if the second condition in~\eqref{eq:f-completeness} holds then $\eta$ goes to $+\infty$.

The family of $f$-complete ends includes a number of interesting examples. Most notably, {\em asymptotically flat} ends are $f$-complete. We say that $(M,g,f)$ is asymptotically flat if there exists a compact set $K$ such that $M\setminus K$ is diffeomorphic to $\R^n$ minus a ball, the metric $g$ converges to the Euclidean metric and $f$ goes to $1$ at infinity along the end. A precise definition of asymptotic flatness is given below, see Definition~\ref{def:AF}. A notable example of asymptotically flat substatic solution is the Reissner--Nordstr\"om solution, corresponding to~\eqref{eq:models} with $\Lambda=0$. In fact, the family of $f$-complete ends is quite more general: for instance, it is sufficient to require a suitable behaviour of $f$ at infinity, without any assumption on the topology and geometry of the end, as clarified by the following proposition. 
\begin{proposition}
\label{pro:fcomplete_AF}
Let $(M,g,f)$ be a substatic triple and let $r$ be the $g$-distance from a point (or more generally from a compact domain). If  there exist a compact set $K\supset\pa M$ and constants $0<c<C$, $0<k<1$ such that 
\begin{equation}
\label{eq:f_pinch}
cr^{-k}<f<Cr^{k}
\end{equation}
at all points in $M\setminus K$, then all ends are $f$-complete.
\end{proposition}

\begin{proof}
Let $\gamma:[0,+\infty)\to M$ be a $g$-unit speed curve going to infinity along the end and let $\delta$ be the $g$-distance between $\gamma(0)$ and the point $p$ we are taking the distance $r$ from (if $r$ is the distance from a compact domain instead, it is sufficient to choose $\delta$ as the maximum distance between $\gamma(0)$ and the points of the domain; the rest of the proof is easy to adapt). It is also convenient to assume that $\gamma(0)$ and $p$ belong to $K$  and that $K$ is geodesically convex with respect to the metric $\tilde g=g/f^2$ (this can of course always be achieved by possibly enlarging $K$).

 Since $\gamma$ has unit speed, we have $d(\gamma(0),\gamma(t))<t$, hence by triangle inequality
\[
r(\gamma(t))<t+\delta\,.
\]
If we then denote by $T$ the maximum value of $t$ such that $\gamma(t)\in K$, estimate~\eqref{eq:f_pinch} tells us that for any $t>T$ it holds
\[
c\,(t+\delta)^{-k}\,<\,f(\gamma(t))\,<\,C\,(t+\delta)^{k}\,.
\]
In particular, since $0<k<1$ we have
\[
\int_{T}^{+\infty}f(\gamma(t)) dt\,>\,c\,\int_{T}^{+\infty}(t+\delta)^{-k}dt
\,>\,c\lim_{t\to+\infty}\frac{(t+\delta)^{1-k}-(T+\delta)^{1-k}}{1-k}
\,=\,+\infty\,.
\]

To conclude, it is sufficient to show that the $\tilde g$-distance $\rho$ of $\gamma(t)$ from a fixed point (we will take $\gamma(0)$ for simplicity) also goes to $+\infty$ as $t\to+\infty$.
For any fixed $t>0$, let $\sigma_t$ be the unit-speed $\tilde g$-geodesic from $\gamma(0)$ to $\gamma(t)$. We reparametrize $\sigma_t$ so that it has speed $1$ with respect to the metric $g$. With a slight abuse of notation, we still denote by $\sigma_t$ the reparametrized curve. We will have $\sigma_t:[0,\tau]\to M$, with $\tau\geq t$. Since we have chosen $K$ to be $\tilde g$-geodesically convex, there will exist a value $T_t$ such that $\sigma_t(s)\in K$ for all $s\leq T_t$ and  $\sigma_t(s)\not\in K$ for all $s>T_t$.
Clearly $\tau-T_t>{\rm d}_g(K,\gamma(t))$. Furthermore, since $\sigma_t$ restricted to $[0,T_t]$ is $\tilde g$-minimizing and both $\sigma_t(0)=\gamma(0)$ and $\sigma_t(T_t)$ belong to $K$, we have
\[
{\rm diam}_{\tilde g}(K)\,>\,\int_0^{T_t} |\dot\sigma_t(s)|_{\tilde g}ds\,=\,\int_{0}^{T_t} \frac{1}{f(\sigma_t(s))}ds>\frac{T_t}{\max_K f}\,.
\]

We are now ready to estimate the $\tilde g$-distance as follows:
\[
\rho(\gamma(t))=\int_{0}^\tau|\dot\sigma_t(s)|_{\tilde g}ds\,=\,\int_{0}^\tau \frac{1}{f(\sigma_t(s))}|\dot\sigma_t(s)|_{g}ds
\,=\,\int_{0}^\tau \frac{1}{f(\sigma_t(s))}ds
\,>\,
\int_{T_t}^\tau \frac{r(\sigma_t(s))^{-k}}{\sqrt{C}}ds\,.
\]
Since $r(\sigma_t(s))<{\rm d}_g(\sigma_t(s),\gamma(0))+\delta<s+\delta$, we then have
\begin{align}
\rho(\gamma(t))&>\frac{(\tau+\delta)^{1-k}-(T_t+\delta)^{1-k}}{\sqrt{C}(1-k)}\,
\\
&>\,
\frac{(T_t+\delta+{\rm d}_g(\gamma(t),K))^{1-k}-(T_t+\delta)^{1-k}}{\sqrt{C}(1-k)}
\\
&>\,
\frac{\left({\rm d}_g(\gamma(t),K)\right)^{1-k}-\left({\rm diam}_{\tilde g}(K) \max_{K}f\right)^{1-k}}{\sqrt{C}(1-k)}\,.
\end{align}
Since $K$ is fixed, the distance ${\rm d}_g(\gamma(t),K)$ is going to $+\infty$, hence $\rho(\gamma(t))\to+\infty$ as $t\to+\infty$ as wished.
\end{proof}

Another well studied family of ends is the following:
\begin{definition}
\label{def:cc}
We say that an end of a substatic triple $(M,g,f)$ is {\em conformally compact} if a neighborhood $E$ of it is the interior of a compact manifold $\overline{E}$ with boundary $\pa\overline{E}$ and the metric $\tilde g=g/f^2$ extends to the boundary of $\overline{E}$ in $\mathscr{C}^3$-fashion. We denote by $\pa E_\infty=\pa\overline{E}\setminus M$ the {\em conformal boundary} of the end. Finally, we require $f^{-1}$ to extend to a $\mathscr{C}^3$-function on $\overline{E}$ in such a way that $f^{-1}=0$ and $d(f^{-1})\neq 0$ on $\pa E_\infty$.
\end{definition}

It is clear that on conformally compact ends the $\tilde g$-distance function $\rho$ does not grow to infinity along the end. On the other hand, it is easily seen that $\eta$ goes to infinity along any ray going into a conformally compact end. 
An example of this behaviour is given by the Schwarzschild-Anti de Sitter solution. 

The two families of $f$-complete ends and conformally compact ends enclose the model solutions  we are interested in. Furthermore, following~\cite{Wylie}, we can prove a splitting theorem for both these types of ends. The proof, given below, makes substantial use of the conformal metric $\tilde g$. It is then convenient to remark that our manifold remains complete with respect to this conformal metric.

\begin{lemma}
\label{le:completeness}
Let $(M,g,f)$ be a substatic triple with ends that are either $f$-complete or conformally compact and let $\pa M_\infty$ be the (possibly empty) conformal infinity (namely, the union of the conformal infinities of the conformally compact ends). Then the manifold $(M\cup\pa M_\infty)\setminus\pa M$ is complete with respect to the metric $\tilde g=g/f^2$.
\end{lemma}

\begin{remark}
We specify that we are referring here to completeness as a metric space, not to geodesic completeness. Clearly geodesic completeness fails in presence of a conformal boundary, since $\tilde g$-geodesics may end at $\pa M_\infty$.
\end{remark}

\begin{proof}
The fact that the manifold is complete near the conformal boundary is immediate from the definition of conformally compact ends. The $\tilde g$-completeness of $f$-complete ends has already been discussed after Definition~\ref{def:f-complete}. It remains to prove $\tilde g$-completeness near $\pa M$. To do this, we show that the boundary components become ends with respect to the conformal metric $\tilde g$. Let $\gamma:[0,\ell]\to\R$ be a $g$-unit speed curve with $\gamma(0)\in \pa M$ and $\gamma(t)$ in the interior of $M$ for any $t>0$. It is enough to show that its $\tilde g$-length is infinite. From the mean value theorem we know that for any $t\in(0,\ell]$ there exists $\xi\in(0,t)$ such that
\[
f(\gamma(t))\,=\,f(\gamma(t))-f(\gamma(0))\,=\,\langle\na f(\gamma(\xi))\,|\,\dot\gamma(t)\rangle\,t\,\leq\,|\na f|(\gamma(\xi))\,t\,.
\]
If $K$ is a compact collar neighborhood of $\pa M$ containing $\gamma$, we then compute
\[
\int_0^\ell |\dot\gamma(t)|_{\tilde g}dt\,=\,\int_0^\ell \frac{dt}{f(\gamma(t))}\,\geq\,\int_0^\ell \frac{1}{\left(\max_K|\na f|\right) t}dt=+\infty\,,
\]
as wished.
\end{proof}

\subsection{Splitting Theorem for $f$-complete ends}
\label{sub:splitting_fcomplete}

In~\cite{Wylie}, the author proves a Splitting Theorem in the ${\rm CD}(0,1)$ framework. Here we translate this result in the substatic setting, obtaining \cref{thm:splitting_intro}. The proof is essentially the one in~\cite{Wylie}, but we prefer to show it for completeness.
The strategy is that of exploiting the Laplacian Comparison given by Theorem~\ref{thm:H_bound} together with standard techniques for the Busemann function, to be coupled with a refined splitting argument.

\begin{proof}[Proof of \cref{thm:splitting_intro}]
The first part of the proof follows quite closely the usual proof of the classical Splitting Theorem for manifolds with nonpositive Ricci curvature, so we avoid to give all the technical details, that can be found in any standard Riemannian geometry book (see for instance~\cite[Theorem~7.3.5]{Petersen}).

We are assuming that there are at least two $f$-complete ends and we know  from Lemma~\ref{le:completeness} that $M\setminus\pa M$ is complete with respect to $\tilde g$, hence we can take points arbitrarily far away in the two ends and connect them via a $\tilde g$-minimizing geodesic. At the limit, we thus produce a globally minimizing geodesic $\sigma:[-\infty,+\infty]$ going from one end to the other.
For a given $t\in\R$ we then consider the functions
$$
\beta_t(x)={\rm d}_{\tilde g}(x,\sigma(t))-t
$$
and the Busemann functions
$$
\beta_+(x)=\lim_{t\to+\infty}\beta_t(x)\,,\qquad
\beta_-(x)=\lim_{t\to-\infty}\beta_t(x)\,.
$$
Under the assumption of $f$-completeness, we know that the $\tilde g$-distance goes to infinity as we approach the end, hence the limits are well defined.
Theorem~\ref{thm:H_bound} tells us that $\De\beta_t+(1/f)\langle \na f\,|\,\na\beta_t\rangle\leq 1/\eta$ for every $t$. Again using the fact that we are assuming $f$-completeness of the ends, we have $\eta\to+\infty$ at infinity. Standard arguments then tell us that $\beta_\pm$ satisfy the inequality 
$$
\De\beta_\pm+\frac{1}{f}\langle \na f\,|\,\na\beta_\pm\rangle\leq 0
$$ 
in the barrier sense. In particular, $\beta_-+\beta_+$ satisfies the same elliptic inequality. Furthermore $\beta_-+\beta_+=0$ on $\sigma$ by construction and  a simple application of the triangle inequality shows that $\beta_-+\beta_+\geq 0$ on the whole $M$. In follows then by maximum principle that $\beta_++\beta_-=0$. We then conclude that $\beta=\beta_+$ solves
\begin{equation}
\label{eq:pde_beta}
\De\beta+\frac{1}{f}\langle \na f\,|\,\na\beta\rangle\,=\, 0
\end{equation}
in the barrier sense. Standard regularity theory tells us that $\beta$ is in fact smooth, hence solves~\eqref{eq:pde_beta} in the classical sense as well. 

We denote by $\widetilde\na$ the Levi-Civita connection with respect to the metric $\tilde g$.
Our next step is to prove that $\widetilde\na\beta$ is in fact the splitting direction. We first exploit~\eqref{eq:pde_beta} to compute
$$
\De_{\tilde g}\beta\,=\,-(n-1)\frac{1}{f}\big\langle\widetilde\na f\,\big|\,\widetilde\na\beta\big\rangle_{\tilde g}
$$
Furthermore, rewriting the substatic condition~\eqref{eq:substatic} in terms of the new metric, we find
$$
\Ric_{\tilde g}\,\geq\,(n-1)\frac{1}{f}\widetilde\na^2 f-2(n-1)\frac{1}{f^2}df\otimes df
$$
The above formulas can be applied in combination with the Bochner formula to obtain
\begin{align}
\De_{\tilde g}\big|\widetilde\na\beta\big|_{\tilde g}^2\,&=\,\big|\widetilde\na^2\beta\big|_{\tilde g}^2+2\Ric_{\tilde g}\big(\widetilde\na\beta,\widetilde\na\beta\big)+2\big\langle \widetilde\na\De_{\tilde g}\beta\,\big|\,\widetilde\na\beta\big\rangle_{\tilde g}
\\
&\geq\,2\,\left[\big|\widetilde\na^2\beta\big|_{\tilde g}^2-\frac{(\De_{\tilde g}\beta)^2}{n-1}\right]-(n-1)\frac{1}{f}\big\langle\widetilde\na\big|\widetilde\na\beta\big|_{\tilde g}^2\,\big|\,\widetilde\na f\big\rangle_{\tilde g}
\end{align}
A standard estimate for the Hessian tells us that
$$
\big|\widetilde\na^2\beta\big|_{\tilde g}^2-\frac{(\De_{\tilde g}\beta)^2}{n-1}\,\geq\,-\frac{1}{n-1}\frac{\De_{\tilde g}\beta}{|\widetilde\na\beta|_{\tilde g}^2}\big\langle\widetilde\na\big|\widetilde\na\beta\big|_{\tilde g}^2\,\big|\,\widetilde\na\beta\big\rangle_{\tilde g}\,.
$$
Substituting this in the previous inequality, we get
$$
\De_{\tilde g}\big|\widetilde\na\beta\big|_{\tilde g}^2\,\geq\,
\bigg\langle\widetilde\na\big|\widetilde\na\beta\big|_{\tilde g}^2\,\bigg|\,\frac{2}{f}\frac{\langle\widetilde\na f\,|\,\widetilde\na\beta \rangle_{\tilde g}}{|\widetilde\na\beta|_{\tilde g}^2}\widetilde\na\beta+(n-1)\frac{1}{f}\widetilde\na f\bigg\rangle_{\tilde g}
$$
By Cauchy--Schwarz the vector on the right-hand side is bounded on any compact set. In particular, $|\widetilde\na\beta|_{\tilde g}^2$ satisfies the maximum principle.
To conclude from this that $|\widetilde\na\beta|_{\tilde g}$ is constant, we first observe via triangle inequality that $\beta(x)-\beta(y)\leq {\rm d}_{\tilde g}(x,y)$ for any two points $x$, $y$. This immediately implies that $|\widetilde\na\beta|_{\tilde g}\leq 1$ on $M$. On the other hand, since $\beta(\sigma(\tau))=\tau$, we conclude that $|\widetilde\na\beta|_{\tilde g} = 1$ on $\sigma$. The strong maximum principle then implies that $|\widetilde\na\beta|_{\tilde g}=1$ on the whole manifold. In particular, the previous inequalities must be equalities, namely
\begin{align}
\Ric_{\tilde g}\big(\widetilde\na\beta,\widetilde\na\beta\big)\,&=\,(n-1)\frac{1}{f}\widetilde\na^2 f\big(\widetilde\na\beta,\widetilde\na\beta\big)-2(n-1)\frac{1}{f^2}\big\langle\widetilde\na f\,\big|\,\widetilde\na\beta\big\rangle_{\tilde g}\,,
\\
\big|\widetilde\na^2\beta\big|_{\tilde g}^2\,&=\,\frac{(\De_{\tilde g}\beta)^2}{n-1}\,.
\end{align}
Furthermore, the fact that $|\widetilde\na\beta|_{\tilde g}=1$ on the whole manifold grants us that we can use $\beta$ as a coordinate and that the manifold $M$ is diffeomorphic to $\R\times\Sigma$, for some $(n-1)$-dimensional manifold $\Sigma$. With respect to coordinates $\{\beta,\theta^1,\dots,\theta^{n-1}\}$, the conformal metric writes as
$$
\tilde g\,=\,d\beta\otimes d\beta+\tilde g_{ij}d\theta^i\otimes d\theta^j\,.
$$
Again, since $|\widetilde\na\beta|_{\tilde g}=1$, for any vector $X$ it holds
$$
\widetilde\na^2\beta\big(\widetilde\na\beta,X\big)\,=\,\big\langle\widetilde\na\big|\widetilde\na\beta\big|_{\tilde g}^2\,\big|\,X\big\rangle_{\tilde g}\,=\,0\,.
$$
It follows immediately from this and the identity $|\widetilde\na^2\beta|_{\tilde g}^2=(\De_{\tilde g}\beta)^2/(n-1)$ that, in the coordinates in which $\tilde g$ has the form~\eqref{eq:gamma}, for any $i,j=1,\dots,n-1$ it holds 
$$
\widetilde\na^2_{ij}\beta\,=\,\frac{\De_{\tilde g}\beta}{n-1}\tilde g_{ij}\,=\,-\frac{1}{f}\frac{\pa f}{\pa \beta}\tilde g_{ij}\,,
$$
where the latter identity makes use of~\eqref{eq:useful_formulas}.
On the other hand, from the definition of Hessian we have $\widetilde\na^2_{ij}\beta=-\Gamma_{ij}^\beta=\pa_\beta \tilde g_{ij}/2$, hence
$$
\frac{\pa\tilde g_{ij}}{\pa\beta}\,=\,-\frac{2}{f}\frac{\pa f}{\pa\beta}\tilde g_{ij}\,.
$$
This identity can be solved explicitly, yielding
$$
\tilde g_{ij}\,=\,\frac{1}{f^2}(g_0)_{ij}\,,
$$
where $(g_0)_{ij}$ does not depend on $\rho$. Comparing with~\eqref{eq:gamma} and recalling $g=f^2\tilde g$, we have obtained
$$
g\,=\,f^2 d\beta\otimes d\beta+g_0\,.
$$

Finally, we remark again that in this proof we did not have to ask for the boundary of $M$ to be empty since, as observed in Lemma~\ref{le:completeness}, with respect to the conformal metric $\tilde g$ the boundary components become ends, hence they cannot obstruct minimizing geodesics. Therefore, the argument to produce a line between two $f$-complete ends goes through and the manifold splits. 
But then this would imply $\partial M = (-\infty, +\infty) \times \partial \Sigma$, which contradicts our initial assumption that the boundary is compact. It follows that the boundary must be empty if there is more than one $f$-complete end.
\end{proof}

\begin{remark}
We point out that it is actually possible to obtain a stronger thesis in Theorem~\ref{thm:splitting_intro} above. In fact, proceeding as in the proof of Theorem~\ref{thm:BG_nb}, one can also show that identity~\eqref{eq:rigidity_2} is in force, and from there deduce that $f=f_1 f_2$, where $f_1$ is a function of $s$ whereas $f_2$ does not depend on $s$. 
We do not give the details on this computation, which has already been performed in the conformal ${\rm CD}(0,1)$ framework~\cite[Proposition~2.2]{Wylie}. Recalling the relation between ${\rm CD}(0,1)$ and substatic discussed in \cref{app:CD_subst}, one can easily translate this result in our setting.
One may also write down explicitly the substatic condition in the directions tangential to the cross section, to obtain some information on the triple $(\Sigma,g_0,f_2)$. Again, in the ${\rm CD}(0,1)$ setting, this has been done in~\cite[Proposition~2.3]{Wylie}, where it is shown that the triple $(\Sigma,g_0,-(n-1)\log f_2)$ satisfies the ${\rm CD(0,1)}$ condition (in fact, it is even ${\rm CD}(0,0)$). It is not immediately clear whether this fact translates nicely in our setting. These refinements of the thesis of \cref{thm:splitting_intro} will not be needed in the rest of the paper.
\end{remark}



\subsection{Splitting theorem for conformally compact ends}

We now discuss conformally compact ends. For such ends, by definition the metric extends to the conformal infinity sufficiently smoothly so that the mean curvature $\HHH_{\tilde g}$ of the conformal infinity $\pa E_\infty$ is well defined. On the other hand, the mean curvatures $\HHH$ and $\HHH_{\tilde g}$ of a hypersurface with respect to the two different metrics can be seen to be related by
$$
\HHH_{\tilde g}\,=\,f\HHH-(n-1)\langle \na f\,|\,\nu\rangle\,.
$$
Alternatively, setting $\ffi=1/f$ we can write
$$
\HHH\,=\,\ffi\HHH_{\tilde g}-(n-1)\big\langle \widetilde\na \ffi\,\big|\,\nu_{\tilde g}\big\rangle_{\tilde g}
\,.
$$
By definition of conformal compactness we know that $\ffi$ extend in a $\mathscr{C}^3$ fashion to the conformal boundary by setting $\ffi= 0$ on $\pa E_\infty$. In particular $|\widetilde\na\ffi|_{\tilde g}$ is bounded, which implies that the quantity $\HHH/f$ can be extended to zero in a continuous fashion on $\pa E_\infty$. Taking now as $\rho$ the $\tilde g$-distance from $\pa E_\infty$, recalling the Riccati equation~\eqref{eq:H_Riccati} we have that
$$
\frac{\pa}{\pa\rho}\left(\frac{\HHH}{f}\right)\,\leq\,-\frac{1}{n-1}\,\HHH^2\,,\qquad {\frac{\HHH}{f}}_{|_{\rho=0}}=0\,.
$$
The assumption of $\mathscr{C}^3$-regularity of the conformal boundary made in Definition~\ref{def:cc} was needed precisely to make sense of the $\rho$-derivative of $\HHH$. 
Proceeding exactly as in Subsection~\ref{sub:growth_mc_hypersurfaces}, from this formula on the evolution of the mean curvature we obtain the Laplacian comparison
\begin{equation}
\label{eq:laplacian_comparison_cc}
\frac{\HHH}{f}\,=\,\De\rho+\frac{1}{f}\langle\na f\,|\,\na\rho\rangle\,\leq\,0\,.  
\end{equation}
This is the main ingredient to prove the Splitting Theorem in the conformally compact setting:

\begin{theorem}
\label{thm:splitting_cc}
Let $(M,g,f)$ be a substatic triple with conformally compact ends.
Then there is at most one end.
\end{theorem}

\begin{proof}
Again, a ${\rm CD}(0,1)$-version of this argument can be found in~\cite[Theorem~5.1]{Wylie}. The proof follows closely the one in~\cite[Theorem~B-(1)]{Kasue}, where a splitting theorem for compact manifolds is discussed.
Suppose by contradiction that the conformal infinity has at least two connected components. Let $S_-,S_+$ be the two components with least distance. Then there exists a $\tilde g$-geodesic $\sigma$ minimizing the distance between them. 

Let $\beta_-$ (resp. $\beta_+$) be the distance from $S_-$ (resp. $S_+$) with respect to $\tilde g$. The discussion above  grants us that both $\beta_-$ and $\beta_+$ satisfy the Laplacian comparison~\eqref{eq:laplacian_comparison_cc}. In particular, so does $\beta_-+\beta_+$. Since by construction $\beta_-+\beta_+$ reaches its minimum value ${\rm dist}(S_-,S_+)$ on the geodesic $\sigma$, we then conclude by the strong maximum principle that $\beta_-+\beta_+$ is constant and equal to ${\rm dist}(S_-,S_+)$ on the whole manifold. It follows immediately that $\beta=\beta_+$ satisfies
\[
\De\beta+\frac{1}{f}\langle\na f\,|\,\na\beta\rangle\,=\,0
\]
in the barrier sense. We are now exactly in the same situation reached in the proof of the Splitting Theorem for $f$-complete ends. We can then proceed exactly as after formula~\eqref{eq:pde_beta} to conclude that $(M,g)$ is isometric to a twisted product
\[
\left((a,b)\times\Sigma,\,f^2\,ds\otimes ds+g_\Sigma\right)\,.
\]  
On the other hand, such a manifold is not conformally compact, as the metric $\tilde g=g/f^2$ is degenerate as $s$ approaches $a$ or $b$. We have thus reached a contradiction, implying that there were not multiple conformally compact ends in the first place.
\end{proof}



This result generalizes~\cite[Theorem~I.1]{Chrusciel_Simon}, where the same thesis is obtained for conformally compact vacuum static solutions with negative cosmological constant. 
It is interesting to notice that the proof proposed in~\cite{Chrusciel_Simon} also makes use of the conformal metric $\tilde g=g/f^2$, which is exploited to invoke a spacetime censorship result from~\cite[Theorem~2.1]{Gal_Sch_Wit_Woo}. 


\subsection{Splitting Theorem for mixed ends}

For completeness, we include here the case where there are ends with different behaviours. This case was not considered in~\cite{Wylie} but the proof is similar.

\begin{theorem}
\label{thm:splitting_mixed}
Let $(M,g,f)$ be a substatic triple with ends that are either conformally compact or $f$-complete.  If there is at least one $f$-complete end, then there cannot be any conformally compact end.
\end{theorem}

\begin{proof}
This time the proof follows~\cite[Theorem~C-(2)]{Kasue}. Suppose that there is a conformally compact end and an $f$-complete end.
Then, one constructs a globally minimizing $\tilde g$-geodesic $\sigma$ starting at a connected component $S$ of the conformal boundary and reaching infinity. Let $\beta_-$ be the distance from $S$ and $\beta_+$ be the Busemann function relative to $\sigma$. As in the previous cases, from the Laplacian comparisons for both $\beta_-$ and $\beta_+$ and the fact that $\beta_-+\beta_+$ achieves its minimum value $0$ on $\sigma$, we deduce that $\beta=\beta_+$ satisfies
$$
\De\beta+\frac{1}{f}\langle\na f\,|\,\na\beta\rangle\,=\,0
$$
in the barrier sense. We now proceed as in the other cases to show that the manifold must be a twisted product 
$$
\left((0,+\infty)\times\Sigma,\,f^2\,ds\otimes ds+g_\Sigma\right)\,.
$$ 
Again as in the proof of Theorem~\ref{thm:splitting_cc}, we observe that the end corresponding to $s=0$ cannot be conformally compact as the metric $\tilde g=g/f^2$ becomes degenerate as $s\to 0$. We have thus reached a contradiction, meaning that it is impossible to have an $f$-complete end and a conformally compact end at the same time.
\end{proof}

This theorem, together with the other results in this Section (Theorem~\ref{thm:splitting_intro} and Theorem~\ref{thm:splitting_cc}) strongly narrows the acceptable configurations of ends for a substatic triple. We sum up the topological information we have collected in the following statement.
 
\begin{corollary}
Let $(M,g,f)$ be a substatic triple with ends that are either conformally compact or $f$-complete. If there is more than one end, then there are exactly two ends, both $f$-complete, and $\pa M=\emptyset$.
\end{corollary}

\section{Asymptotic Volume Ratio and Willmore-type inequality}
\label{sec:AVR}

In this section we focus on $f$-complete ends and we introduce the notion of asymptotic volume ratio (AVR),
in analogy with the classical case of nonpositive Ricci curvature, as the limit of the Bishop--Gromov monotonic quantity. In order to have a well defined AVR, we will need to focus our attention to the special case of uniform ends. Building on the notion of  AVR we will finally prove the Willmore-type inequality mentioned in the introduction.

\subsection{Uniform $f$-complete ends}

Here we introduce and comment the notion of uniformity of an $f$-complete end. For convenience, instead of working on the whole $M$, we focus our attention on the end only. In other words, starting from the next definition and for most of this subsection, instead of working on the whole substatic triple $(M,g,f)$, we just consider a neighborhood $E$ of our end and we focus our attention on the restriction $(E,g,f)$, which we refer to as a substatic $f$-complete end. It is easy to show that the definitions and statements below do not depend on the choice of the neighborhood $E$ of our end.

\begin{definition}
\label{def:uniform}
Let $(E,g,f)$ be a substatic $f$-complete end. We say that $(E,g,f)$ is  \emph{uniform}, if, for any two  compact hypersurfaces $\Sigma_1$, $\Sigma_2$ contained in the interior of $E$ and every $\delta>0$, there exists a compact set $K\supset \partial E$ such that for any two unit speed $\tilde g$-geodesics $\sigma_1$, $\sigma_2$ minimizing the distance between $\Sigma_1$, $\Sigma_2$ and a point $p=\sigma_1(t_1)=\sigma_2(t_2)$ outside $K$, it holds
\begin{equation}
\label{eq:uniformity_requirement}
\left|\frac{\int_0^{t_1} f^2(\sigma_1(t))dt}{\int_0^{t_2} f^2(\sigma_2(t))dt}\,-\,1 \right|\,\leq\,\delta\,.
\end{equation}
\end{definition}

While the definition above is slightly technical, we point out that there are natural cases in which uniformity is guaranteed. We give here a couple of easily described families of uniform $f$-complete ends. The following result for instance guarantees us that an end is uniform as long as $f$ goes to one at infinity. 

\begin{proposition}
\label{prop:uniform_AF}
Let $(E,g,f)$ be a substatic end. If $f\to 1$ at infinity, then $(E,g,f)$ is $f$-complete and uniform.
\end{proposition}

\begin{proof}
The fact that the ends are $f$-complete has already been shown in far greater generality in \cref{pro:fcomplete_AF}.
Let now $\Sigma_1$, $\Sigma_2$ be two hypersurfaces. 
Since $f\to 1$ at infinity, for every $\ep$ the set $K_\ep=\{|f-1|>\ep\}$ is compact. In particular $1-\ep<f<1+\ep$ outside $K_\ep$.
We consider now a $\tilde g$-geodesically convex compact set $K$ containing $K_\ep$ and the two hypersurfaces  $\Sigma_1$, $\Sigma_2$.

Let $\sigma_1$, $\sigma_2$ be two unit speed $\tilde g$-geodesics minimizing the distance between $\Sigma_1$, $\Sigma_2$ and a point $p=\sigma_1(t_1)=\sigma_2(t_2)$ outside $K_\ep$. For $j=1,2$, let $0<T_j<t_j$ be the largest number such that $\sigma_j(T_j)\in K_\ep$.
We then have, for $j=1,2$,
\[
0\,<\,\int_0^{T_j} f^2(\sigma_j(t))dt\,<\,T_j\max_{K}f^2\,,
\qquad(1-\ep)^2(t_j-T_j)\,<\,\int_{T_j}^{t_j} f^2(\sigma_j(t))dt\,<\,(1+\ep)^2(t_j-T_j)\,.
\]
As a consequence, we estimate:
\[
\frac{(1-\ep)^2(t_1-T_1)}{T_2\max_{K}f^2+(1-\ep)^2(t_2-T_2)}\,<\,\frac{\int_0^{t_1} f^2(\sigma_1(t))dt}{\int_0^{t_2} f^2(\sigma_2(t))dt}\,<\,\frac{T_1\max_{K}f^2+(1+\ep)^2(t_1-T_1)}{(1-\ep)^2(t_2-T_2)}\,.
\]
Since $f$ is bounded at infinity, the quantity $\max_{K}f^2$ is bounded by a constant independent of $K$. Furthermore, we have $T_j<{\rm diam}_{\tilde g}K_\ep$, for $j=1,2$, by construction. On the other hand, if we take the compact set $K$ to be much larger than $K_\ep$, we can make $t_1$ and $t_2$ arbitrarily large. The uniformity estimate~\eqref{eq:uniformity_requirement} follows then easily.
\end{proof}

Another case in which uniformity is guaranteed is under the assumption that the norm of the gradient of $f$ decays sufficiently fast.

\begin{proposition}
\label{prop:condition-uniformatinf}
Let $(E, g, f)$ be a substatic $f$-complete end and let $\rho$ be the $\tilde g$-distance from a point, where $\tilde g=g/f^2$.
If for some $\ep>0$ there exist a compact set $K\supset\pa E$ and a constant $C>0$ such that \[
|\na f|<C\rho^{-1-\ep}
\]
outside $K$, then $(E, g, f)$ is uniform.
\end{proposition}

\begin{proof}
Fix the compact hypersurfaces $\Sigma_1$, $\Sigma_2$, the point $x$ and the constant $\ep>0$. Let $K\supset\pa E$ be the compact set such that $|\na f|<C\rho^{-1-\ep}$ outside $K$, where $\rho$ is the $\tilde g$-distance from $x$. Up to enlarging $K$, we can suppose that $x$, $\Sigma_1$ and $\Sigma_2$ are inside $K$.

Let $p$ be a point outside $K$. For $i=1,2$, consider the unit speed $\tilde g$-geodesic $\sigma_i:[0,t_i]\to M$ minimizing the distance between $\Sigma_i$ and $p$ and such that $\sigma_i(0)\in\Sigma_i$, $\sigma_i(t_i)=p$. 
We compare the value of $f$ at a point $\sigma_i(\tau)$ and at the point $p$. Integrating along the geodesic, we find
$$
\log f(p)\,=\,\log f(\sigma_i(\tau))+\int_{\tau}^{t_i}\big\langle\widetilde\na \log f\,\big|\,\dot\sigma_i\big\rangle_{\tilde g}(\sigma_i(t))\,dt\,,
$$
hence
\begin{equation}
\label{eq:logf_estimate}
\left|\log\left(\frac{f(\sigma_i(\tau))}{f(p)}\right)\right|\,\leq\,\int_{\tau}^{t_i}\frac{1}{f}\big|\widetilde\na f\big|_{\tilde g}(\sigma_i(t))\,dt
=\,\int_{\tau}^{t_i}|\na f|(\sigma_i(t))\,dt\,.
\end{equation}
We now exploit our hypothesis. Assume that the segment ${\sigma_i}_{|_{[\tau,t_i]}}$ is outside $K$. Then $|\na f|\leq C\rho^{-1-\ep}$ at the points of ${\sigma_i}_{|_{[\tau,t_i]}}$, where $\rho$ is the distance from $x$. Notice that for all $\tau\leq t\leq t_i$, $\sigma_i(t)$ is at distance $t$ from $\Sigma_i$, so that by triangle inequality, for any $y\in K$ it holds
$$
t-\max_{y\in\Sigma_i}{\rm d}_{\tilde g}(x,y)\leq\rho(\sigma_i(t))\leq t+\max_{y\in\Sigma_i}{\rm d}_{\tilde g}(x,y)\,.
$$
If we then take 
\begin{equation}
\label{eq:aux_uniform}
t\geq 2\max_{y\in\Sigma_i}{\rm d}_{\tilde g}(x,y)
\end{equation}
we get 
$$
|\na f|(\sigma_i(t))\,\leq\,
C\rho^{-1-\ep}(\sigma_i(t))
\,\leq\,C(t-\max_{y\in\Sigma_i}{\rm d}_{\tilde g}(x,y))^{-1-\ep}\,\leq\,C\left(\frac{t}{2}\right)^{-1-\ep}=2^{1+\ep}Ct^{-1-\ep}\,.
$$
If we then suppose that $\tau\geq 2\max_{y\in\Sigma_i}{\rm d}_{\tilde g}(x,y)$, we deduce from~\eqref{eq:logf_estimate} that
\begin{equation}
\left|\log\left(\frac{f(\sigma_i(\tau))}{f(p)}\right)\right|\,\leq\,2^{1+\ep}C\frac{\tau^{-\ep}-t_i^{-\ep}}{\ep}
\end{equation}
that is,
\begin{equation}
\label{eq:logf_estimate_2}
e^{-\frac{2^{1+\ep}C}{\ep}\tau^{-\ep}}\,\leq\,{\rm exp}\left[2^{1+\ep}C\frac{t_i^{-\ep}-\tau^{-\ep}}{\ep}\right]\,\leq\,\frac{f(\sigma_i(\tau))}{f(p)}\,\leq\,{\rm exp}\left[2^{1+\ep}C\frac{\tau^{-\ep}-t_i^{-\ep}}{\ep}\right]\,\leq\,e^{\frac{2^{1+\ep}C}{\ep}\tau^{-\ep}}\,.
\end{equation}

Let now $\kappa>0$ be a large number and consider the compact set 
\[
K_\kappa\,=\,B^{\tilde g}_\kappa(K)\,:=\,\{y\in M\,:\,{\rm d}_{\tilde g}(K,y)\leq\kappa\}\,.
\]
Since $\Sigma_i\subset K$, notice that if $\sigma_i(t)\not\in K_\kappa$ then $t>\kappa$. It follows that, up to taking $\kappa$ large enough, inequality~\eqref{eq:aux_uniform} holds and in particular $|\na f|(\sigma_i(t))\leq2^{1+\ep}Ct^{-1-\ep}$ for any $t$ such that $\sigma_i(t)\not\in K_\kappa$.

Conversely, also notice by triangle inequality that if 
$$
t>\kappa+\max_{y\in\Sigma_i,z\in \pa K\setminus\pa E}{\rm d}_{\tilde g}(y,z)\,,
$$
then $\sigma_i(t)\not\in K_\kappa$. Up to taking $\kappa$ large enough, we can then also suppose that $\sigma_i(t)\not\in K_\kappa$ for every $t>2\kappa$. In particular, for any $\tau>2\kappa$ we can apply estimate~\eqref{eq:logf_estimate_2}, obtaining
\begin{equation}
\label{eq:logf_estimate_3}
e^{-\frac{2C}{\ep}\kappa^{-\ep}}\,\leq\,\frac{f(\sigma_i(\tau))}{f(p)}\,\leq\,e^{\frac{2C}{\ep}\kappa^{-\ep}}
\end{equation}
We are finally ready to prove uniformity at infinity with respect to the compact set $K_\kappa$, for $\kappa$ sufficiently large. As already noticed, $\sigma_i(t)\in K_{2\kappa}$ for $t<2\kappa$, hence
$$
\frac{\int_0^{t_1}f^2(\sigma_1(t))dt}{\int_0^{t_2}f^2(\sigma_2(t))dt}\,\leq\,
\frac{2\kappa\max_{K_{2\kappa}}f^2+\int_{2\kappa}^{t_1}f^2(\sigma_1(t))dt}{2\kappa\min_{K_{2\kappa}}f^2+\int_{2\kappa}^{t_2}f^2(\sigma_2(t))dt}\,\leq\,
\frac{2\kappa\max_{K_{2\kappa}}f^2+(t_1-2\kappa)f(p)^2 e^{\frac{2C}{\ep}\kappa^{-\ep}}}{2\kappa\min_{K_{2\kappa}}f^2+(t_2-2\kappa)f(p)^2e^{-\frac{2C}{\ep}\kappa^{-\ep}}}\,.
$$
Notice that  $t_1$ and $t_2$ are comparable (their difference is bounded via the triangle inequality by the maximum of the distance between points of $\Sigma_1$ and $\Sigma_2$). Therefore, for any $\tilde\ep>0$ arbitrarily small, we can find $\tilde \kappa$ much larger than $\kappa$ so that, assuming $p$ is outside $K_{\tilde\kappa}$ (in particular $t_1,t_2$ are also much larger than $\kappa$) it holds
$$
\frac{\int_0^{t_1}f^2(\sigma_1(t))dt}{\int_0^{t_2}f^2(\sigma_2(t))dt}\,\leq\,(1+\tilde\ep)e^{\frac{4C}{\ep}\kappa^{-\ep}}\,.
$$
Up to choosing $\kappa$ large enough, we can also make the exponential term in the inequality above as close to $1$ as necessary.
Of course, exchanging the roles of $\sigma_1$ and $\sigma_2$ we also find the opposite bound. This proves uniformity.
\end{proof}

Equipped with the notion of uniformity of the ends, we are now ready to define the substatic version of the asymptotic volume ratio.

\begin{definition}
\label{def:AVR}
Let $(M,g,f)$ be a substatic solution and let $E$ be a uniform $f$-complete end. Let $\rho$ be the distance function to a point or a hypersurface with respect to the metric $\tilde g=g/f^2$ and $\eta$ be the solution to~\eqref{eq:eta} or~\eqref{eq:eta_Sigma}, respectively.
The {\em Asymptotic Volume Ratio $\mathrm{AVR}(E, g, f)$} of $E$ is defined as
$$
{\rm AVR}(E, g, f)\,=\,\frac{1}{|\Sph^{n-1}|}\lim_{t\to+\infty}\int_{\{\rho=t\}\cap E}\frac{1}{\eta^{n-1}}d\sigma\,.
$$
If $(M,g,f)$ has a unique end $E$, we refer to $\mathrm{AVR}(E, g, f)$  with $\mathrm{AVR}(M, g, f)$.
\end{definition}

\smallskip

The following basic fact motivates the introduction of the notion of uniform ends.

\begin{proposition}
\label{pro:AVR_wellposed}
The substatic Asymptotic Volume Ratio  is well-defined on any uniform $f$-complete end. In other words, its definition does not depend on the choice of the point/hypersurface we are taking the distance $\rho$ from.
\end{proposition}

\begin{proof}
Let $\rho$ be the $\tilde g$-distance from a point or a hypersurface. We consider the functional $V(t)$ defined in~\eqref{eq:V_k}, that we recall here for the reader's convenience:
\[
V(t)\,=\,\frac{1}{|\mathbb{B}^n|t^k}\int_{\{0\leq\rho\leq t\}}\frac{\rho^{k-1}}{f\eta^{n-1}}d\mu\,,
\]
where $k>0$ is a constant.
A simple application of L'H\^{o}pital's rule tells us immediately that
\begin{align}
\lim_{t\to+\infty}V(t)\,&=\,\frac{1}{|\mathbb{B}^n|}\lim_{t\to+\infty}\frac{1}{t^k}\int_{\{\rho\leq t\}}\frac{\rho^{k-1}}{f\eta^{n-1}}d\mu
\\
&=\,\frac{n}{k|\mathbb{S}^{n-1}|}\lim_{t\to+\infty}\int_{\{\rho= t\}}\frac{1}{\eta^{n-1}}d\sigma
\\
\label{eq:AVR_A_V}
&=\,\frac{n}{k}\lim_{t\to+\infty}A(t)\,.
\end{align}

In order to conclude the proof, it is then enough to show that $\lim_{t\to+\infty}V(t)$ is independent of the choice of the point/hypersurface. In the rest of the proof, it is convenient to set $k=1$ in our functional $V$.
Let $\eta_1$, $\eta_2$ be reparametrized distances with respect to two different points (resp. two different hypersurfaces) and let $\delta$ be the distance between the two points (resp. the maximum distance between points of the two hypersurfaces) with respect to the metric $\tilde g=g/f^2$. By triangle inequality we have the inclusion $\{\rho_1\leq t-\delta\}\subset\{\rho_2\leq t\}$, therefore
\begin{align}
V_1(t)-V_2(t)&=\frac{1}{|\mathbb{B}^n|t}\int_{\{\rho_1\leq t\}}\frac{1}{f\eta_1^{n-1}}d\mu-\frac{1}{|\mathbb{B}^n|t}\int_{\{\rho_2\leq t\}}\frac{1}{f\eta_2^{n-1}}d\mu
\\
&\leq 
\frac{1}{|\mathbb{B}^n|t}\int_{\{\rho_1\leq t\}}\frac{1}{f\eta_1^{n-1}}d\mu-\frac{1}{|\mathbb{B}^n|t}\int_{\{\rho_1\leq t-\delta\}}\frac{1}{f\eta_2^{n-1}}d\mu
\\
&\leq 
\frac{1}{|\mathbb{B}^n|t}\int_{\{t-\delta\leq\rho_1\leq t\}}\frac{1}{f\eta_1^{n-1}}d\mu+\frac{1}{|\mathbb{B}^n|t}\int_{\{\rho_1\leq t-\delta\}}\frac{1}{f\eta_1^{n-1}}\left(1-\frac{\eta_1^{n-1}}{\eta_2^{n-1}}\right)d\mu\,.
\end{align}
Concerning the first integral, applying again L'H\^{o}pital's rule we find that its limit is the same as the limit of $nA(t)-nA(t-\delta)$ at $t\to+\infty$, where $A$ is the usual area functional with respect to the distance $\rho_1$. Since $A(t)$ has a finite limit at infinity and $\delta$ is fixed, this limit is zero. From the uniformity of the end and the fact that $V(t)$ is bounded, we deduce that the second integral also goes to zero. Hence, we have found that $\lim_{t\to+\infty}[V_1(t)-V_2(t)]\leq 0$.
Switching the roles of $V_1$ and $V_2$ we find that the opposite inequality is also in place, hence the limits of $V_1(t)$ and $V_2(t)$ are the same, as wished.
\end{proof}

\begin{remark}
As noted in \cref{rem:eta}, $\eta$ represents the distance along radial $\tilde{g}$-geodesics with respect to the metric $\overline{g}= f^2 g$. Providing a suitable Bishop-Gromov-type Theorem in terms of the $\overline{g}$-distance in place of $\eta$ may be useful to cook a notion of Asymptotic Volume Ratio that does not need the notion of uniformity to be well defined.
\end{remark}

The following is a  basic yet fundamental consequence of the Splitting Theorem~\ref{thm:splitting_intro}. 

\begin{lemma}
\label{lem:positiveavr->1end}
Let $(M, g, f)$ be a substatic triple with $f$-complete ends. If there is more than one uniform $f$-complete end, then all ends have vanishing asymptotic volume ratio.
\end{lemma}

\begin{proof}
Suppose that there is more than one uniform $f$-complete end. Then the Splitting Theorem~\ref{thm:splitting_intro} implies that the manifold splits as a twisted product
$$
(\R\times\Sigma,\,f^2\,ds\otimes ds+g_\Sigma)\,,
$$ 
for some $(n-1)$-dimensional Riemannian manifold $(\Sigma,g_\Sigma)$. Let $\rho$ be the $\tilde g$-distance from the cross section $\{s=0\}$ and $\eta$ be defined by~\eqref{eq:eta_Sigma}, as usual. Notice that the level sets of $\rho$ and $\eta$ are also level sets of $s$, hence in particular the metric induced on any level set of $\rho$ is $g_\Sigma$. It follows then that
\[
{\rm AVR}(E, g, f)\,=\,\frac{1}{|\Sph^{n-1}|}\lim_{t\to+\infty}\int_{\{\rho=t\}\cap E}\frac{1}{\eta^{n-1}}d\sigma\,=\,\lim_{t\to+\infty}
\frac{|\Sigma|}{|\Sph^{n-1}|\eta_{|_{\{\rho=t\}}}^{n-1}}
\,.
\]
Since the end is $f$-complete, we have $\eta\to+\infty$ at infinity, hence the above limit vanishes. 
\end{proof}

In light of the above Lemma, our main geometric inequalities \cref{eq:isoperimetric_intro} and \cref{eq:willmore-intro} will only involve one end and the global $\mathrm{AVR}(M, g, f)$.

In this framework, we now discuss some cases in which we are able to give more precise estimates for the Asymptotic Volume Ratio. A first simple estimate, in the case where the boundary is empty, is obtained from~\eqref{eq:BG_effective}. Taking the limit of this formula as $t\to+\infty$, assuming that the boundary is empty (so that the term $A(t)$ appearing in that formula converges to the asymptotic volume ratio) we find the following
\[
{\rm AVR}(M,g,f)\,f(p)^{n-1}\,\leq\,1\,.
\]
This must hold for any point $p\in M$.
In particular, it follows that if $\pa M=\emptyset$ and $f$ is not bounded then the Asymptotic Volume Ratio must vanish.

An important family of substatic manifolds having nonzero AVR is that of asymptotically flat triples, that we now define precisely.

\begin{definition}
\label{def:AF}
A substatic triple $(M,g,f)$ is said to be {\em asymptotically flat} if
\begin{itemize}
\item[$(i)$] there exists a compact domain $K\supset\pa M$ and a diffeomorphism (called {\em chart at infinity}) between $M\setminus K$ and $\R^n$ minus a ball.
\smallskip
\item[$(ii)$] in the chart at infinity, it holds $|g_{ij}-\delta_{ij}|=o(1)$ and $|f-1|=o(1)$
as $|x|\to+\infty$.
\end{itemize}
\end{definition}

We remark that the usual definition of asymptotic flatness 
requires a higher degree of convergence of the metric $g$ to the Euclidean one. However, the above definition is sufficient to compute precisely the asymptotic volume ratio.

\begin{proposition}
Let $(M,g,f)$ be an asymptotically flat substatic triple. Then ${\rm AVR}(M,g,f)=1$.
\end{proposition}

\begin{proof}
The fact that the end is $f$-complete and uniform follows from \cref{prop:uniform_AF}.
Let $K$ be a compact set as in Definition~\ref{def:AF} and let $S=\{|x|=R\}$ be a large coordinate sphere contained in the chart at infinity. From Proposition~\ref{pro:AVR_wellposed} we know that the asymptotic volume ratio does not depend on the hypersurface we are taking the distance from. It is then convenient to work with the $\tilde g$-distance $\rho$ from the coordinate sphere $S$.
As it follows from~\eqref{eq:AVR_A_V}, we can also compute the AVR via the following limit 
\begin{equation}
\label{eq:AVR_af}
{\rm AVR}(M,g,f)\,=\,\frac{1}{|\mathbb{B}^n|}\lim_{t\to+\infty}\frac{1}{t^n}\int_{\{0\leq \rho\leq t\}}\frac{\rho^{n-1}}{f\eta^{n-1}}d\mu\,.
\end{equation}
If the radius $R$ of the coordinate sphere $S$ is large, then we can assume $|f-1|<\ep$ and $|g_{ij}-\delta_{ij}|<\ep$ in $\{|x|>R\}$ for some fixed small $\ep$. It is then easily seen that there exists $\delta=\delta(\ep)$ such that 
\[
\{R\leq |x|\leq (R+t)(1-\delta)\}\subset\{0\leq\rho\leq t\}\subset\{R\leq|x|\leq(R+t)(1+\delta)\}
\]
for all $t$. Since $\eta$ grows as $f^2$ along $\tilde g$-geodesics, we have $(1-\ep)^2\rho<\eta<(1+\ep)^2\rho$. It follows that the integral in~\eqref{eq:AVR_af} grows as the Euclidean volume of the annulus $\{R<|x|<R+t\}$, or more explicitly as: 
\[
\left[(R+t)^n-R^n\right]\,|\mathbb{B}^n|\,\cong\,|\mathbb{B}^n|t^n\,.
\]
The wished result follows easily. 
\end{proof}

\subsection{Willmore inequality}
\label{sub:Willmore}

As a consequence of our definition of AVR and the Bishop--Gromov monotonicity of the area functional $A(t)$ (Theorem~\ref{thm:BG_nb}), we obtain the Willmore inequality for hypersurfaces with nonnegative mean curvature of \cref{thm:Willmore_intro}. The following statement provides more details about the equality case.

\begin{theorem}[Willmore inequality]
\label{thm:Willmore}
Let $(M,g,f)$ be a substatic solution with a uniform $f$-complete end. Let $\Sigma$ be a hypersurface that is homologous to the boundary. Suppose that the mean curvature $\HHH$ of $\Sigma$ with respect to the normal pointing towards infinity satisfies $\HHH> 0$ pointwise. Then
\begin{equation}
\label{eq:willmore}
\int_\Sigma \left[\frac{\HHH}{(n-1)f}\right]^{n-1}d\sigma\,\geq\,
{\rm AVR}(M,g,f)\,|\Sph^{n-1}|\,.
\end{equation}

If the equality holds, then the set $U=\{\rho>0\}$ is isometric to $[0,+\infty)\times\Sigma$ with metric 
$$
g\,=\,f^2 d\rho\otimes d\rho+\eta^2g_0\,,
$$
where $g_0$ is a metric on the level set $\Sigma$. Furthermore, in $U$ the functions $f$ and $\eta$ satisfy
$$
\eta\,=\,(\alpha+\beta)^{\frac{1}{n-1}}\,,\qquad f^2\,=\,\frac{\dot\alpha}{(n-1)(\alpha+\beta)^{\frac{n-2}{n-1}}}
$$
where $\alpha$ is a function of $\rho$ and $\beta$ is a function on $\Sigma$.
\end{theorem}


\begin{proof}
We recall from Theorem~\ref{thm:BG_nb} that $A(t)$ is monotonically nonincreasing. Taking the limit as $t\to+\infty$ we then get
$$
\frac{1}{|\Sph^{n-1}|}\int_\Sigma \left[\frac{\HHH}{(n-1)f}\right]^{n-1}d\sigma\,=\,A(0)\,\geq\,\lim_{t\to+\infty}A(t)\,=\,\AVR(M,g,f)\,.
$$
This proves the inequality.

Furthermore, if the equality holds, then from the rigidity statement in Theorem~\ref{thm:BG_nb} it follows
$$
\frac{1}{\eta}\frac{\pa\eta}{\pa\theta^i}\,=\,\psi\eta-\ffi\frac{1}{\eta^{n-1}}\,,\qquad
\frac{1}{f}\frac{\pa f}{\pa\theta^i}\,=\,\psi\eta+\frac{n-2}{2}\ffi\frac{1}{\eta^{n-1}}\,,
$$
where $\phi$ and $\psi$ are functions on $\Sigma$.
From the first of these equations, in particular we get
$$
\frac{\pa}{\pa\theta^i}\left(\frac{1}{\eta}\right)\,=\,\psi-\ffi\frac{1}{\eta^n}\,.
$$
We now focus on this identity near infinity: since our end is $f$-complete, we know that $\eta$ is going to $+\infty$. Furthermore, from the uniformity at infinity we can also prove that $1/\eta$ goes to zero uniformly at infinity. To prove that, it is enough to apply the uniformity at infinity property to the two hypersurfaces $\Sigma$ and $\Sigma_\delta=\{\rho=\delta\}$. Notice that then the function $\eta_\delta$ associated to $\Sigma_\delta$ differs from $\eta$ just by a constant $k$, that is $\eta_\delta=\eta+k$. Then uniformity at infinity implies precisely that for any $\ep$ there exists a compact set such that 
$$
\left|\frac{k}{\eta}\right|\,=\,\left|\frac{k+\eta}{\eta}-1\right|\,\leq\,\ep\,.
$$
Hence, $1/\eta$ goes to $0$ uniformly at infinity, as wished. 

Given $\ep>0$, fix $R>0$ big enough so that $1/\eta<\ep$ in $[R,+\infty)\times V$.
If $\psi$ is not everywhere vanishing, then there is an open set $V\subset \Sigma$ such that $\psi>\delta$ in $V$ (the case $\psi<-\delta$ is done in the exact same way). Therefore we would get 
$$
\frac{\pa}{\pa\theta^i}\left(\frac{1}{\eta}\right)\,>\,\delta-|\ffi|\,\ep^n\,\geq\,\delta-\ep^n\max_\Sigma|\ffi|=\tilde\delta
$$
in $[R,+\infty)\times V$. Up to taking $\ep$ small enough, we can assume that $\tilde\delta>0$.  But then, for any two points $p_\rho=(\rho,\theta^1,\dots,\theta^i,\dots,\theta^n)$, $q_\rho=(\rho,\theta^1,\dots,\theta^i+\lambda,\dots,\theta^n)$ belonging to $[R,+\infty)\times V$, we would deduce
$$
\frac{1}{\eta}(p_\rho)\,=\,\frac{1}{\eta}(q_\rho)+\int_0^\lambda \frac{\pa}{\pa \theta^i}{\left(\frac{1}{\eta}\right)}_{|_{(\rho,\theta^1,\dots,\theta^i+t,\dots,\theta^n)}}dt\,\geq\,\lambda\tilde\delta\,,
$$
which in turn would imply $\lim_{\rho\to+\infty}(1/\eta)(p_\rho)\geq\lambda\tilde\delta>0$, contradicting the fact that $1/\eta\to 0$ at infinity.
It follows then that 
\[
\psi=0\,.
\]
Our constraints on $f$ and $\eta$ then become
\begin{equation}
\label{eq:aux_rigidity_willmore}
\frac{1}{\eta}\frac{\pa\eta}{\pa\theta^i}\,=\,-\ffi\frac{1}{\eta^{n-1}}\,,\qquad
\frac{1}{f}\frac{\pa f}{\pa\theta^i}\,=\,\frac{n-2}{2}\ffi\frac{1}{\eta^{n-1}}\,.
\end{equation}
The first equation can be rewritten as
$$
\frac{\pa}{\pa\theta^i}\eta^{n-1}\,=\,-(n-1)\ffi\,.
$$
Since $\ffi$ does not depend on $\rho$, it follows then that
\[
\eta^{n-1}\,=\,\alpha+\beta\,,
\]
where $\alpha$ is a function of $\rho$ and $\beta$ is a function on $\Sigma$. Taking the derivative with respect to $\rho$ of this formula and using the fact that $\pa_\rho\eta=f^2$, we then deduce
$$
f^2\,=\,\frac{\dot\alpha}{(n-1)(\alpha+\beta)^{\frac{n-2}{n-1}}}
$$
It is easy to check that for any $f$ and $\eta$ of this form (that is, for any choice of $\alpha$ and $\beta$), formulas~\eqref{eq:aux_rigidity_willmore} are satisfied.
\end{proof}

\section{Isoperimetric Inequality for Substatic manifolds}
\label{sec:isoperimetric}

In this Section, we focus our attention on substatic manifolds $(M, g, f)$ admitting an \emph{exhaustion of outward minimising hypersurfaces} homologous to $\partial M$. A hypersurface $\Sigma$ homologous to $\partial M$ is outward minimizing if, 
 denoting by $\Omega$ the compact domain with $\pa\Omega=\Sigma\sqcup\pa M$,
we have 
\begin{equation}
P(\Omega) \leq P(F) 
\end{equation}
for any bounded set $F \supset \Omega$. We say that a sequence $(S_j)_{j \in \N}$ of hypersurfaces homologous to $\partial M$ exhaust $M$ if, given a compact set $K\subset M$, there exists an element $S$ in the sequence such that $K \subset \Omega$, for  $\Omega$ satisfying $\partial \Omega = S \sqcup \partial M$. Conditions ensuring the existence of such an exhaustion are discussed in \cite{Fogagnolo_Mazzieri-minimising}.

We start from showing that $\partial M$ is a priori 
area minimizing. In showing so, we also derive that $\partial M$ is \emph{outermost}, that is, there exist no minimal submanifolds homologous to $\partial M$ other than the boundary itself.
These facts are the first main reason why we require the existence of (nonminimal) outward minimizing sets homologous to $\partial M$.
Since the following auxiliary result does not need any a priori growth at infinity assumption, we think it may have an independent interest.
\begin{proposition}[The boundary is outermost area-minimizing]
\label{prop:outermost}
Let $(M, g, f)$ be a substatic triple with horizon boundary. Assume that there exists an outward minimizing smooth hypersurface $S$ homologous to $\partial M$. Then, the horizon is  outward minimizing, meaning that 
\begin{equation}
\label{eq:minimizingboundary}
\abs{\Sigma} \geq \abs{\partial M}
\end{equation}
for any hypersurface $\Sigma$ homologous to $\partial M$. Moreover, it is outermost, that is there exists no other minimal hypersurfaces homologous to $\partial M$. 
\end{proposition}
\begin{proof}
Let $\Omega$ be such that $\partial \Omega = S\sqcup \partial M$. Indeed $S$ being outward minimizing is mean-convex, and consequently the Maximum Principle implies it is disjoint from $\partial M$ (see e.g.~\cite[Corollary 4.2]{Lee_book}; it is a consequence of the strong comparison principle for quasilinear equations).
We flow $\Omega$ by weak Mean Curvature Flow, referring to the notion considered in~\cite{Ilmanen-mcf}. In particular the analysis carried out by White~\cite{White-mcf} applies. Moreover, observe that the mean curvature of $S$ is necessarily nonnegative, and in particular $\Omega$ is mean-convex in the sense of~\cite[Section~3]{White-mcf}. Since $\partial M$ constitutes itself a (steady) MCF, the well-known~\cite[Inclusion Property 5.3]{Ilmanen-mcf} ensures that the possibly singular evolving sets $\partial \Omega_t \setminus \partial M$ remain homologous to the horizon. By \cite[Theorem 11.1]{White-mcf}, $\partial \Omega_t$ must converge smoothly to a minimal hypersurface $\Sigma$, necessarily homologous to $\partial M$. We show that $\Sigma$ can be the horizon only. Indeed, if this were not the case, $\Sigma$ would be detached from $\partial M$ by the Maximum Principle, and~$(iii)$ in \cref{pro:H_bound_Sigma}  would apply, foliating an outer neighbourhood of $\Sigma$ with hypersurfaces of nonpositive mean curvature. But this is a contradiction, through the Maximum Principle for the mean curvature operator, applied on tangency points, with the smooth mean-convex Mean Curvature Flow smoothly approaching $\Sigma$. Then, the Mean Curvature Flow of $\Omega$ converges smoothly to $\partial M$. 
Observe that this also implies that no minimal hypersurface homologous to $\partial M$ contained in $\Omega$ can exist. Indeed, if there were one, it would obviously remain fixed under MCF, and thus $\Omega_t$ converging to $\partial M$ would eventually go beyond it, contradicting~\cite[Inclusion Property 5.3]{Ilmanen-mcf}. The nonminimal outward minimizing sets forming an exhaustion, we in particular proved that $\partial M$ is outermost.

Finally, recall that the outward minimizing property of the initial set is preserved along the flow, as it can be easily checked applying~\cite[One-Sided Minimization Theorem~3.5]{White-mcf} (see~\cite[Lemma 5.6]{huisken_inverse_2001} for a proof in the smooth flow setting). Then, $\partial M$ being one-sided limit of outward minimizing hypersurfaces homologous to the boundary, is outward minimizing as well.    
\end{proof}

As already pointed out in the Introduction, our proof of \cref{thm:isoperimetric_intro} ultimately builds on the application of the Willmore-type inequality~\eqref{eq:willmore} on hypersurfaces homologous to $\partial M$ bounding a set that is isoperimetric with respect to the volume weighted by $f$. In order to bypass the lack of existence, we will consider \emph{constrained} isoperimetric sets. We find convenient to extend $(M, g)$ over the horizon, letting $(N, g_N)$ be the extended Riemannian manifold. This can be obtained through gluing another copy of $M$ along its boundary, and endowing it with a smooth metric that coincides with $g$ on the original manifold. The existence of such metric is ensured by \cite[Theorem A]{Pigola_Veronelli}. Let $S$ be homologous to $\partial M$ and disjoint from it, and  let $\Omega\subset M$ have boundary $S \sqcup \partial M$.  Extend $\Omega$ too, so to find $\Omega_N \subset N$ satisfying $\Omega_N \cap M = \Omega$, and let 
\begin{equation}
B^{N\setminus M}_\epsilon (\partial M) = \{p \in N\setminus M \, | \,  \mathrm{d}_N(p, \partial M) \leq \epsilon\}, 
\end{equation}
for $\epsilon > 0$ such that $B^{N\setminus M}_\epsilon (\partial M) \subset \Omega_N$  , where $\mathrm{d}_N$ is the distance induced by the metric $g_N$.
We are going to consider sets of finite perimeter $E_V$ in $(N, g_N)$ satisfying
\begin{equation}
\label{eq:isoproblem}
\abs{E_V\cap M}_f = V \qquad\qquad
P(E_V) = \inf\left\{P(F) \, | \, B^{N\setminus M}_\epsilon (\partial M) \subset F\subset \Omega_N,  \abs{F\cap M}_f = V\right\}
\end{equation}
for $V < \abs{\Omega}_f$, where we recall that given $E \subset M$ we defined
\begin{equation}
\label{eq:weighted}
\abs{E}_f = \int_E f d\mu.
\end{equation}
The following result gathers the main properties these constrained isoperimetric sets satisfy. 
\begin{theorem}[Existence and structure of constrained $f$-isoperimetric sets]
\label{thm:f-isoperimetric}
Let $(M, g, f)$ be a substatic triple with horizon boundary, of dimension $n \leq 7$. Let $S$ be a strictly mean-convex outward minimizing hypersurface homologous to $\partial M$, and let $(N, g_N), \Omega, \Omega_N$ and $B^N_\epsilon (\partial M)$ as above, for $\epsilon > 0$. Then, for any $V < \abs{\Omega}_f$, there exists $E_V\subset \Omega_N$ satisfying \cref{eq:isoproblem}.
Moreover, 
\begin{itemize}
\item[$(i)$]$\partial E_V \cap \partial M = \emptyset$. Moreover, $\partial (E_V \cap M) = \Sigma \sqcup \partial M$, where $\Sigma$ is a $\mathscr{C}^{1, 1}$-hypersurface. 
\item[$(ii)$]
The set $\Sigma \setminus S$ is a smooth hypersurface. Moreover, there exists a \emph{positive} constant $\lambda$ such that $\HHH (x) = \lambda f(x)$ for any $x \in \Sigma \setminus S$.
\item[$(iii)$] We have
\begin{equation}
\label{eq:supH/f}
\lambda \geq \frac{\HHH}{f}(x) > 0
\end{equation}
for $(n-1)$-almost any $x \in \Sigma$.
\end{itemize}
\end{theorem}
\begin{proof}
The existence of $E_V$ directly follows from the Direct Method. Indeed, let $(F_j)_{j\in \N}$ be a minimizing sequence for \cref{eq:isoproblem}. Then, by compactness, up to subsequences it converges to a set $E_V \subset \Omega_N$ in $L^1$. In particular, we have
\begin{equation}
\lim_{j \to + \infty} \int_{M}f\abs{\chi_{E_V} - \chi_{F_j}} d\mu \leq \lim_{j \to + \infty}\sup_\Omega f\abs{(E_V\bigtriangleup F_j)\cap M} = 0.
\end{equation}
So, $\abs{E_V \cap M} = V$. By the convergence almost everywhere one also deduces that $B_\epsilon^{N\setminus M}(\pa M) \subset E_V \subset \Omega_N$ is satisfied too. The lower semicontinuity of the perimeter also ensures that the infimum in \cref{eq:isoproblem} is attained by $E_V$. 

\smallskip

As far as the regularity of $\partial (E_V \cap  M)$ is concerned, let us first crucially observe that $E_V \cap M$ is (constrained) isoperimetric in $M$ endowed with the conformal metric $\overline{g} = f^2 g$ with respect to a perimeter and volume with the same weight, namely with respect to
\begin{equation}
\label{weighted}
{P}(E) = \int_{\partial^*E} f^{1-n} d\sigma_{\overline{g}} , \quad \abs{E}_f = \int_{E} f^{1-n} d \mu_{\overline{g}}, 
\end{equation}
where $d\sigma_{\overline{g}}$ and $d \mu_{\overline{g}}$ denote the area and volume measure induced by $\overline{g}$ respectively.
In particular, away from $S$ and $\partial M$, where $\overline{g}$ becomes singular, we have that classical regularity for the weighted isoperimetric problem applies \cite[Section 3.10]{Morgan}, and implies the claimed smoothness. In order to prove the global $\mathscr{C}^{1, 1}$-regularity,
we mainly follow the nice exposition in \cite[Section 6]{MondinoSpadaro}, taking advantage also of \cite[Section 17]{Maggi}. We first show that $E_V \cap M$ is an almost minimizer for the perimeter. This amounts to say that there exists $r_0$ such that for every $x \in\Omega$ and every $r < r_0$
\begin{equation}
\label{eq:almostminimizer}
P(E_V \cap M) \leq P(F) + \mathrm{C} r^n 
\end{equation}
holds for any $F$ such that $(E_V\cap M) \bigtriangleup F \Subset B(x, r)$, for some constant $\mathrm{C}$ independent of $x$ and $r$. Observe that $B(x, r)$ can intersect $\Omega_N \setminus \Omega$, and  this is the main reason  we did extend our substatic manifold. Let for simplicity $E = E_V \cap M$,  consider two small enough balls $B_1$ and $B_2$ centered on $\partial E \setminus \partial M$ with $B_1, B_2 \Subset M \setminus \partial M$, and let $X_1$ and $X_2$ be  variation vector fields compactly supported in $B_1$ and $B_2$ respectively. Let $E_t^i = \psi_t^i(E)$, where $\psi_t^i$ is the flow of $X_i$ at time $t$, for $i = 1, 2$. By \cite[Proposition 17.8]{Maggi}, we have
\begin{equation}
\label{expansionvolume}
\abs{E_t^i}_f = \abs{E} + t\int_{\partial E} f\langle X_i | \nu_E\rangle d\sigma + O(t^2) 
\end{equation}
as $t \to 0$, where $\nu_E$ is a unit normal for $E$. 
If $f > c> 0$ uniformly on $B_1 \cup B_2$, we deduce
\begin{equation}
\label{controlvolume}
\left|\abs{E_t^i} - \abs{E}\right| \geq \mathrm{C} \abs{t}
\end{equation}
for $t$ in some small neighbourhood of $0$ and for some uniform constant $\mathrm{C}$. Moreover,
the perimeter satisfies the usual expansion \cite[Theorem 17.5]{Maggi}
\begin{equation}
\label{expansionperimeter}
P(E^i_t) = P(E)  + t\int_{\partial E} \div_{\partial E} X_i d\sigma +O(t^2)
\end{equation}
as $t \to 0$,
where $\div_{\partial E}$ denotes the tangential divergence  $\div_{\partial E} X_i = \div (X_i) - \langle X_i | \nu_E\rangle$. Thus,
\begin{equation}
\label{controlperimeter}
\left|\abs{P(E^i_t)} - \abs{P(E)} \right| \leq \mathrm{C} \abs{t},
\end{equation}
again for $t$ in some small neighbourhood of $0$ and for some uniform constant $C$. We can now conclude the proof of \cref{eq:almostminimizer} for the suitable competitors $F$ as done for the proof of \cite[Lemma 6.3]{MondinoSpadaro}. Namely, letting $F$ as above, we recover the possibly lost or gained $f$-volume $\delta$ by the competitor $F \cap \Omega$ by slightly deforming $E$ inside $B_i$ with $i$ chosen so that $(F \bigtriangleup E) \cap B_i = \emptyset$. Observe that $\abs{\delta} \leq \mathrm{C} r^n$ for some suitable constant. Exploiting \cref{controlvolume} and \cref{controlperimeter} we get a set $\tilde{F}$ with  $\abs{E}_f = \abs{\tilde{F}}_f$ such that 
\begin{equation}
\label{volumerecovered} 
P(\tilde{F}) \leq P(F\cap \Omega) + \mathrm{C} \abs{\delta} \leq P(F\cap \Omega) + \mathrm{C} r^n  
\end{equation}
for some suitable $\mathrm{C} >0$, uniform for any $r < r_0$, with $r_0$ small enough. Since $E$ is constrained $f$-isoperimetric, we have then
\begin{equation}
\label{chainalmost}
P(E) \leq P(\tilde{F}) \leq P(F) + P(\Omega) - P(F \cup \Omega) + \mathrm{C} r^n.
\end{equation}
Observe that $F \cup \Omega$ may intersect $\Omega_N \setminus \Omega$.
On the other hand, it is easy to notice that sets with smooth boundary are in fact automatically almost minimizers for the perimeter (see e.g. the derivation of \cite[(6-9)]{MondinoSpadaro}), and so $P(\Omega) \leq P(F \cup \Omega) + \mathrm{C} r^n$. Plugging it into \cref{chainalmost} concludes the proof of $E= E_V \cap M$ being almost minimizing. From this crucial property, one deduces that $\partial E$ is $C^{1, 1/2}$ in a neighbourhood of $\partial E \cap \partial \Omega$ exactly as exposed in  \cite[Proof of Proposition 6.1]{MondinoSpadaro}.  

\medskip

To establish the optimal $\mathscr{C}^{1, 1}$ regularity, we first take advantage of~$ (ii)$, which we proceed to prove. In order to make the comparison with the references easier, along this proof we are going to assume that $\nu_E$ is the interior unit normal to $E$, in the extended manifold.
Let $Y$ be a vector field supported in some small ball centered at some point of $\partial E$, with a flow $\psi$ such that $\psi_t(E) \subset \Omega$, for $t$ small enough. Let now $X$ satisfy the same assumptions, with the additional requirement to be supported around the smooth part of $\partial E \setminus \partial \Omega$ and such that $\mathrm{supp} X \cap \mathrm{supp} Y = \emptyset$.  Assume also that the composition of the two flows gives a $f$-volume preserving diffeomorphism. Then, the first variation formula gives
\begin{equation}
\label{firstvariation}
0 \leq \int_{\partial E} \div_{\partial E} (Y + X) d\sigma = \int_{\partial E} \div_{\partial E} Y d\sigma - \int_{\partial E} \HHH\langle X | \nu_E \rangle d\sigma.
\end{equation}
Assume for the time being that $Y$ is supported around a point where $\partial E \setminus \partial \Omega$ is smooth. Then we can integrate by parts also in the integrand involving $Y$, and obtain
\begin{equation}
0 \leq - \int_{\partial E} \HHH\langle Y + X | \nu_E \rangle d\sigma.
\end{equation}
Repeating the argument with $-X$ and $-Y$ in place of $X$ and $Y$, that is possible in the present case since these vector fields are supported  away from $\partial \Omega$, we actually get
\begin{equation}
\label{firstvariationsmooth}
0 =  \int_{\partial E} \HHH\langle Y + X | \nu_E \rangle d\sigma.
\end{equation}
Moreover, the $f$-volume being preserved entails, by \cref{expansionvolume},
\begin{equation}
\label{variationvolume}
0 = \int_{\partial E} f \langle Y + X | \nu_E\rangle d \sigma. 
\end{equation}
Choose now $X$ and $Y$ so that $X = \alpha \nu_E$ and $Y = -\beta \nu_E$ on $\partial E$, with $\alpha$ and $\beta$ being smooth functions compactly supported on $\partial E$. Combining \eqref{variationvolume} with \cref{firstvariationsmooth} with this choice of $X$ and $Y$, we obtain
\begin{equation}
\frac{\int_{\partial E} \left(\frac{\HHH}{f}\right) f \alpha d\sigma}{\int_{\partial E} f\alpha d\sigma} = \frac{\int_{\partial E} \left(\frac{\HHH}{f}\right) f \beta d\sigma}{\int_{\partial E} f\alpha d\sigma} = \frac{\int_{\partial E} \left(\frac{\HHH}{f}\right) f \beta d\sigma}{\int_{\partial E} f\beta d\sigma}. 
\end{equation}
Then, since the support of $\alpha$ and $\beta$ on $\partial E$ can be chosen arbitrarily close to any two points in the smooth part of $\partial E \setminus \partial \Omega$, we conclude that there exists $\lambda \in \R$ such that $\HHH = \lambda f$ in the smooth part of $\partial E \setminus \partial \Omega$. We plug this information into \cref{firstvariation}, for a vector field $Y$ that is now supported in a ball centered on a point of $\partial E \cap \partial \Omega$. Coupling with \cref{variationvolume}, this yields
\begin{equation}
\label{eq:diffineq}
  \int_{\partial E} \div_{\partial E} Y d\sigma + \lambda \int_{\partial E} f \langle Y | \nu_E \rangle d\sigma \geq 0.   
\end{equation}
Writing \cref{eq:diffineq} in local coordinates in a neighbourhood  of a point in $\partial E\cap \partial \Omega$, where $\partial E$ is given by the graph of a function $u: B \to \R$ and $\partial \Omega$ as the graph of a function $\psi : B \to \R$ with $B \subset \R^{n-1}$, and $Y$ as a normal vector field cut off with a function $\phi$, it is a routine computation to check that, for some quasilinear elliptic operator $L$, it holds
\begin{equation}
\label{eq:variational-obstacle}
\int_B \phi [-L u + \lambda f(\cdot, u(\cdot))] d\mu \geq 0, 
\end{equation}
where $\phi \in \mathscr{C}^{\infty}_c (B)$ and $\phi \geq 0$ in $\{u = \psi\}$. We address the interested reader to \cite[Section 6C]{MondinoSpadaro} and \cite[Section 4.1]{FocardiGeraciSpadaro} for details of these computations. In particular, the function $u$ succumbs to the regularity theory for obstacle problems, that is $u \in \mathscr{C}^{1,1}$, see \cite[Theorem 3.8]{FocardiGeraciSpadaro}. As a consequence, $\partial E$ has a notion of mean curvature defined almost everywhere. Observe that the quasilinear elliptic operator $L u$ provides the mean curvature of $\partial E$ at the point $(x, u(x))$, with $x \in B$.
We are going to take advantage also of the basic step leading to the regularity result recalled above. Namely, as nicely presented in \cite[Proposition 3.2]{FocardiGeraciSpadaro}, the variational property \cref{eq:variational-obstacle} implies that $u$ is also  solution to the Euler--Lagrange equation
\begin{equation}
\label{eq:euler-lagrange-obstacle}
    \int_B \phi [-L u + \lambda f(\cdot, u(\cdot))] d\mu = \int_B \xi \,  d\mu,
\end{equation}
given any $\phi \in \mathscr{C}^{\infty}_c (B)$, where
\begin{equation}
0 \leq \xi \leq  [-L \psi + \lambda f(\cdot, \psi(\cdot))]^+ \chi_{\{u = \psi\}}.
\end{equation}

\medskip

We now proceed to show that $\lambda > 0$, that $\partial E_V$ is disjoint from $\partial M$ and that \cref{eq:supH/f} holds, completing thus the proof.
Let $Y$ in \cref{eq:diffineq} be supported on a neighbourhood of a point $\partial E \cap S$. Integrating by parts the first summand in \cref{eq:diffineq}, and letting $Y = \alpha \nu_E$ for some compactly supported nonnegative test function $\alpha$, we get
\begin{equation}
\label{eq:supH/f-passaggio}
\int_{\partial E} (\lambda f - \HHH) \alpha \, d \sigma \geq 0.
\end{equation}
The arbitrariness of $\alpha$ implies that 
\begin{equation}
\label{eq:supH/f-passaggio2}
\lambda \geq \frac{\HHH}{f} (x)  \qquad \text{for $(n-1)$-almost any $x \in \partial E \setminus \partial M$}.
\end{equation}
Since $S$ is mean-convex, if the $(n-1)$-induced measure of $\partial E \cap S$ is strictly positive, then the (weak) mean curvature of such region is strictly positive \cite[Lemma 6.10]{MondinoSpadaro}, and consequently \cref{eq:supH/f-passaggio2} directly implies $\lambda > 0$ and \cref{eq:supH/f}. If instead the intersection is $(n-1)$-negligible, then \cref{eq:euler-lagrange-obstacle} implies that $Lu = \lambda f$ holds in the weak sense outside of $\partial M$. In particular, classical regularity theory implies that $\partial E \setminus \partial M$ is smooth and that its mean curvature is given by $\lambda f$. Assume by contradiction that $\lambda < 0$. Hence, by the Maximum Principle, $\partial E \cap S = \emptyset$. Then, $S$ being outward minimizing acts as a barrier to minimize the perimeter among sets homologous to $\partial M$ containing $E$ \cite[Theorem 2.10]{Fogagnolo_Mazzieri-minimising}. We call $E^*$ such a minimizer. The boundary of such set is $\mathscr{C}^{1,1}$ \cite{SternbergZiemerWilliams}, and obviously it is outward minimizing, so that its (weak) mean curvature is nonnegative. Having assumed that $\lambda < 0$, and since $\partial M$ is minimal, we deduce that $\partial E^*$ is minimal itself, and by the Maximum Principle disjoint from $\partial M$. However, this is a contradiction with $\partial M$ being outermost, proved in \cref{prop:outermost}. 
We established that $\lambda \geq 0$, and that $\lambda > 0$ if $\partial E \cap S$ has positive $(n-1)$-measure.  We focus our attention to \cref{eq:euler-lagrange-obstacle} again, considering a neighbourhood of a point where $\partial E$ meets $\partial M$. Crucially observe that by the minimality of $\partial M = \{f = 0\}$ the right hand side of \cref{eq:euler-lagrange-obstacle} in this case vanishes, and that thus by classical regularity $\partial E$ is smooth in a neighbourhood of $\partial M$. Since its mean curvature $\HHH = \lambda f \geq 0$, the Maximum Principle implies that $\partial E$ is disjoint from $\partial M$. The possibility that $\lambda = 0$ in the case of negligible intersection $\partial E \cap S$ is finally ruled out by  \cref{prop:outermost} again, and \cref{eq:supH/f} becomes simply \cref{eq:supH/f-passaggio2}. 
\end{proof}
The Willmore-type inequality \cref{eq:willmore-intro} and the description of $f$-isoperimetric sets provided by \cref{thm:f-isoperimetric} allow to carry out the proof of the Isoperimetric Inequality of \cref{thm:isoperimetric_intro}.
\begin{proof}[Proof of \cref{thm:isoperimetric_intro}]
If there is more than one end, then by \cref{lem:positiveavr->1end} all ends have vanishing asymptotic volume ratio. In this case,
\cref{eq:isoperimetric_intro} reduces to the just proved \cref{eq:minimizingboundary}. 
Obviously, the same is true for the one-ended case if
 $\mathrm{AVR}(M, g, f) = 0$. The core of the Theorem then lies in the one-ended case with  $\mathrm{AVR}(M, g, f) > 0$. 

\smallskip

We first carry out the proof in the more involved and relevant case of nonempty boundary. 
Let $S$ be one of the outward minimizing hypersurfaces in the outward minimizing exhaustion, and let $\Omega$ be such that $\partial \Omega = \partial M \sqcup S$. For $V < \abs{\Omega}_f$, consider an $f$-isoperimetric set $E_V$ constrained in $\Omega$ with $f$-volume equal to $V$, that is, satisfying \cref{eq:isoproblem}. $E_V$ exists and is subject to the properties described in \cref{thm:f-isoperimetric}. In particular, $\partial E_V = \partial M \sqcup \Sigma_V$, with $\Sigma_V$ a $\mathscr{C}^{1, 1}$ hypersurface. Varying $V$, we also define $I_f: (0, \abs{\Omega}_f) \to (0, +\infty)$ by $I_f (V) = \abs{\Sigma_V}$, the $f$-isoperimetric profile of $\Omega$. It is argued as in the classical case that $I_f$ is continuous. Indeed, $E_{V + \epsilon}$ converges in $L^1$ to some $\tilde{E}$, that in particular satisfies $\abs{\tilde{E}}_f = V$. We have that, for a fixed $\delta > 0$, by lower semicontinuity, for any $\epsilon$ close enough to zero such that
\begin{equation}
\label{eq:profilecont1}
I_f(V) \leq  P(\tilde{E}) \leq P(E_{V + \epsilon}) + \delta = I_f(V+\epsilon) + \delta  \leq P(E_V) + P(B_\epsilon) +\delta = I_f(V) + P(B_\epsilon) +\delta.
\end{equation}
In the above inequality, $B_\epsilon$ is chosen so that $\abs{F \cup B_\epsilon}_f = V + \epsilon$ or $\abs{F \setminus B_\epsilon}_f = V + \epsilon$, according to the sign of $\epsilon$. Letting first $\epsilon \to 0$, and then $\delta \to 0^+$ establishes the continuity of $I_f$. Let now $\epsilon > 0$. Let $\Sigma^\epsilon$ be an inward variation of $\Sigma_V$ supported in $\Sigma_V \setminus S$ such that $\abs{E_V^\epsilon}_f = V - \epsilon$,  where $E_V^\epsilon$ is such that $\partial E_V^\epsilon = \partial \Sigma^\epsilon \sqcup \partial M$. We have
\begin{equation}
\label{eq:derivative-profile}
\liminf_{\epsilon \to 0^+}\frac{I_f^{\frac{n}{n-1}} (V) - I_f^{\frac{n}{n-1}}(V-\epsilon)}{\epsilon} \geq \liminf_{\epsilon \to 0^+}\frac{\abs{\Sigma}^{\frac{n}{n-1}} - \abs{\Sigma^\epsilon}^{\frac{n}{n-1}}}{\epsilon} 
\end{equation}
Assume now that $\Sigma^\epsilon$ is obtained through a normal variation field coinciding with $\phi \nu$ on $\Sigma$, with $\phi \in \mathscr{C}^{\infty}_c(\Sigma)$.
Since the first variation of $f$-volume is given by the $f$-weighted area, the right hand side is computed as 
\begin{equation}
\label{eq:derivative-profile2}
\liminf_{\epsilon \to 0^+}\frac{\abs{\Sigma}^{\frac{n}{n-1}} - \abs{\Sigma^\epsilon}^{\frac{n}{n-1}}}{\epsilon} =\frac{n}{n-1} \abs{\Sigma_V}^{\frac{1}{n-1}}\frac{\int_\Sigma \HHH \phi d\sigma}{\int_\Sigma f \phi d\sigma} = \frac{n}{n-1} \abs{\Sigma_V}^{\frac{1}{n-1}} \lambda,
\end{equation}
where $\HHH = \lambda f$ on the support of $\phi$ is due to \cref{thm:f-isoperimetric}.
Letting now $W$ be the infimum of $\int_\Sigma \HHH / f d\sigma$ taken among strictly mean-convex smooth hypersurfaces $\Sigma$ homologous to $\partial M$, we actually have, by the Substatic Willmore-type inequality \cref{eq:willmore-intro}
\begin{equation}
\label{eq:willmore-applied}
\AVR(M, g, f) \abs{\Sf^{n-1}} \leq W \leq \frac{1}{(n-1)^{n-1}}\int_{\Sigma_V} \left(\frac{\HHH}{f}\right)^{n-1} d\sigma \leq \frac{1}{(n-1)^{n-1}} \abs{\Sigma_V} \lambda^{n-1},  
\end{equation}
since $0 < \HHH/f \leq \lambda$ by \cref{thm:f-isoperimetric} for $(n-1)$-almost any point on $\Sigma_V$. The above inequality holds for the $\mathscr{C}^{1,1}$ $\Sigma_V$ since we can approximate it with smooth, strictly mean-convex hypersurfaces through Mean Curvature Flow, see \cite[Lemma 5.6]{huisken_inverse_2001}. Observe that this is possible since, by \cref{thm:f-isoperimetric}, $\Sigma_V$ is disjoint from $\partial M$. Combining \cref{eq:derivative-profile}, \cref{eq:derivative-profile2} and \cref{eq:willmore-applied} yields
\begin{equation}
\label{eq:derivativeprofilewillmore}
\liminf_{\epsilon \to 0^+}\frac{I_f^{\frac{n}{n-1}} (V) - I_f^{\frac{n}{n-1}}(V-\epsilon)}{\epsilon} \geq n \left[\AVR(M, g, f) \abs{\Sf^{n-1}}\right]^{\frac{1}{n-1}}.
\end{equation}
Comparing with the reference warped product $f$-isoperimetric profile given by 
\begin{equation}
\label{eq:referencewarped}
J_f(V) = n^{\frac{n-1}{n}}\left[\AVR(M, g, f) \abs{\Sf^{n-1}}\right]^{\frac{1}{n}} V^{\frac{n-1}{n}},
\end{equation}
whose derivative equals the right-hand side of \cref{eq:derivativeprofilewillmore},
we deduce at once that the continuous function $I_f^{{n}/{(n-1)}}  - J_f^{{n}/{(n-1)}}$ has nonnegative Dini derivative, and is thus monotone nondecreasing. Hence, for $V_0 < V$, we get
\begin{equation}
    \label{eq:comparingprofiles}
    I_f^{\frac{n}{n-1}}(V) \geq  n \left[\AVR(M, g, f) \abs{\Sf^{n-1}}\right]^{\frac{1}{n-1}} V + I_f^{\frac{n}{n-1}}(V_0) - n \left[\AVR(M, g, f) \abs{\Sf^{n-1}}\right]^{\frac{1}{n-1}} V_0.
\end{equation}
Recall now that $I_f(V_0) = \abs{\Sigma_{V_0}}$ for some $\Sigma_{V_0}$ homologous to $\partial M$. The boundary being minimizing, by \cref{prop:outermost}, implies then that $I_f (V_0) \geq \abs{\partial M}$. Plugging it into \cref{eq:comparingprofiles}, and then letting $V_0$ go to $0$, leaves us with
\begin{equation}
\label{eq:almostinequality}
    I_f^{\frac{n}{n-1}}(V) - \abs{\partial M}^{\frac{n}{n-1}} \geq n \left[\AVR(M, g, f) \abs{\Sf^{n-1}}\right]^{\frac{1}{n-1}} V. 
\end{equation}
By definition of Isoperimetric profile, the above inequality implies \cref{eq:isoperimetric_intro} for any hypersurface homologous to $\partial M$ inside $\Omega$, enclosing a set of volume $V$. But the volumes being arbitrary and the outward minimizing envelopes forming an exhaustion, the proof of \cref{eq:isoperimetric_intro} is actually complete. 

\smallskip

We are left to characterize the situation when some smooth $\Sigma$ homologous to $\partial M$ fulfils the equality in \cref{eq:isoperimetric_intro}.
Let $V$ be the $f$-volume subtended by $\Sigma$.  Let $\Omega = \partial M \sqcup S$, with $S$ strictly mean-convex outward minimizing and $\Sigma\subset\Omega$. As above, let $I_f$ be the $f$-isoperimetric profile of $\Omega$ and $J_f$ the reference warped product $f$-isoperimetric profile defined in \cref{eq:referencewarped}. By approximation, we observe that any other $f$-isoperimetric set constrained in $\Omega$ of $f$-volume $V_0$ satisfies the Isoperimetric inequality \cref{eq:isoperimetric_intro}. Hence, by \cref{eq:comparingprofiles}, we have
\begin{equation}
\label{eq:profiles-rigidity}
 \abs{\partial M}^{\frac{n}{n-1}} = I_f^{\frac{n}{n-1}}(V) - J_f^{\frac{n}{n-1}}(V) \geq I_f^{\frac{n}{n-1}}(V_0) - J_f^{\frac{n}{n-1}}(V_0) \geq \abs{\partial M}^{\frac{n}{n-1}}.   
\end{equation}
As a consequence, any $f$-isoperimetric set of volume $V_0 \leq V$ satisfies the equality in the $f$-Isoperimetric inequality \cref{eq:isoperimetric_intro}. Observe now, that by approximation with smooth sets in the possibly extended Riemannian manifold $(N, g_N)$, this implies that such constrained $f$-Isoperimetric sets of volume $V_0$ are in fact \emph{globally} $f$-Isoperimetric, and consequently the regularity observed in \cref{thm:f-isoperimetric} implies that any $\Sigma_{V_0}$ is smooth. Retracing the steps that lead to \cref{eq:comparingprofiles}, we have that the smooth hypersurface $\Sigma_{V_0}$ satisfies the equality in the Willmore-type inequality in  \cref{thm:Willmore_intro}. This triggers the rigidity stated there, and yields, for $\Omega_{V_0}$ the domain enclosed between $\Sigma_{V_0}$ and $\partial M$, the isometry between $(M \setminus \Omega_{V_0}, g)$ and $[{s}_0, +\infty) \times \Sigma$ endowed with
\begin{equation}
\label{eq:rigiditywillmoreiso}
   g = f^2 d\rho\otimes d\rho + \eta^2 g_{\Sigma_{V_0}}. 
\end{equation}
In particular, since $M$ has one end, the hypersurface $\Sigma_{V_0}$ is necessarily connected. 
We now observe that, again due to the global $f$-isoperimetry of $\Sigma_{V_0}$, the value of $\HHH/f$ is constant on such hypersurface. But then, retracing the computations that lead to the isometry with \cref{eq:rigiditywillmoreiso}, more precisely coupling \cref{eq:H/f_eta} with \cref{eq:aux_rigidity_willmore}, we deduce that $f$ and $\eta$ in \cref{eq:rigiditywillmoreiso} depend only on $\rho$. Introduce now a new coordinate $s$ defined by $f^2 (\rho) d\rho = d s$. Recall that $\eta$ satisfies $\partial_\rho \eta = f^2$, and thus $\partial_s \eta = 1$. Possibly translating the variable $s$, we thus have
\begin{equation}
\label{eq:rigidityins}
    g = \frac{ds \otimes ds}{f^2(s)} + s^2 g_{\Sigma_{V_0}},
\end{equation}
for any $V_0 \leq V$, and $s \geq s_0$ for some $s_0 > 0$. 

Since we have proved that $f$ is a function of the distance $\rho$ from $\Sigma_{V_0}$ only, in particular we have shown that $f$ must be constant on $\Sigma_{V_0}$. This must hold for all $V_0\leq V$.
Moreover,  the (Hausdorff) distance between $\Sigma_{V_0}$ and $\partial M$ goes to $0$ as ${V_0}\to 0$, because otherwise the volume enclosed along the (sub)sequence would be necessarily bounded away from $0$. 
But then, the level sets of $f$ forming a regular foliation of a neighbourhood of $\partial M$, we deduce that $\Sigma_{V_0}$ must actually be a level set of $f$, in particular diffeomorphic to $\partial M$, for $V_0$ small enough. Letting $V_0$ to zero we thus extend the expression \cref{eq:rigidityins} to the whole manifold, that is \cref{eq:rigidiso_intro}. The connectednedness of $\partial M$ is again a consequence of $(M, g)$ being one ended. 

As far as the characterization of $\Sigma$ is concerned, we already showed that $f$ is constant on it. 
If, by contradiction, $s$ were not constant on $\Sigma$, then, letting $s_{\mathrm{min}} = \min\{s(p): p \in \Sigma\}$ and $s_{\mathrm{max}} = \max\{s(p): p \in \Sigma\}$, $\Sigma$ would lie in the region $[s_{\mathrm{min}}, s_{\mathrm{max}}] \times \partial M$, where, $f$ being constant, by \cref{eq:rigidiso_intro} has the metric of a truncated cone. By \cref{eq:rigidityins} for $V_0 = V$, $\Sigma$ is a totally umbilical constantly mean-curved hypersurface in such cone. Moreover,
the constancy of $f$ on such region makes the substatic condition simplify to nonnegative Ricci curvature.
By \cite[Lemma 3.8]{morgan-ritore}, $\Sigma$ could then only be a level set of $s$ or bound a flat round ball. The first possibility gives a contradiction with the initial assumption that $s$ were not constant on $\Sigma$, the second one with $\Sigma$ being homologous to $\partial M$. This concludes the proof of $\Sigma$ being a level set of $s$, and of \cref{thm:isoperimetric_intro} in the nonempty-boundary case.

\smallskip

We finally discuss the empty-boundary-case. It is immediately checked that the $f$-Isoperimetric inequality \cref{eq:isoperimetric_intro} follows with a pure simplification of the proof given above. When a hypersurface $\Sigma$ satisfies with equality the $f$-Isoperimetric inequality, arguing as done above for \cref{eq:rigidityins} we reach for an isometry between $(M \setminus \Omega_\Sigma, g)$ and $I = [\overline{s}, +\infty) \times \Sigma$ endowed with
\begin{equation}
\label{eq:rigiditypartial-nobordo}
     g = \frac{ds \otimes ds}{f^2(s)} + s^2 g_{\Sigma},
\end{equation}
for $\Omega_\Sigma$ enclosed by $\Sigma$. Again, $\Sigma$ must be connected, since $M$ is one ended by \cref{lem:positiveavr->1end}.
Now, we claim that $\Sigma$  satisfies
\begin{equation}
\label{eq:HK-equality}
\frac{n-1}{n} \int_\Sigma \frac{f}{\HHH} d\sigma = \int_{\Omega_\Sigma} f d\mu, 
\end{equation}
in fact saturating the substatic Heintze-Karcher inequality \cite[Theorem 1.3]{Li_Xia_19} (see also \cite[Theorem 3.6]{fogagnolo-pinamonti}) in boundaryless substatic manifolds. The analysis of the equality case worked out in \cite[Theorem 3.1-$(ii)$]{borghini-fogagnolo-pinamonti} then provides us with an isometry between $(\Omega_\Sigma, g)$ and $I \times \Sf^{n-1}$ endowed with
\begin{equation}
\label{eq:rigidity-interior}
   g = \frac{ds \otimes ds}{f^2(s)} + \left(\frac{s}{f(x)}\right)^2 g_{\Sf^{n-1}}, 
\end{equation}
for $x \in \Omega_\Sigma$
with $\Sigma$ becoming a level set of $s$. Coupled with \cref{eq:rigiditypartial-nobordo} on the complement of $\Omega$, this yields the desired rigidity statement.

In order to check \cref{eq:HK-equality}, just observe that, since as above one has that $\Sigma$ is $f$-isoperimetric, $\HHH/f$ \emph{is constant on such hypersurface}. Moreover, since it satisfies equality in the Willmore-type inequality \cref{eq:willmore-intro}, one has 
\begin{equation}
\label{eq:idmeancurv}
\frac{\HHH}{f} \,\abs{\Sigma}^{\frac{1}{n-1}} =(n-1)\left[\mathrm{AVR}(M, g, f) \abs{\Sf^{n-1}}\right]^{\frac{1}{n-1}}.
\end{equation}
Coupling with
\begin{equation}
\label{eq:identity-isop}
\abs{\Sigma}^{\frac{n}{n-1}} = n \left[\mathrm{AVR}(M, g, f) \abs{\Sf^{n-1}}\right]^{\frac{1}{n-1}} \abs{\Omega_\Sigma}_f,
\end{equation}
it is straightforwardly seen that \cref{eq:HK-equality} holds, completing the proof.
\end{proof}

\appendix

\section{Comments on the substatic condition}

\subsection{Physical motivation}
\label{app:phys_mot}

Here we give a physical interpretation of substatic triples, following~\cite[Lemma~3.8]{Wang_Wang_Zhang}.
Let 
$$
L=\R\times M\,,\qquad \gr\,=\,-f^2dt\otimes dt+g
$$
be a static spacetime satisfying the Einstein Field Equation
$$
\Ric_\gr\,+\,\left(\Lambda-\frac{1}{2}\RRR_\gr\right)\gr\,=\,T\,,
$$
where $T$ is the stress-energy tensor and $\Lambda\in\R$ is the cosmological constant. Using standard formulas to express the Ricci tensor of a warped product, we find out that
$$
\Ric_\gr(\pa_t,\pa_t)\,=\,f\De f\,,\qquad \Ric_\gr(\pa_i,\pa_j)\,=\,\Ric(\pa_i,\pa_j)-\frac{1}{f}\nana f(\pa_i,\pa_j)\,.
$$
In particular a simple computation gives
$$
\RRR_\gr\,=\,\RRR-\frac{2}{f}\De f\,.
$$
Putting these pieces of information inside the Einstein Field Equation, we get
\begin{align}
T_{tt}\,&=\,\left(-\Lambda +\frac{\RRR}{2}\right)f^2\,,
\\
T_{it}\,&=\,0\,,
\\
T_{ij}\,&=\,\RRR_{ij}-\frac{1}{f}\nana_{ij} f+\left(\Lambda-\frac{\RRR}{2}+\frac{\De f}{f}\right)g_{ij}\,.
\end{align}
We now assume that the Null Energy Condition is satisfied. Namely, for any vector $X=\pa_t+Y^i\pa_i$ with $\gr(X,X)=0$ (that is, $g(Y,Y)=f^2$), we require $T(X,X)=T_{tt}+T_{ij}Y^i Y^j\geq 0$.
Using the above identities, this hypothesis tells us
$$
0\,\leq\,T_{tt}+T_{ij}Y^i Y^j\,=\,\left(-\Lambda +\frac{\RRR}{2}\right)f^2+\left(\Ric-\frac{1}{f}\nana f+\frac{\De f}{f}g\right)(Y,Y)+\left(\Lambda -\frac{\RRR}{2}\right)g(Y,Y)\,.
$$
Recalling that $g(Y,Y)=f^2$, we have obtained
$$
\hbox{ NEC holds}\ \ \Leftrightarrow\ \  \left(\Ric-\frac{1}{f}\nana f+\frac{\De f}{f}g\right)(Y,Y)\geq 0 \ \hbox{ for all $Y$ with $g(Y,Y)=f^2$.}
$$
By rescaling of $Y$, we then conclude that the Null Energy Condition on static spacetimes is equivalent to
$$
\Ric-\frac{1}{f}\nana f+\frac{\De f}{f}g\,\geq\,0\,.
$$
In other words, a static spacetime satisfies the Null Energy Condition if and only if its spacelike slices are substatic.

Finally, we briefly discuss the physical interpretation of the conformal metric $\tilde g=g/f^2$. In the context of static spacetimes, this metric is usually referred to as optical metric and has the property that $\tilde g$-geodesics lift to null geodesics in the spacetime metric $\gr$.
This follows easily from the fact that the trajectories of null geodesics do not change under a conformal change of metric, hence the null geodesics of $\gr$ are the same as the null geodesics of $f^2\gr=-dt\otimes dt+\tilde g$.

\subsection{Relation between ${\rm CD}(0,1)$ and substatic condition}
\label{app:CD_subst}

Let $(M,g,f)$ be a substatic triple and let $\tilde g=g/f^2$. 
We want to show that $(M,\tilde g,\psi)$ satisfies the ${\rm CD}(0,1)$ condition, where $\psi=-(n-1)\log f$.
To this end, we need to rewrite the substatic condition in terms of the conformal metric. We start from the following formulas:
\begin{align}
\widetilde\na^2 f\,&=\,\nana f+\frac{1}{f}\left(2df\otimes df-|\na f|^2 g\right)\,,
\\
\De_{\tilde g} f\,&=\,f^2\De f-(n-2)f|\na f|^2\,,
\end{align}

In particular
\begin{align}
\frac{1}{f}\nana f-\frac{1}{f}\De f g\,&=\,\frac{1}{f}\widetilde\na^2 f-\frac{1}{f^2}\left(2df\otimes df-|\na f|^2 g\right)
-\frac{1}{f}\De_{\tilde g} f\,\tilde g-(n-2)|\na f|^2\tilde g
\\
&=\,\frac{1}{f}\widetilde\na^2 f-\frac{2}{f^2}df\otimes df
-\frac{1}{f}\De_{\tilde g} f\tilde g-(n-3)\frac{1}{f^2}\big|\widetilde\na f\big|_{\tilde g}^2\,\tilde g
\end{align}

On the other hand, it is well known that the Ricci tensor $\Ric$ of $g$ and the Ricci tensor $\Ric_{\tilde g}$ of the conformal metric $\tilde g$ are related as follows
\[
\Ric\,=\,\Ric_{\tilde g}-\frac{n-2}{f}\widetilde\na^2 f+\frac{2(n-2)}{f^2}df\otimes df-\left(\frac{1}{f}\De_{\tilde g} f+\frac{n-3}{f^2}|\widetilde\na f|^2_{\tilde g}\right)\tilde g\,.
\]
Putting together the above formulas, we get
\begin{align}
\Ric-\frac{1}{f}\nana f+\frac{1}{f}\De f g\,&=\,
\Ric_{\tilde g}-\frac{n-1}{f}\widetilde\na^2 f+\frac{2(n-1)}{f^2}df\otimes df
\\
&=\,\Ric_{\tilde g}-(n-1)\widetilde\na^2 (\log f)+(n-1)d\log f\otimes d\log f
\\
&=\,\Ric_{\tilde g}+\widetilde\na^2 \psi+\frac{1}{n-1}d\psi\otimes d\psi\,.
\end{align}
It follows then that, if $(M,g,f)$ satisfies the substatic condition, then $(M,\tilde g,\psi)$ satisfies the ${\rm CD}(0,1)$ condition. 

\subsection{Li--Xia connections}
\label{app:LiXia}

In~\cite{Li_Xia_17}, Li and Xia consider the family of connections $\D^{u\a\gamma}$, where $u\in\mathscr{C}^\infty(M)$, $\a,\gamma\in\R$, defined by
\[
\D^{u\a\gamma}_X Y=\nabla_X Y+\alpha \left[X(u)Y+Y(u)X\right]+\gamma\, g(X,Y)\na u\,.
\]
They then compute the Ricci tensor $\Ric^{u\a\gamma}$ induced by a connection $\D^{u\a\gamma}$, showing that it is related to the usual Ricci tensor by
\[
\Ric^{u\a\gamma}=\Ric-\left[(n-1)\a+\gamma\right]\nana u+\left[(n-1)\a^2-\gamma^2\right]du\otimes du+\left[\gamma\De u+\left(\gamma^2+(n-1)\a\gamma\right)|\na u|^2\right]g\,.
\]
When $\a=0$ and $\gamma=1$, in particular we have
\[
\Ric^{u01}\,=\,\Ric-\nana u-du\otimes du+\left[\De u+|\na u|^2\right]g\,,
\]
which can be rewritten as follows by setting $u=\log f$:
\[
\Ric^{u01}\,=\,\Ric-\frac{1}{f}\nana f+\frac{\De f}{f}g\,.
\]
It is then clear that the condition $\Ric^{u01}\geq 0$ is equivalent to the substatic condition.

Choosing instead $\a=1/(n-1)$, $\gamma=0$, setting $v=-u$ one gets
\[
\Ric^{u\frac{1}{n-1}0}=\Ric+\nana v+\frac{1}{n-1}dv\otimes dv\,,
\]
hence $\Ric^{u\frac{1}{n-1}0}\geq 0$ gives  the ${\rm CD}(0,1)$ condition. In fact, the connection $\D^{u\frac{1}{n-1}0}$ had already been considered in the work~\cite{wylie-yeroshkin} that was focused on the ${\rm CD}(0,1)$ case only.

Here we show that the two connections $\D^{u01}$ and $\D^{u\frac{1}{n-1}0}$ are in fact conformally related: let $(M,g)$ be a Riemannian manifold and let $\na,\widetilde\na$ be the Levi-Civita connections corresponding to the metrics $g$, $\tilde g=g/f^2$, respectively. It is easy to show that $\na$ and $\widetilde\na$ are related a follows
\[
\na_X Y\,=\,\widetilde\na_XY+\frac{1}{f}\left[X(f)Y+Y(f) X- g(X,Y)\na f\right]\,=\,
\widetilde\na_XY+X(u)Y+Y(u) X-g(X,Y)\na u
\,,
\]
hence, setting $\psi=-(n-1)u$:
\begin{align}
\D_X^{u01}Y\,&=\,\na_XY+g(X,Y)\na u
\\
&=\,\widetilde\na_XY+X(u)Y+Y(u) X
\\
&=\,\widetilde\na_XY+\frac{1}{n-1}\left [X(-\psi)Y+Y(-\psi) X\right]\,.
\end{align}
This is then precisely $\D^{-\psi\frac{1}{n-1}0}$ using $\widetilde\na$ as the Levi-Civita connection in place of $\na$.

Notice that in this subsection $\psi$ and $f$ have been introduced as the functions satisfying $u=\log f$ and $\psi=-(n-1)u$, hence they are related by $\psi=-(n-1)\log f$, in agreement with Subsection~\ref{app:CD_subst}.


\printbibliography

@article {White-mcf,
    AUTHOR = {White, Brian},
     TITLE = {The size of the singular set in mean curvature flow of
              mean-convex sets},
   JOURNAL = {J. Amer. Math. Soc.},
  FJOURNAL = {Journal of the American Mathematical Society},
    VOLUME = {13},
      YEAR = {2000},
    NUMBER = {3},
     PAGES = {665--695},
      ISSN = {0894-0347},
   MRCLASS = {53C44 (49Q20)},
  MRNUMBER = {1758759},
MRREVIEWER = {Harold Parks},
       DOI = {10.1090/S0894-0347-00-00338-6},
       URL = {https://doi.org/10.1090/S0894-0347-00-00338-6},
}

@book {Lee_book,
    AUTHOR = {Lee, Dan A.},
     TITLE = {Geometric relativity},
    SERIES = {Graduate Studies in Mathematics},
    VOLUME = {201},
 PUBLISHER = {American Mathematical Society, Providence, RI},
      YEAR = {2019},
     PAGES = {xii+361},
      ISBN = {978-1-4704-5081-6},
   MRCLASS = {83-01 (83C05 83C57)},
  MRNUMBER = {3970261},
MRREVIEWER = {Stephen McCormick},
       DOI = {10.1090/gsm/201},
       URL = {https://doi.org/10.1090/gsm/201},
}

@article{mccormick,
  title={On a {M}inkowski-like inequality for asymptotically flat static manifolds},
  author={McCormick, Stephen},
  journal={Proceedings of the American Mathematical Society},
  volume={146},
  number={9},
  pages={4039--4046},
  year={2018}
}

@article{wei,
  title={On the {M}inkowski-type inequality for outward minimizing hypersurfaces in Schwarzschild space},
  author={Wei, Yong},
  journal={Calculus of Variations and Partial Differential Equations},
  volume={57},
  pages={1--17},
  year={2018},
  publisher={Springer}
}

@misc{Har_Wan,
      title={A rigidity theorem for asymptotically flat static manifolds and its applications}, 
      author={Brian Harvie and Ye-Kai Wang},
      year={2023},
      eprint={2305.08570},
      archivePrefix={arXiv},
      primaryClass={math.DG}
}

@article {morgan-ritore,
    AUTHOR = {Morgan, Frank and Ritor\'{e}, Manuel},
     TITLE = {Isoperimetric regions in cones},
   JOURNAL = {Trans. Amer. Math. Soc.},
  FJOURNAL = {Transactions of the American Mathematical Society},
    VOLUME = {354},
      YEAR = {2002},
    NUMBER = {6},
     PAGES = {2327--2339},
      ISSN = {0002-9947},
   MRCLASS = {53C42 (49Q20)},
  MRNUMBER = {1885654},
MRREVIEWER = {Constantin Vernicos},
       DOI = {10.1090/S0002-9947-02-02983-5},
       URL = {https://doi.org/10.1090/S0002-9947-02-02983-5},
}

@book {bray-thesis,
    AUTHOR = {Bray, Hubert Lewis},
     TITLE = {The {P}enrose inequality in general relativity and volume
              comparison theorems involving scalar curvature},
      NOTE = {Thesis (Ph.D.)--Stanford University},
 PUBLISHER = {ProQuest LLC, Ann Arbor, MI},
      YEAR = {1997},
     PAGES = {103},
      ISBN = {978-0591-60594-5},
   MRCLASS = {Thesis},
  MRNUMBER = {2696584},
       URL =
              {http://gateway.proquest.com/openurl?url_ver=Z39.88-2004&rft_val_fmt=info:ofi/fmt:kev:mtx:dissertation&res_dat=xri:pqdiss&rft_dat=xri:pqdiss:9810085},
}

@article {fogagnolo-pinamonti,
    AUTHOR = {Fogagnolo, Mattia and Pinamonti, Andrea},
     TITLE = {New integral estimates in substatic {R}iemannian manifolds and
              the {A}lexandrov theorem},
   JOURNAL = {J. Math. Pures Appl. (9)},
  FJOURNAL = {Journal de Math\'{e}matiques Pures et Appliqu\'{e}es. Neuvi\`eme S\'{e}rie},
    VOLUME = {163},
      YEAR = {2022},
     PAGES = {299--317},
      ISSN = {0021-7824},
   MRCLASS = {49Q10 (49J40 53C21 53C24 58J32 83C20)},
  MRNUMBER = {4438902},
       DOI = {10.1016/j.matpur.2022.05.007},
       URL = {https://doi.org/10.1016/j.matpur.2022.05.007},
}

@misc{brendle_sobolevinequalitiesmanifoldsnonnegative_2021,
  title = {Sobolev Inequalities in Manifolds with Nonnegative Curvature},
  author = {Brendle, Simon},
  year = {2021},
  note = {To appear in \emph{Comm. Pure App. Math}},
}

@misc{borghini-fogagnolo-pinamonti,
      title={The equality case in the substatic Heintze-Karcher inequality}, 
      author={Stefano Borghini and Mattia Fogagnolo and Andrea Pinamonti},
      year={2023},
      eprint={2307.04253},
      archivePrefix={arXiv},
      primaryClass={math.DG}
}

@misc{benatti2022minkowski,
      title={Minkowski Inequality on complete Riemannian manifolds with nonnegative Ricci curvature}, 
      author={Luca Benatti and Mattia Fogagnolo and Lorenzo Mazzieri},
      year={2022},
      eprint={2101.06063},
      archivePrefix={arXiv},
      primaryClass={math.DG}
}

@misc{fujitani2022functional,
      title={Some functional inequalities under lower Bakry-\'{E}mery-Ricci curvature bounds with $\varepsilon$-range}, 
      author={Yasuaki Fujitani},
      year={2022},
      eprint={2211.12310},
      archivePrefix={arXiv},
      primaryClass={math.DG}
}

@article {antonelli-fogagnolo-pozzetta,
    AUTHOR = {Antonelli, Gioacchino and Fogagnolo, Mattia and Pozzetta,
              Marco},
     TITLE = {The isoperimetric problem on {R}iemannian manifolds via {G}romov–{H}ausdorff asymptotic analysis},
   JOURNAL = {Communications in Contemporary Mathematics},
   YEAR ={2022},
    DOI = {10.1142/S0219199722500687}
}

@article {antonelli-nardulli-pozzetta,
    AUTHOR = {Antonelli, Gioacchino and Nardulli, Stefano and Pozzetta,
              Marco},
     TITLE = {The isoperimetric problem {\it via} direct method in
              noncompact metric measure spaces with lower {R}icci bounds},
   JOURNAL = {ESAIM Control Optim. Calc. Var.},
  FJOURNAL = {ESAIM. Control, Optimisation and Calculus of Variations},
    VOLUME = {28},
      YEAR = {2022},
     PAGES = {Paper No. 57, 32},
      ISSN = {1292-8119},
   MRCLASS = {49Q20 (49J45 53A35 53C23)},
  MRNUMBER = {4467099},
MRREVIEWER = {Luca Lussardi},
       DOI = {10.1051/cocv/2022052},
       URL = {https://doi.org/10.1051/cocv/2022052},
}

@article {nardulli,
    AUTHOR = {Nardulli, Stefano},
     TITLE = {Generalized existence of isoperimetric regions in non-compact
              {R}iemannian manifolds and applications to the isoperimetric
              profile},
   JOURNAL = {Asian J. Math.},
  FJOURNAL = {Asian Journal of Mathematics},
    VOLUME = {18},
      YEAR = {2014},
    NUMBER = {1},
     PAGES = {1--28},
      ISSN = {1093-6106},
   MRCLASS = {49Q20 (53A10 58C35 58E30)},
  MRNUMBER = {3215337},
MRREVIEWER = {C\'{e}sar Rosales},
       DOI = {10.4310/AJM.2014.v18.n1.a1},
       URL = {https://doi.org/10.4310/AJM.2014.v18.n1.a1},
}

@article {sakurai-convexity,
    AUTHOR = {Sakurai, Yohei},
     TITLE = {One dimensional weighted {R}icci curvature and displacement
              convexity of entropies},
   JOURNAL = {Math. Nachr.},
  FJOURNAL = {Mathematische Nachrichten},
    VOLUME = {294},
      YEAR = {2021},
    NUMBER = {10},
     PAGES = {1950--1967},
      ISSN = {0025-584X},
   MRCLASS = {49Q22 (53C21)},
  MRNUMBER = {4371277},
MRREVIEWER = {Mathias Viktor Joachim Braun},
       DOI = {10.1002/mana.201900143},
       URL = {https://doi.org/10.1002/mana.201900143},
}

@article {mari-rigoli-setti,
    AUTHOR = {Mari, Luciano and Rigoli, Marco and Setti, Alberto G.},
     TITLE = {On the {$1/H$}-flow by {$p$}-{L}aplace approximation: new
              estimates via fake distances under {R}icci lower bounds},
   JOURNAL = {Amer. J. Math.},
  FJOURNAL = {American Journal of Mathematics},
    VOLUME = {144},
      YEAR = {2022},
    NUMBER = {3},
     PAGES = {779--849},
      ISSN = {0002-9327},
   MRCLASS = {53E10 (53C21)},
  MRNUMBER = {4436145},
MRREVIEWER = {Yong Wei},
       DOI = {10.1353/ajm.2022.0016},
       URL = {https://doi.org/10.1353/ajm.2022.0016},
}

@article {wang-willmore,
    AUTHOR = {Wang, Xiaodong},
     TITLE = {Remark on an inequality for closed hypersurfaces in complete
              manifolds with nonnegative {R}icci curvature},
   JOURNAL = {Ann. Fac. Sci. Toulouse Math. (6)},
  FJOURNAL = {Annales de la Facult\'{e} des Sciences de Toulouse.
              Math\'{e}matiques. S\'{e}rie 6},
    VOLUME = {32},
      YEAR = {2023},
    NUMBER = {1},
     PAGES = {173--178},
      ISSN = {0240-2963,2258-7519},
   MRCLASS = {53C40 (53C21)},
  MRNUMBER = {4574743},
}

@article {cheeger-gromoll,
    AUTHOR = {Cheeger, Jeff and Gromoll, Detlef},
     TITLE = {The splitting theorem for manifolds of nonnegative {R}icci
              curvature},
   JOURNAL = {J. Differential Geometry},
  FJOURNAL = {Journal of Differential Geometry},
    VOLUME = {6},
      YEAR = {1971/72},
     PAGES = {119--128},
      ISSN = {0022-040X},
   MRCLASS = {53C20},
  MRNUMBER = {303460},
MRREVIEWER = {J. R. Vanstone},
       URL = {http://projecteuclid.org/euclid.jdg/1214430220},
}

@article {kleiner,
    AUTHOR = {Kleiner, Bruce},
     TITLE = {An isoperimetric comparison theorem},
   JOURNAL = {Invent. Math.},
  FJOURNAL = {Inventiones Mathematicae},
    VOLUME = {108},
      YEAR = {1992},
    NUMBER = {1},
     PAGES = {37--47},
      ISSN = {0020-9910},
   MRCLASS = {53C20 (49Q20)},
  MRNUMBER = {1156385},
MRREVIEWER = {Viktor Schroeder},
       DOI = {10.1007/BF02100598},
       URL = {https://doi.org/10.1007/BF02100598},
}

@article{balogh-kristaly,
author={Balogh, Zolt{\'a}n M.
and Krist{\'a}ly, Alexandru},
title={Sharp isoperimetric and Sobolev inequalities in spaces with nonnegative Ricci curvature},
journal={Mathematische Annalen},
year={2022},
issn={1432-1807},
doi={10.1007/s00208-022-02380-1},
url={https://doi.org/10.1007/s00208-022-02380-1},
}

@misc{wylie-yeroshkin,
      title={On the geometry of Riemannian manifolds with density}, 
      author={William Wylie and Dmytro Yeroshkin},
      year={2016},
      eprint={1602.08000},
      archivePrefix={arXiv},
      primaryClass={math.DG}
}

@article {kuwae-li,
    AUTHOR = {Kuwae, Kazuhiro and Li, Xiang-Dong},
     TITLE = {New {L}aplacian comparison theorem and its applications to
              diffusion processes on {R}iemannian manifolds},
   JOURNAL = {Bull. Lond. Math. Soc.},
  FJOURNAL = {Bulletin of the London Mathematical Society},
    VOLUME = {54},
      YEAR = {2022},
    NUMBER = {2},
     PAGES = {404--427},
      ISSN = {0024-6093},
   MRCLASS = {58J65 (58J05 58J60 60H30 60J45 60J60)},
  MRNUMBER = {4414994},
MRREVIEWER = {Maria Gordina},
       DOI = {10.1112/blms.12568},
       URL = {https://doi.org/10.1112/blms.12568},
}

@book {ohta-book,
    AUTHOR = {Ohta, Shin-ichi},
     TITLE = {Comparison {F}insler geometry},
    SERIES = {Springer Monographs in Mathematics},
 PUBLISHER = {Springer, Cham},
      YEAR = {2021},
     PAGES = {xxii+316},
   MRCLASS = {53C60 (53-02 53B40 53C23 58J35)},
  MRNUMBER = {4386101},
       DOI = {10.1007/978-3-030-80650-7},
       URL = {https://doi.org/10.1007/978-3-030-80650-7},
}

@article {kuwae-sakurai,
    AUTHOR = {Kuwae, Kazuhiro and Sakurai, Yohei},
     TITLE = {Comparison geometry of manifolds with boundary under lower
              {$N$}-weighted {R}icci curvature bounds with
              {$\varepsilon$}-range},
   JOURNAL = {J. Math. Soc. Japan},
  FJOURNAL = {Journal of the Mathematical Society of Japan},
    VOLUME = {75},
      YEAR = {2023},
    NUMBER = {1},
     PAGES = {151--172},
      ISSN = {0025-5645},
   MRCLASS = {53C21 (53C20)},
  MRNUMBER = {4539013},
       DOI = {10.2969/jmsj/87278727},
       URL = {https://doi.org/10.2969/jmsj/87278727},
}

@article {lu-minguzzi-ohta,
    AUTHOR = {Lu, Yufeng and Minguzzi, Ettore and Ohta, Shin-ichi},
     TITLE = {Comparison theorems on weighted {F}insler manifolds and
              spacetimes with {$\epsilon$}-range},
   JOURNAL = {Anal. Geom. Metr. Spaces},
  FJOURNAL = {Analysis and Geometry in Metric Spaces},
    VOLUME = {10},
      YEAR = {2022},
    NUMBER = {1},
     PAGES = {1--30},
   MRCLASS = {53C21 (53C50 53C60)},
  MRNUMBER = {4388774},
MRREVIEWER = {Bing-Ye Wu},
       DOI = {10.1515/agms-2020-0131},
       URL = {https://doi.org/10.1515/agms-2020-0131},
}

@article {cavalletti-manini1,
    AUTHOR = {Cavalletti, Fabio and Manini, Davide},
     TITLE = {Isoperimetric inequality in noncompact {${MCP}$} spaces},
   JOURNAL = {Proc. Amer. Math. Soc.},
  FJOURNAL = {Proceedings of the American Mathematical Society},
    VOLUME = {150},
      YEAR = {2022},
    NUMBER = {8},
     PAGES = {3537--3548},
      ISSN = {0002-9939},
   MRCLASS = {49Q20 (49Q22)},
  MRNUMBER = {4439475},
MRREVIEWER = {Alp\'{a}r R. M\'{e}sz\'{a}ros},
       DOI = {10.1090/proc/15945},
       URL = {https://doi.org/10.1090/proc/15945},
}

@misc{antonelli-pasqualetto-pozzetta-semola,
      title={Asymptotic isoperimetry on non collapsed spaces with lower Ricci bounds}, 
      author={Gioacchino Antonelli and Enrico Pasqualetto and Marco Pozzetta and Daniele Semola},
      year={2022},
      eprint={2208.03739},
      archivePrefix={arXiv},
      primaryClass={math.DG},
      note={Accepted for publication in Math. Ann.},
}

@misc{pozzetta2023isoperimetry,
      title={Isoperimetry on manifolds with Ricci bounded below: overview of recent results and methods}, 
      author={Marco Pozzetta},
      year={2023},
      eprint={2303.11925},
      archivePrefix={arXiv},
      primaryClass={math.DG}
}

@misc{cavalletti-manini2,
      title={Rigidities of Isoperimetric inequality under nonnegative Ricci curvature}, 
      author={Fabio Cavalletti and Davide Manini},
      year={2022},
      eprint={2207.03423},
      archivePrefix={arXiv},
      primaryClass={math.MG},
}

@misc{johne,
      title={Sobolev inequalities on manifolds with nonnegative Bakry-\'Emery Ricci curvature}, 
      author={Florian Johne},
      year={2021},
      eprint={2103.08496},
      archivePrefix={arXiv},
      primaryClass={math.DG},
}

@article{agostiniani_sharpgeometricinequalitiesclosed_2020,
  title = {Sharp Geometric Inequalities for Closed Hypersurfaces in Manifolds with Nonnegative {{Ricci}} Curvature},
  author = {Agostiniani, Virginia and Fogagnolo, Mattia and Mazzieri, Lorenzo},
  year = {2020},
  journal = {Inventiones mathematicae},
  issn = {1432-1297},
  doi = {10.1007/s00222-020-00985-4},
}

@article {bray-morgan,
    AUTHOR = {Bray, Hubert and Morgan, Frank},
     TITLE = {An isoperimetric comparison theorem for {S}chwarzschild space
              and other manifolds},
   JOURNAL = {Proc. Amer. Math. Soc.},
  FJOURNAL = {Proceedings of the American Mathematical Society},
    VOLUME = {130},
      YEAR = {2002},
    NUMBER = {5},
     PAGES = {1467--1472},
      ISSN = {0002-9939},
   MRCLASS = {53C42},
  MRNUMBER = {1879971},
MRREVIEWER = {Constantin Vernicos},
       DOI = {10.1090/S0002-9939-01-06186-X},
       URL = {https://doi.org/10.1090/S0002-9939-01-06186-X},
}

@article {brendle-alexandrov,
    AUTHOR = {Brendle, Simon},
     TITLE = {Constant mean curvature surfaces in warped product manifolds},
   JOURNAL = {Publ. Math. Inst. Hautes \'{E}tudes Sci.},
  FJOURNAL = {Publications Math\'{e}matiques. Institut de Hautes \'{E}tudes
              Scientifiques},
    VOLUME = {117},
      YEAR = {2013},
     PAGES = {247--269},
      ISSN = {0073-8301},
   MRCLASS = {53A10 (53C45)},
  MRNUMBER = {3090261},
MRREVIEWER = {Andrew Bucki},
       DOI = {10.1007/s10240-012-0047-5},
       URL = {https://doi.org/10.1007/s10240-012-0047-5},
}

@article {SternbergZiemerWilliams,
    AUTHOR = {Sternberg, Peter and Ziemer, William P. and Williams, Graham},
     TITLE = {{$C^{1,1}$}-regularity of constrained area minimizing
              hypersurfaces},
   JOURNAL = {J. Differential Equations},
  FJOURNAL = {Journal of Differential Equations},
    VOLUME = {94},
      YEAR = {1991},
    NUMBER = {1},
     PAGES = {83--94},
      ISSN = {0022-0396},
   MRCLASS = {49Q25 (58E12)},
  MRNUMBER = {1133542},
MRREVIEWER = {Carlo Bardaro},
       DOI = {10.1016/0022-0396(91)90104-H},
       URL = {https://doi.org/10.1016/0022-0396(91)90104-H},
}

@book {Maggi,
    AUTHOR = {Maggi, Francesco},
     TITLE = {Sets of finite perimeter and geometric variational problems},
    SERIES = {Cambridge Studies in Advanced Mathematics},
    VOLUME = {135},
      NOTE = {An introduction to geometric measure theory},
 PUBLISHER = {Cambridge University Press, Cambridge},
      YEAR = {2012},
     PAGES = {xx+454},
      ISBN = {978-1-107-02103-7},
   MRCLASS = {49-01 (26B20 28-02 49-02 49Q05 49Q20)},
  MRNUMBER = {2976521},
MRREVIEWER = {Giovanni Alberti},
       DOI = {10.1017/CBO9781139108133},
       URL = {https://doi.org/10.1017/CBO9781139108133},
}

@article {Morgan,
    AUTHOR = {Morgan, Frank},
     TITLE = {Regularity of isoperimetric hypersurfaces in {R}iemannian
              manifolds},
   JOURNAL = {Trans. Amer. Math. Soc.},
  FJOURNAL = {Transactions of the American Mathematical Society},
    VOLUME = {355},
      YEAR = {2003},
    NUMBER = {12},
     PAGES = {5041--5052},
      ISSN = {0002-9947},
   MRCLASS = {49Q20 (49N60 53C42)},
  MRNUMBER = {1997594},
MRREVIEWER = {Monica Torres},
       DOI = {10.1090/S0002-9947-03-03061-7},
       URL = {https://doi.org/10.1090/S0002-9947-03-03061-7},
}

@article {FocardiGeraciSpadaro,
    AUTHOR = {Focardi, M. and Geraci, F. and Spadaro, E.},
     TITLE = {The classical obstacle problem for nonlinear variational
              energies},
   JOURNAL = {Nonlinear Anal.},
  FJOURNAL = {Nonlinear Analysis. Theory, Methods \& Applications. An
              International Multidisciplinary Journal},
    VOLUME = {154},
      YEAR = {2017},
     PAGES = {71--87},
      ISSN = {0362-546X},
   MRCLASS = {35R35 (35J87 49N60)},
  MRNUMBER = {3614645},
       DOI = {10.1016/j.na.2016.10.020},
       URL = {https://doi.org/10.1016/j.na.2016.10.020},
}

@article {MondinoSpadaro,
    AUTHOR = {Mondino, Andrea and Spadaro, Emanuele},
     TITLE = {On an isoperimetric-isodiametric inequality},
   JOURNAL = {Anal. PDE},
  FJOURNAL = {Analysis \& PDE},
    VOLUME = {10},
      YEAR = {2017},
    NUMBER = {1},
     PAGES = {95--126},
      ISSN = {2157-5045},
   MRCLASS = {49J40 (35B65 35R01 49Q10 49Q20)},
  MRNUMBER = {3611014},
MRREVIEWER = {Faustino Maestre},
       DOI = {10.2140/apde.2017.10.95},
       URL = {https://doi.org/10.2140/apde.2017.10.95},
}

@article {Pigola_Veronelli,
    AUTHOR = {Pigola, Stefano and Veronelli, Giona},
     TITLE = {The smooth {R}iemannian extension problem},
   JOURNAL = {Ann. Sc. Norm. Super. Pisa Cl. Sci. (5)},
  FJOURNAL = {Annali della Scuola Normale Superiore di Pisa. Classe di
              Scienze. Serie V},
    VOLUME = {20},
      YEAR = {2020},
    NUMBER = {4},
     PAGES = {1507--1551},
      ISSN = {0391-173X},
   MRCLASS = {53C23 (53C20 53C40)},
  MRNUMBER = {4201188},
MRREVIEWER = {Hans-Bert Rademacher},
}

@article{huisken_inverse_2001,
  title = {The Inverse Mean Curvature Flow and the {{Riemannian Penrose}} Inequality},
  author = {Huisken, Gerhard and Ilmanen, Tom},
  year = {2001},
  journal = {Journal of Differential Geometry},
  volume = {59},
  number = {3},
  pages = {353--437},
  issn = {0022-040X},
  url = {https://mathscinet.ams.org/mathscinet-getitem?mr=1916951},
  mrnumber = {1916951},
}

@article {Ilmanen-mcf,
    AUTHOR = {Ilmanen, Tom},
     TITLE = {Generalized flow of sets by mean curvature on a manifold},
   JOURNAL = {Indiana Univ. Math. J.},
  FJOURNAL = {Indiana University Mathematics Journal},
    VOLUME = {41},
      YEAR = {1992},
    NUMBER = {3},
     PAGES = {671--705},
      ISSN = {0022-2518},
   MRCLASS = {58E15 (53A07)},
  MRNUMBER = {1189906},
MRREVIEWER = {Anders Linn\'{e}r},
       DOI = {10.1512/iumj.1992.41.41036},
       URL = {https://doi.org/10.1512/iumj.1992.41.41036},
}

@book {Petersen,
    AUTHOR = {Petersen, Peter},
     TITLE = {Riemannian geometry},
    SERIES = {Graduate Texts in Mathematics},
    VOLUME = {171},
   EDITION = {Third},
 PUBLISHER = {Springer, Cham},
      YEAR = {2016},
     PAGES = {xviii+499},
%      ISBN = {978-3-319-26652-7; 978-3-319-26654-1},
   MRCLASS = {53-01 (53C20 53C21 53C23)},
  MRNUMBER = {3469435},
       DOI = {10.1007/978-3-319-26654-1},
       URL = {https://doi.org/10.1007/978-3-319-26654-1},
}

@article {Fogagnolo_Mazzieri-minimising,
    AUTHOR = {Fogagnolo, Mattia and Mazzieri, Lorenzo},
     TITLE = {Minimising hulls, {$p$}-capacity and isoperimetric inequality
              on complete {R}iemannian manifolds},
   JOURNAL = {J. Funct. Anal.},
  FJOURNAL = {Journal of Functional Analysis},
    VOLUME = {283},
      YEAR = {2022},
    NUMBER = {9},
     PAGES = {Paper No. 109638, 49},
      ISSN = {0022-1236},
   MRCLASS = {49Q10 (31C15 49J40 53C21 53E10)},
  MRNUMBER = {4459004},
       DOI = {10.1016/j.jfa.2022.109638},
       URL = {https://doi.org/10.1016/j.jfa.2022.109638},
}

@article {Kasue,
    AUTHOR = {Kasue, Atsushi},
     TITLE = {Ricci curvature, geodesics and some geometric properties of
              {R}iemannian manifolds with boundary},
   JOURNAL = {J. Math. Soc. Japan},
  FJOURNAL = {Journal of the Mathematical Society of Japan},
    VOLUME = {35},
      YEAR = {1983},
    NUMBER = {1},
     PAGES = {117--131},
      ISSN = {0025-5645},
   MRCLASS = {53C20 (53C22)},
  MRNUMBER = {679079},
MRREVIEWER = {Yu. Burago},
       DOI = {10.2969/jmsj/03510117},
       URL = {https://doi.org/10.2969/jmsj/03510117},
}

@article {Li_Xia_17,
    AUTHOR = {Li, Junfang and Xia, Chao},
     TITLE = {An integral formula for affine connections},
   JOURNAL = {J. Geom. Anal.},
  FJOURNAL = {Journal of Geometric Analysis},
    VOLUME = {27},
      YEAR = {2017},
    NUMBER = {3},
     PAGES = {2539--2556},
      ISSN = {1050-6926},
   MRCLASS = {53B05 (53B21 53C21)},
  MRNUMBER = {3667440},
MRREVIEWER = {L\'{e}onard Todjihounde},
       DOI = {10.1007/s12220-017-9771-x},
       URL = {https://doi.org/10.1007/s12220-017-9771-x},
}

@article {Li_Xia_19,
    AUTHOR = {Li, Junfang and Xia, Chao},
     TITLE = {An integral formula and its applications on sub-static
              manifolds},
   JOURNAL = {J. Differential Geom.},
  FJOURNAL = {Journal of Differential Geometry},
    VOLUME = {113},
      YEAR = {2019},
    NUMBER = {3},
     PAGES = {493--518},
      ISSN = {0022-040X},
   MRCLASS = {53C40 (53C21)},
  MRNUMBER = {4031740},
MRREVIEWER = {Theodoros Vlachos},
       DOI = {10.4310/jdg/1573786972},
       URL = {https://doi.org/10.4310/jdg/1573786972},
}

@article {Wang_Wang_Zhang,
    AUTHOR = {Wang, Mu-Tao and Wang, Ye-Kai and Zhang, Xiangwen},
     TITLE = {Minkowski formulae and {A}lexandrov theorems in spacetime},
   JOURNAL = {J. Differential Geom.},
  FJOURNAL = {Journal of Differential Geometry},
    VOLUME = {105},
      YEAR = {2017},
    NUMBER = {2},
     PAGES = {249--290},
      ISSN = {0022-040X},
   MRCLASS = {53C42 (53C50 83C20)},
  MRNUMBER = {3606730},
MRREVIEWER = {Wei-huan Chen},
       DOI = {10.4310/jdg/1486522815},
       URL = {https://doi.org/10.4310/jdg/1486522815},
}

@book {Evans,
    AUTHOR = {Evans, Lawrence C.},
     TITLE = {Partial differential equations},
    SERIES = {Graduate Studies in Mathematics},
    VOLUME = {19},
   EDITION = {Second},
 PUBLISHER = {American Mathematical Society, Providence, RI},
      YEAR = {2010},
     PAGES = {xxii+749},
      ISBN = {978-0-8218-4974-3},
   MRCLASS = {35-01},
  MRNUMBER = {2597943},
MRREVIEWER = {Diego M. Maldonado},
       DOI = {10.1090/gsm/019},
       URL = {https://doi.org/10.1090/gsm/019},
}

@book {doCarmo,
    AUTHOR = {do Carmo, Manfredo Perdig\~{a}o},
     TITLE = {Riemannian geometry},
    SERIES = {Mathematics: Theory \& Applications},
      NOTE = {Translated from the second Portuguese edition by Francis
              Flaherty},
 PUBLISHER = {Birkh\"{a}user Boston, Inc., Boston, MA},
      YEAR = {1992},
     PAGES = {xiv+300},
      ISBN = {0-8176-3490-8},
   MRCLASS = {53-01},
  MRNUMBER = {1138207},
MRREVIEWER = {Bang-yen Chen},
       DOI = {10.1007/978-1-4757-2201-7},
       URL = {https://doi.org/10.1007/978-1-4757-2201-7},
}

@article{Wylie,
  title={A warped product version of the {C}heeger--{G}romoll splitting theorem},
  author={Wylie, William},
  journal={Transactions of the American Mathematical Society},
  volume={369},
  number={9},
  pages={6661--6681},
  year={2017},
}

@article {Chrusciel_Simon,
    AUTHOR = {Chru\'{s}ciel, Piotr T. and Simon, Walter},
     TITLE = {Towards the classification of static vacuum spacetimes with
              negative cosmological constant},
   JOURNAL = {J. Math. Phys.},
  FJOURNAL = {Journal of Mathematical Physics},
    VOLUME = {42},
      YEAR = {2001},
    NUMBER = {4},
     PAGES = {1779--1817},
      ISSN = {0022-2488},
   MRCLASS = {83C20 (83C57)},
  MRNUMBER = {1820431},
MRREVIEWER = {Lars \AA ke Andersson},
       DOI = {10.1063/1.1340869},
       URL = {https://doi.org/10.1063/1.1340869},
}

@article {Mantegazza_Mennucci,
    AUTHOR = {Mantegazza, Carlo and Mennucci, Andrea Carlo},
     TITLE = {Hamilton-{J}acobi equations and distance functions on
              {R}iemannian manifolds},
   JOURNAL = {Appl. Math. Optim.},
  FJOURNAL = {Applied Mathematics and Optimization},
    VOLUME = {47},
      YEAR = {2003},
    NUMBER = {1},
     PAGES = {1--25},
      ISSN = {0095-4616},
   MRCLASS = {49L20 (49Q15 53C22)},
  MRNUMBER = {1941909},
MRREVIEWER = {Pierre Cardaliaguet},
       DOI = {10.1007/s00245-002-0736-4},
       URL = {https://doi.org/10.1007/s00245-002-0736-4},
}

@article{Gal_Sch_Wit_Woo,
  title={Topological censorship and higher genus black holes},
  author={Galloway, G. J. and Schleich, K. and Witt, D. M. and Woolgar, E.},
  journal={Physical Review D},
  volume={60},
  number={10},
  pages={104039},
  year={1999},
  publisher={APS}
}

@article{kennard_wylie_yeroshkin,
  title={The weighted connection and sectional curvature for manifolds with density},
  author={Kennard, Lee and Wylie, William and Yeroshkin, Dmytro},
  journal={The Journal of Geometric Analysis},
  volume={29},
  pages={957--1001},
  year={2019},
  publisher={Springer}
}

@article{Fang_Yuan,
  title={Brown--York mass and positive scalar curvature II: Besse’s conjecture and related problems},
  author={Fang, Yi and Yuan, Wei},
  journal={Annals of Global Analysis and Geometry},
  volume={56},
  pages={1--15},
  year={2019},
  publisher={Springer}
}

@article{He,
  title={Critical metrics of the volume functional on three-dimensional manifolds},
  author={He, Huiya},
  journal={arXiv preprint arXiv:2101.05621},
  year={2021}
}

@article {Zeng,
    AUTHOR = {Zeng, Fanqi},
     TITLE = {Some almost-{S}chur type inequalities and applications on
              sub-static manifolds},
   JOURNAL = {Electron. Res. Arch.},
  FJOURNAL = {Electronic Research Archive},
    VOLUME = {30},
      YEAR = {2022},
    NUMBER = {8},
     PAGES = {2860--2870},
      ISSN = {2688-1594},
   MRCLASS = {53C21 (53C42)},
  MRNUMBER = {4432005},
MRREVIEWER = {Xiaodong\ Wang},
       DOI = {10.3934/era.2022145},
       URL = {https://doi.org/10.3934/era.2022145},
}

@article {Cheng,
    AUTHOR = {Cheng, Xu},
     TITLE = {An almost-{S}chur type lemma for symmetric {$(2,0)$} tensors
              and applications},
   JOURNAL = {Pacific J. Math.},
  FJOURNAL = {Pacific Journal of Mathematics},
    VOLUME = {267},
      YEAR = {2014},
    NUMBER = {2},
     PAGES = {325--340},
      ISSN = {0030-8730,1945-5844},
   MRCLASS = {53C21},
  MRNUMBER = {3207587},
MRREVIEWER = {Tan\ Zhang},
       DOI = {10.2140/pjm.2014.267.325},
       URL = {https://doi.org/10.2140/pjm.2014.267.325},
}
\end{document}